\documentclass[11pt]{amsart}

\usepackage[T1]{fontenc}
\usepackage{mathpazo}
\usepackage[none]{hyphenat}
\usepackage{mathpple}
\usepackage{pgf,tikz,pgfplots}
\usepackage{graphicx}
\usepackage{float}
\usepackage{mathrsfs}
\usetikzlibrary{arrows}

\advance\textheight by \topskip
\usepackage{amsmath,amsfonts,amssymb}
\usepackage{amsrefs}
\usepackage{amsthm}
\usepackage[all]{xy}

\usepackage{hyperref,graphicx}

\numberwithin{equation}{section}

\newcommand{\op}{\operatorname}
\newcommand{\C}{\mathbb{C}}

\newcommand{\R}{\mathbb{R}}
\newcommand{\Q}{\mathbb{Q}}
\newcommand{\Z}{\mathbb{Z}}
\renewcommand{\H}{\mathbf{H}}

\newcommand{\Etau}{{\text{E}_\tau}}
\newcommand{\E}{{\mathcal E}}

\newcommand{\im}{\op{im}}

\newcommand{\h}{\mathbf{h}}

\providecommand{\abs}[1]{\left\lvert#1\right\rvert}

\newcommand{\abracket}[1]{\left\langle#1\right\rangle}
\newcommand{\bbracket}[1]{\left[#1\right]}
\newcommand{\fbracket}[1]{\left\{#1\right\}}
\newcommand{\bracket}[1]{\left(#1\right)}
\newcommand{\ket}[1]{|#1\rangle}

\newcommand{\ora}[1]{\overrightarrow#1}

\newcommand{\mc}{\mathcal}
\newcommand{\cinfty}{C^{\infty}}
\newcommand{\pa}{\partial}

\renewcommand{\dbar}{\bar\pa}
\newcommand{\OO}{{\mathcal O}}

\newcommand{\BV}{Batalin-Vilkovisky }

\newcommand{\Ol}{\mathcal O_{loc}}

\newcommand{\iso}{\cong}

\newcommand{\B}{\mathcal{B}}
\newcommand{\V}{\mathcal{V}}

\renewcommand{\Im}{\op{Im}}

\DeclareMathOperator{\End}{End}

\DeclareMathOperator{\Sym}{Sym}
\DeclareMathOperator{\Hom}{Hom}

\DeclareMathOperator{\Tr}{Tr}

\DeclareMathOperator{\Id}{Id}

\DeclareMathOperator{\Dens}{Dens}

\DeclareMathOperator{\cw}{cw}

\newcommand{\A}{\mathcal A}

\renewcommand{\L}{\mathcal L}

\theoremstyle{plain}
\newtheorem{thm}{Theorem}[section]
\newtheorem{thm-defn}{Theorem/Definition}[section]
\newtheorem{lem}[thm]{Lemma}
\newtheorem{lem-defn}[thm]{Lemma/Definition}

\newtheorem{prop}[thm]{Proposition}

\newtheorem{cor}[thm]{Corollary}

\theoremstyle{definition}
\newtheorem{defn}[thm]{Definition}
\newtheorem{eg}[thm]{Example}

\theoremstyle{remark}
\newtheorem{rmk}[thm]{Remark}

\allowdisplaybreaks[4]  
\begin{document}

  \title{Vertex algebras and quantum master equation}
  \author{Si Li}
    
  \address{
S. Li: Department of Mathematical Sciences and Yau Mathematical Sciences Center, Tsinghua University, Beijing, China; 
}
\email{sili@mail.tsinghua.edu.cn}
  
\date{}
  \maketitle


\begin{abstract}
We contribute to study the two dimensional chiral conformal field theories that arise from chiral deformations of free theories such as the $\beta\gamma$-system, $bc$-system and chiral bosons. Our method is based on the  effective \BV quantization theory. We establish an exact correspondence between renormalized quantum master equations for chiral deformations and Maurer-Cartan equations for the Lie algebra of modes of chiral vertex algebras. The generating functions on the torus are proven to have modular property with mild holomorphic anomaly. As an application, we construct an exact solution of quantum B-model (BCOV theory) in complex one dimension that solves the full higher genus mirror symmetry conjecture on elliptic curves. 
\end{abstract}

  \tableofcontents

\section{Introduction}
Quantum field theory provides a rich source of mathematical thoughts. One important feature of quantum field theory that lies secretly behind many of its surprising mathematical predictions is the role of infinite dimensionality. A famous example is the mysterious mirror symmetry conjecture between symplectic and complex geometries, which can be viewed as a version of infinite dimensional Fourier transform. Typically, many quantum problems are formulated in terms of ``path integrals", which require measures that are mostly not yet known to mathematicians. Nevertheless, asymptotic analysis can always be performed with the help of the celebrated idea of renormalization. 

Despite the great success of renormalization theory in physics applications, its use in mathematics is relatively limited but extremely powerful when it does apply. One such example is Kontsevich's solution \cite{Kontsevich-DQ} to the deformation quantization problem on arbitrary Poisson manifolds. Kontsevich's explicit formula of star product is obtained via graph integrals on a compactification of configuration space on the disk, which can be viewed as a geometric renormalization of the perturbative expansion of Poisson sigma model (see  \cite{CF}). Another recent example is Costello's homotopic theory \cite{Kevin-book}  of effective renormalizations in the \BV  formalism. This leads to a systematic  construction of factorization algebras via quantum field theories \cite{Kevin-Owen}. For example, a natural geometric interpretation of the Witten genus is obtained in such a way \cite{Kevin-genus}.

To facilitate geometric applications of effective renormalization methods, it would be crucial to connect renormalized quantities to geometric objects. We  will be mainly interested in quantum field theory with gauge symmetries. The most general framework of quantizing  gauge  theories is the \BV formalism \cite{BV}, where the quantum consistency of gauge transformations is described by the so-called \emph{quantum master equation}. There have developed several mathematical approaches to incorporate \BV formalism with renormalizations since their birth. The central quantity of all approaches lies in the renormalized quantum master equation. In this paper, we will mainly discuss the perturbative formalism in \cite{Kevin-book}, which has developed a convenient framework that is rooted in the homotopic culture of derived algebraic geometry. A brief introduction to the philosophy of this approach is discussed in Section \ref{section-BV}. 

The simplest nontrivial cases are quantum mechanical models, which can be viewed as quantum field theories in one dimension. An an example, consider the classical mechanical system with phase space a symplectic manifold $X$. Fedosov \cite{Fedosov} gave a beautiful construction of deformation quantization on $X$, which amounts to constructing a flat connection (called an abelian connection) on the Weyl bundle over $X$. It turns out that such geometric data has a precise quantum field theory interpretation in terms of effective renormalization \cite{GG, LLG,GLX}. It is shown in \cite{LLG} that every Fedosov deformation quantization arises from a perturbative quantization of topological mechanics. More precisely,  every solution to the renormalized quantum master equation produces a Fedosov's abelian connection.  The algebraic nature of this correspondence is reviewed in Section \ref{section-example}. Such a correspondence leads to a  geometric approach to the algebraic index theorem \cite{Fedosov-index,Nest-Tsygan},  where the index formula follows from the homotopic renormalization group flow together with an exact semi-classical approximation via BV integration \cite{LLG,GLX}.

In this paper, we study systematically the renormalized quantum master equation in two dimensions. We will focus on quantum theories obtained by chiral deformations of free conformal field theory(CFT)'s (see Section \ref{section-2d} for our precise set-up) on flat spaces. One important feature of such two dimensional chiral theories is that they are free of ultra-violet divergence (see Theorem \ref{thm-defn}). This greatly simplifies the analysis of quantization since singular counter-terms are not required.  However, the renormalized quantum master equation requires quantum deformations by chiral local functionals. Such quantum deformations could in principle be very complicated. 

One of our main results in this paper (Theorem \ref{main-thm}) is an exact description of the renormalized quantum master equations for chiral deformations of free $\beta\gamma-bc$ systems. Briefly speaking, Theorem \ref{main-thm} states that solutions to the renormalized quantum master equations (QME) correspond to solutions to the Maurer-Cartan (MC) equation of the modes Lie algebra of the chiral vertex algebra (see Definition \ref{defn-modes}) associated to the free theory

$$
\boxed{\text{renormalized QME}}\Longleftrightarrow \boxed{\text{MC equation for chiral modes Lie algebra}}
$$
Our main Theorem \ref{main-thm} is stated for chiral deformations of $\beta\gamma-bc$ systems, but it actually works in parallel  for chiral bosons as well. We do not present such a general discussion in this paper, but illustrate how to modify by a concrete example in Section \ref{section-B} (see Theorem \ref{thm-chiral-boson}).

Geometrically, the MC element leads to a cohomological operator $\mathscr Q$ (degree 1 and squares zero) acting on the vertex algebra $\V$, which can be identified with the BRST operator. The cohomology
$$
H(\V, \mathscr Q)
$$
gives a new chiral vertex algebra, which should be identified with local operators associated to the chiral deformed theory.  In particular, this chiral deformation can be used to realize the usual BRST construction of vertex algebras. 

Furthermore, we prove a general result on the modularity property of the generating functions on elliptic curves and their holomorphic anomaly (Theorem \ref{thm-modularity}). This work is partially motivated from understanding Dijkgraaf's description \cite{Dijkgraaf-chiral} of chiral deformation of conformal field theories. The Maurer-Cartan equation serves as an integrability condition for chiral vertex operators, which is often related to integrable hierarchies in concrete cases. We discuss one such example in Section \ref{section-B}.  See \cite{L-review,HLTY} for further discussions from the perspective of integrable hierarchy.

The above correspondence can be viewed as the two dimensional vertex algebra analogue of the one dimensional result in \cite{LLG}. In fact, one main motivation of the current work is to explore the analogue of index theorem for chiral vertex operators in terms of the method of semi-classical method proceeded in \cite{LLG, GLX}.  A family version of the above correspondence should lead to the analogue of Fedosov's connection on vertex algebra bundles over geometric spaces.

As an application in Section \ref{section-B},  we construct  an exact solution of quantum B-model on elliptic curves, which leads to the solution of the corresponding higher genus mirror symmetry conjecture. Mirror symmetry is a famous duality between symplectic (A-model) and complex (B-model) geometries that arises from superconformal field theories. It has been a long-standing challenge for mathematicians to construct  quantum B-model on compact Calabi-Yau manifolds. There is a categorical approach \cite{Partition, TCFT, KS} to the quantum B-model partition function associated to a Calabi-Yau category based on a classification of  two-dimensional
topological field theories. Unfortunately, it is extremely difficult to perform this categorical computation (recently a first non-trivial  categorical computation is carried out in \cite{CT} for one-point functions on the elliptic curve ).  Another approach is through quantum field theory. In \cite{Si-Kevin}, we construct a gauge  theory of polyvector fields on Calabi-Yau manifolds (called {BCOV theory}) as a generalization of the Kodaira-Spencer gauge theory \cite{BCOV}. It is proposed in \cite{Si-Kevin} (as a generalization of \cite{BCOV}) that the \BV quantization of BCOV theory leads to quantum B-model that is mirror to the A-model Gromov-Witten theory of counting higher genus curves. Our construction in Section \ref{section-B} gives a concrete realization of this program. This leads to the first mathematically fully established example of higher genus mirror symmetry on compact Calabi-Yau manifolds. 

Our result in Section \ref{section-B} also leads to an interesting result in physics. Quantum BCOV theory can be viewed as a complete description of topological B-twisted closed string field theory in the sense of Zwiebach \cite{Zwiebach}.  Zwiebach's closed string field theory describes the dynamics of closed strings in term of the so-called string vertices. Despite the beauty of this construction, string vertices are very difficult to compute and few concrete examples are known. Our exact solution in Section \ref{section-B} can be viewed as giving an explicit realization of Zwiebach's string vertices for B-twisted topological string on elliptic curves.  

\noindent \textbf{Acknowledgement:} The author would like to thank Kevin Costello, Cumrun Vafa, Andrei Losev, Robert Dijkgraaf, 	Jae-Suk Park, Owen Gwilliam, Ryan Grady, Qin Li, and Brian Williams for discussions on quantum field theories, and thank Jie Zhou for discussions on modular forms. Part of the work was done during visiting Perimeter Institute for theoretical physics and IBS center for geometry and physics. The author thanks for their hospitality and provision of excellent working enviroments. Special thank goes to Yunchen Li and Xinyi Li, whose birth and growth have inspired and reformulated many aspects of the presentation of the current work. This work is supported by the National Key Research and Development Program of China  (NO. 2020YFA0713000),  Beijing Municipal Natural Science Foundation (NO. Z180003), and National Natural Science Foundation of China (NO. 11801300).

\noindent \hypertarget{Conventions}{\textbf{Conventions}}
\begin{itemize}
\item Let $V$ be a $\Z$-graded  $k$-vector space. We use $V_m$ to denote its degree $m$ component. Given $a\in V_m$, we let $\bar a=m$ be its degree. 
\begin{itemize}
\item $V[n]$ denotes  the degree shifting of $V$ such that $V[n]_m=V_{n+m}$. 

\item $V^*$ denotes its dual such that $V^*_m=\Hom_k(V_{-m}, k)$. Our base field $k$ will  mainly be $\R$ or $\C$. 
\item $\Sym^m(V)$ and $\wedge^m(V)$ denote the graded symmetric product and graded skew-symmetric product respectively. We also denote
$$
  \Sym(V):=\bigoplus_{m\geq 0}\Sym^m(V), \quad \widehat{\Sym}(V):=\prod_{m\geq 0} \Sym^m(V). 
$$
Given $I=\sum\limits_{m\geq 0}I_m\in \widehat{\Sym}(V^*)$, and $a_1, \cdots, a_m\in V$, we denote its $m$-order Taylor coefficient 
$$
   {\pa\over \pa a_1}\cdots{\pa\over \pa a_m} I(0):= I_m(a_1, \cdots, a_m)
$$
where we have viewed $I_m$ as a multi-linear map $V^{\otimes m}\to k$.  

\item Given $P\in \Sym^2(V)$, it defines a ``second order operator" $\pa_P$ on $\Sym(V^*)$ or $\widehat{\Sym}(V^*)$ by 
$$
  \pa_P:  \Sym^m(V^*)\to \Sym^{m-2}(V^*), \quad I \to \pa_P I,
$$
where for any $a_1, \cdots, a_{m-2}\in V$, 
$$
  \pa_P I(a_1, \cdots a_{m-2}):= I(P, a_1, \cdots a_{m-2}). 
$$
\item $V[z]$, $V[[z]]$ and $V((z))$ denote polynomial series, formal power series and Laurent series respectively in a variable $z$ valued in $V$.
\end{itemize}

\item Let $V$ be a $\Z$-graded  $k$-vector space as above. Let $\{e_i\}$ be a basis of $V$ such that 
$$
V= \bigoplus_{i} \Z e_i. 
$$ 
We will identify
$$
\C[e_i]=\Sym(V), \quad \C[[e_i]]= \widehat{\Sym}(V). 
$$
This specifies the meaning of the completion $\C[[e_i]]$ used in this paper, i.e., the formal completion with respect to the polynomial degree. 

\item Let $\A$ be a graded  algebra. $[-,-]$ always means the graded commutator, i.e., for elements $a,b$ with specific degrees, 
$$
   [a,b]:= a\cdot b-(-1)^{\bar a \bar b}b \cdot a. 
$$
We always assume Koszul sign rule in dealing with graded objects.  

\item $\otimes$ without subscript means tensoring over the real numbers $\R$.  
\item Given a manifold $X$, we denote the space of real smooth forms by
$$
\Omega^\bullet(X)=\bigoplus_{k}\Omega^k(X)
$$ 
where $\Omega^k(X)$ is the subspace of $k$-forms. If furthermore X is a complex manifold, we denote the space of complex smooth forms  by
$$
\Omega^{\bullet, \bullet}(X)=\bigoplus_{p,q}\Omega^{p,q}(X)= \Omega^\bullet(X)\otimes \C
$$ 
where  $\Omega^{p,q}(X)$ is the subspace of $(p,q)$-forms. 
\item $\Dens(X)$ denotes the density line bundle on a manifold $X$. When $X$ is oriented, we naturally identify $\Dens(X)$ with top differential forms on $X$.  
\item Let E be a vector bundle on a manifold $X$.  $\E=\Gamma_c(X, E)$ denotes the space of smooth sections with compact supports, and $\E^\prime=D(X, E)$ denotes the distributional sections. If $E^*$ is the dual bundle of $E$, then we have a natural pairing 
$$
   \E^\prime \otimes \Gamma_c(X, E^*\otimes \Dens(X)) \to \R. 
$$

\item We will also work with infinite dimensional functional spaces that carry natural topologies. The above notions for $V$ will be generalized as follows. We refer the reader to \cite{treves} or Appendix 2 of \cite{Kevin-book} for further details.

\begin{itemize}
\item All topological vector spaces we consider will be {\it nuclear} (or at least {\it locally convex Hausdorff}) and we use $\hat\otimes$ to denote the {\it completed projective tensor product}. For example, 
\begin{itemize}
\item Given two manifolds $M, N$, we have a canonical isomorphism
\[
C^\infty (M) \hat\otimes C^\infty (N) = C^\infty (M \times N);
\]
\item Similarly, if $\mathcal{D} (M)$ denotes distributions, then again we have
\[
\mathcal{D} (M) \hat\otimes \mathcal{D} (N) = \mathcal{D} (M \times N).
\]
\end{itemize}
\end{itemize}

\item $\H$ denotes the upper half plane. 
\end{itemize}

\section{\BV formalism and effective renormalization}\label{section-BV}

In this section, we collect basics and fix notations on the quantization of gauge theories in the \BV (BV) formalism. We explain Costello's homotopic renormalization theory of \BV quantization and illustrate its geometric structures by a one-dimensional example of topological quantum mechanics. In the next section, we give a self-contained realization of Costello's framework in two dimensions based on heat kernel estimates. 

\subsection{\BV algebras and the master equation}
\begin{defn} A differential \BV (BV) algebra is a triple $(\A, Q, \Delta)$ where
\begin{itemize}
\item $\A$ is a $\Z$-graded commutative associative unital algebra. 
\item $Q: \A\to \A$ is a derivation of degree $1$ such that $Q^2=0$. 
\item $\Delta: \A \to \A$ is a linear operator of degree $1$ such that $\Delta^2=0$.
\item $\Delta$ is a ``second-order" operator. Precisely, define the binary operator $\fbracket{-,-}: \A\otimes \A \to \A$ 
 $$
    \fbracket{a,b}:=\Delta(ab)-(\Delta a)b- (-1)^{\bar a}a \Delta b, \quad a, b\in \A. 
$$
Then $\fbracket{-,-}$ satisfies the following graded Leibnitz rule
$$
\fbracket{a, bc}=\fbracket{a,b}c+(-1)^{(\bar a+1)\bar b}b\fbracket{a,c}, \quad a, b, c\in \A. 
$$

\item $Q$ and $\Delta$ are compatible: $[Q, \Delta]=Q\Delta+\Delta Q=0$. 
\end{itemize}
\end{defn}

Here $\Delta$ is called the \emph{BV operator}. $\fbracket{-,-}$ is called the \emph{BV bracket}, which measures the failure of $\Delta$ being a derivation. It follows  that $\fbracket{-,-}$ defines a Poisson bracket of degree $1$ satisfying
\begin{itemize}
\item $\fbracket{a,b}=(-1)^{\bar a \bar b}\fbracket{b,a}$. 
\item $\fbracket{a, bc}=\fbracket{a,b}c+(-1)^{(\bar a+1)\bar b}b\fbracket{a,c}$. 
\item $\Delta\fbracket{a,b}=-\fbracket{\Delta a, b}-(-1)^{\bar a}\fbracket{a, \Delta b}$. 
\end{itemize}

The $(Q,\Delta)$-compatibility condition implies the following Leibniz rule
$$
Q\fbracket{a,b}=-\fbracket{Qa, b}-(-1)^{\bar a}\fbracket{a, Qb}. 
$$

\begin{defn}
Let $(\A, Q, \Delta)$ be a differential BV algebra. A degree $0$ element $I\in \A_0$ is said to satisfy \emph{classical master equation} (CME) if 
$$
QI+{1\over 2}\fbracket{I,I}=0.
$$ 
\end{defn}

If $I$ solves CME, then it is easy to see that $Q+\fbracket{I,-}$ defines a differential on $\A$, which can be viewed as a Poisson deformation of $Q$. However, it may not be compatible with $\Delta$.  A sufficient condition for the compatibility is the ``divergence freeness" $\Delta I=0$. A slight generalization of this is the following. 

\begin{defn}
Let $(\A, Q, \Delta)$ be a differential BV algebra. A degree $0$ element $I\in \A[[\hbar]]$ is said to satisfy \emph{quantum master equation} (QME) if 
$$
QI+\hbar \Delta I+{1\over 2}\fbracket{I,I}=0.
$$
Here $\hbar$ is a formal variable representing the quantum parameter. 
\end{defn}

 The ``second-order" property of $\Delta$ implies that QME is formally equivalent (without worrying about the powers of $\hbar^{-1}$) to 
$$
   (Q+\hbar \Delta)e^{I/\hbar}=0. 
$$
If we decompose $I=\sum\limits_{g\geq 0}I_g\hbar^g$, then the $\hbar\to0$ limit of QME is precisely CME
$$
QI_0+{1\over 2}\fbracket{I_0, I_0}=0.
$$

We can rephrase the $(Q,\Delta)$-compatibility as $Q+\hbar \Delta$ squares to zero. It is direct to check that QME implies that $Q+\hbar \Delta+\fbracket{I,-}$ squares to zero as well. 

\subsection{Odd symplectic space and the toy model}
We discuss a toy model of differential BV algebra via $(-1)$-shifted symplectic space. This serves as the main motivation of our quantum field theory examples. 

Let $(V,Q)$ be a finite dimensional dg vector space. The differential $Q: V\to V$ induces a differential on various tensors of $V, V^*$, still denoted by $Q$. Let 
$$
   \omega \in \wedge^2 V^*, \quad Q(\omega)=0, 
$$
be a $Q$-compatible symplectic structure such that $\deg(\omega)=-1$. The non-degeneracy of the symplectic pairing $\omega$ implies that given $\varphi\in V^*$, there exists a unique $v_\varphi\in V$ such that 
$$
   \varphi(-)=\omega(-,v_\varphi).
$$
The condition $\deg(\omega)=-1$ says 
$$
\deg(v_\varphi)=\deg(\varphi)+1.
$$  
This gives an isomorphism of graded vector spaces
\begin{align*}
     V^* &\cong V[1],\quad \varphi \to v_\varphi[1]. 
\end{align*}
 
Let $K=\omega^{-1}\in \Sym^2(V)$ be the Poisson kernel of degree $1$ under
\begin{eqnarray*}
     \wedge^2 V^* &\cong \Sym^2(V)[2]\\
      \omega  &  K[2]
\end{eqnarray*}
where we have used the isomorphism 
\begin{align*}
\wedge^2(V[1])&\cong \Sym^2(V)[2]\\
a[1]\wedge b[1] &\to (-1)^{|b|} a\cdot b [2], \quad a, b \in V.  
\end{align*}

Let 
$$
\OO(V):=\widehat{\Sym}(V^*)=\prod_{n}\Sym^n(V^*).
$$
Then $(\OO(V), Q)$ is a graded-commutative dga. 

The degree $1$ Poisson kernel $K$ defines the following BV operator 
$$
    \Delta_K:=\pa_K: \OO(V) \to \OO(V) \quad \mbox{by}
$$
$$   
     \Delta_K(\varphi_1\cdots \varphi_n)=\sum_{i,j}\pm (K, \varphi_i\otimes \varphi_j) \varphi_1\cdots \hat{\varphi_i} \cdots \hat{\varphi_j}\cdots  \varphi_n, \quad \varphi_i\in V^*. 
$$
Here $(K, \varphi_i\otimes \varphi_j)$ denotes the natural paring between $V\otimes V$ and $V^*\otimes V^*$. $\pm$ is the Koszul sign by permuting $\varphi_i$'s.  The following lemma is well-known (we follow the presentations in \cite{Kevin-book} which will be relevant for our later discussions on effective renormalization). 

\begin{lem}\label{lem-dgs-BV} $(\OO(V), Q, \Delta_K)$ defines a differential BV algebra.
\end{lem} 
\begin{proof}[Sketch of proof] The compatibility of $\omega$ and $Q$ implies $Q(K)=0$, hence 
$$
[Q, \Delta_K]=[Q, \pa_K]=[\pa_{Q(K)}]=0. 
$$
Other properties of differential BV algebra are easily verified. 
\end{proof}
The above construction can be summarized as 
$$
\mbox{$(-1)$-shifted dg symplectic}  \Longrightarrow \mbox{differential BV}.
$$

\begin{rmk}\label{rmk-Poisson} Lemma \ref{lem-dgs-BV} holds for any $K\in \Sym^2(V)$ of degree $1$ such that $Q(K)=0$. In particular, it applies to the case for $(-1)$-shifted dg Poisson structure where $K$ may be degenerate (so $\omega$ does not exist) but $\Delta_K$ is still well-defined. We will see such an example in Section 4. 
\end{rmk}

\subsection{UV problem and homotopic renormalization}
Let us now move on to discuss examples of quantum field theory that we will be mainly interested in. 

\subsubsection{The ultra-violet problem}
One important feature of quantum field theory is its infinite dimensionality. It leads to the main challenge in mathematics to construct measures on infinite dimensional space (called the path integrals). It is also the source of the difficulty of ultra-violet divergence and the motivation for the celebrated idea of renormalization in physics.  Let us address some of these issues via the \BV formalism. 

In the previous section, we discuss the $(-1)$-shifted dg symplectic space $(V, Q, \omega)$. There $V$ is assumed to be finite dimensional. This is the reason we call it ``toy model". Typically in quantum field theory, $V$ will be modified to be the space of smooth sections of certain vector bundles on a smooth manifold, while the differential $Q$ and the pairing $\omega$ come from something ``local". Such $V$ will be called the \emph{space of fields}, which is evidently a very large space with delicate topology. Nevertheless, let us naively perform similar constructions as that in the toy model. 

More precisely, let $X$ be a smooth oriented manifold without boundary. Let $E$ be a $\Z$-graded vector bundle on $X$ of finite rank. Our space of fields $\E$ replacing $V$ will be the space of smooth global sections with compact support
$$
   \E= \Gamma_c(X, E)
$$
equipped with a differential operator of degree $1$
$$
   Q: \E\to \E
 $$
such that $Q^2=0$. We assume that $(\E, Q)$ is an elliptic complex. The symplectic pairing will be 
$$
   \omega(s_1, s_2):= \int_X \bracket{s_1, s_2}, \quad s_1, s_2 \in \E, 
$$
where 
$$
   \bracket{-,-}: E \otimes E \to \Dens(X)
$$
is a non-degenerate graded skew-symmetric bundle morphism of degree $-1$. Here the line bundle $\Dens(X)$ sits at degree $0$. The symplectic pairing $\omega$ is required to be compatible with $Q$.

To perform the toy model construction, we need the following steps
\begin{enumerate}
\item The dual vector space $\E^*$ (analogue of $V^*$). 

This can be defined as the space of continuous linear duals of $\E$ (i.e. distributional sections of the bundle $\Hom(E, \Dens(X)$)
$$
    \E^*:= \Hom(\E, \R)
$$
where $\Hom$ is the space of continuous maps. 
\item The tensor space $(\E^*)^{\otimes k}$ (analogue of $(V^*)^{\otimes k}$). 

This can be defined via the completed tensor product for distributions 
$$
      (\E^*)^{\otimes k}= \E^*\hat \otimes \cdots \hat \otimes \E^*,
$$
i.e. the distributional sections of the bundle $\Hom(E \boxtimes \cdots \boxtimes E, \Dens(X\times \cdots \times X))$ over $X\times \cdots \times X$.  $\Sym^k(\E^*)$ is defined to be the subspace of $(\E^*)^{\otimes k}$ by taking invariants of the graded permutations. Then we have a well-defined notion (via distributions)
$$
   \OO(\E):=\prod_{k\geq 0} \Sym^k(\E^*)
$$
as the analogue of $\OO(V)$. 

\item The Poisson kernel $K_0=\omega^{-1}$ (the analogue of $K$). 

In the finite-dimensional setting, the Poisson kernel is simply the inverse matrix of the symplectic pairing.  In our infinite-dimensional setting, this amounts to asking the Poisson kernel to be the integral kernel of the identity operator with respect to $\omega$, that is, $K_0$ satisfying
 \[\omega(K_0 (x,y),\phi(y) )=\phi(x)\qquad\text{for all}\quad \phi\in \mc E.\]
  Since $\omega$ is given by integration, $K_0$ behaves like a $\delta$-function. We can view $K_0$ as a distributional section of $E\times E$ supported on the diagonal of $X\times X$. Again $K_0$ is graded symmetric. 
  
 Precisely, let $\E^\prime$ be the space of distributional sections of $E$ (See \hyperlink{Conventions}{Conventions}). The pairing $\omega$ induces a natural identification of graded topological vector spaces
$$
\E^\prime[1] \iso   \E^*. 
$$
  
  Let $\Sym^2(\E)\subset \E\hat\otimes \E$ be the subspace of graded-symmetric elements and $\Sym^2(\E^\prime)\subset \E^\prime\hat\otimes \E^\prime$ be the corresponding distributional sections. Unlike the finite dimensional case, $K_0$ is not an element of $\Sym^2(\E)$, but in fact a distributional section 
$$
   K_0\in \Sym^2(\E^\prime).
$$ 
The isomorphism $\E^\prime[1] \iso   \E^*$ induces a natural pairing 
$$
\bracket{\E^\prime\hat\otimes \E^\prime} \times  \bracket{\E\hat\otimes \E}\to \C. 
$$
Then $K_0$ is determined by the value of its pairing with an arbitrary element $\psi_1\otimes\psi_2 \in \E\otimes \E$ 
$$
K_0 \times (\psi_1\otimes\psi_2 )\to \omega(\psi_1, \psi_2). 
$$

\end{enumerate}
It is at Step (3) where we get trouble. In fact, if we naively define the BV operator 
$$
   \Delta_{K_0}: \OO(\E)\stackrel{?}{\to} \OO(\E),
$$
then $\Delta_{K_0}$ is \emph{ill-defined}, since we can not pair a distribution $K_0$ with another distribution from $\OO(\E)$. This difficulty originates from the infinite dimensional nature of the problem. 

\subsubsection{Homotopic renormalization} The solution to the above problem requires the method of renormalization in quantum field theory. There are several different approaches to renormalizations. We will work with Costello's homotopic theory \cite{Kevin-book} in this paper that will be convenient for our applications. 

The key observations are 
\begin{enumerate}
\item $K_0$ is a $Q$-closed distribution: $Q(K_0)=(Q\otimes 1+1\otimes Q) K_0=0$. 
\item elliptic regularity: there is a canonical isomorphism (Atiyah-Bott Lemma \cite{AB})
$$
H^*(\text{smooth}, Q)\iso H^*(\text{distribution}, Q).
$$
\end{enumerate}
It follows that we can find a distribution $P_r\in \Sym^2(\E^\prime)$ and a smooth element $K_r\in \Sym^2(\E)$ such that
$$
   K_0=K_r+ Q(P_r). 
$$
$P_r$ is the familiar notion of a parametrix. In other words, $K_r$ is a smooth representative of the cohomology class of $K_0$. 

\begin{defn} $K_r$ will be called the renormalized BV kernel with respect to the parametrix $P_r$. 
\end{defn}

Since $K_r$ is smooth, there is no problem to pair $K_r$ with distributions. The same formula as in the toy model leads to 

\begin{lem-defn}\label{defn-BV-kernel} We define the renormalized BV operator $\Delta_{K_r}$ 
$$
   \Delta_{K_r}: \OO(\E)\to \OO(\E)
$$
via the smooth renormalized BV kernel $K_r$. The triple $(\OO(\E), Q, \Delta_{K_r})$ defines a differential BV algebra, called the renormalized differential BV algebra (with respect to $P_r$). 
\end{lem-defn}

Therefore $(\OO(\E), Q, \Delta_{K_r})$ can be viewed as a homotopic replacement of the original naive problematic differential BV algebra.  As the formalism suggests, we need to understand relations between difference choices of the parametrices. 

Let $P_{r_1}$ and $P_{r_2}$ be two parametrices, $K_{r_1}$ and $K_{r_2}$ be the corresponding renormalized BV kernels. Let us denote
$$
  P_{r_1}^{r_2}:=P_{r_2}-P_{r_1}. 
$$
Since $Q(P_{r_1}^{r_2})=K_{r_1}-K_{r_2}$ is smooth, $P_{r_1}^{r_2}$ is smooth itself by elliptic regularity. 

\begin{defn} $P_{r_1}^{r_2}$ will be called the regularized propagator. 
\end{defn}

\begin{eg}\label{eg-heat-kernel} Typically, suppose we have an adjoint operator $Q^\dagger$ such that 
$
     [Q, Q^\dagger]
$
is a generalized Laplacian. For example, we may pick a metric on $\E$ and define the adjoint $Q^\dagger$ via the $L^2$ inner product. In all the examples that we will be dealing with, such constructed $[Q, Q^\dagger]$ will be a generalized Laplacian. 

Given $t>0$,  let $K_t$ be the integral kernel for the heat operator $e^{-t[Q, Q^\dagger]}$ with respect to the symplectic pairing $\omega$. In other words, $K_t$ is a section of $\E \boxtimes \E$ that is graded-symmetric and satisfies 
 \[\omega(K_t (x,y),\phi(y) )=\bracket{e^{-t[Q, Q^\dagger]}\phi}(x)\qquad\text{for all}\quad \phi\in \mc E.\]
By the theory of generalized Laplacian, such integral kernel $K_t$ exists and is smooth any for $t>0$. It approaches the $\delta$-function distribution when $t\to 0$ (see for example \cite{BGV}). 

The operator equation (using the fact that $[Q, Q^\dagger]$ graded commutes with $Q$ and $Q^\dagger$)
$$
  \bbracket{Q, \int_{t_1}^{t_2} Q^\dagger e^{-t[Q, Q^\dagger]}dt}=\int_{t_1}^{t_2} \bbracket{Q,Q^\dagger} e^{-t[Q, Q^\dagger]}dt=e^{-t_1[Q, Q^\dagger]}- e^{-t_2[Q, Q^\dagger]}
$$
is translated to the following kernel equation 
$$
Q(P_{t_1}^{t_2}):=  (Q\otimes 1+1\otimes Q) P_{t_1}^{t_2}=K_{t_1}-K_{t_2}. 
$$
Here 
$$
   P_{t_1}^{t_2}=\int_{t_1}^{t_2} (Q^\dagger\otimes 1) K_t dt 
$$
is the integral kernel for the operator $Q^\dagger e^{-t[Q, Q^\dagger]}$ with respect to the symplectic pairing $\omega$. 

We find that $K_t$  can be viewed as a renormalized BV kernel for any $t>0$. In this case,  the paramatrix is $P_0^t$ and the regularized propagator is $P_{t_1}^{t_2}$. Such heat kernel renormalization will be our main applications. 
\end{eg}

Next we explain how to formulate quantum master equations. The key obervation is that $P_{r_1}^{r_2}$ can be viewed as a homotopy linking two different renormalized differential BV algebras. 
In fact, analogous to the definition of renormalized BV operator, 
\begin{defn}
We define $\pa_{P_{r_1}^{r_2}}: \OO(\E)\to \OO(\E)$ as the second-order operator of contracting with the smooth kernel $P_{r_1}^{r_2}\in \Sym^2(\E)$ (see also Conventions). 
\end{defn}

\begin{lem}\label{lem-HRG} The following equation holds formally as operators on $\OO(\E)[[\hbar]]$
$$
     \bracket{Q+\hbar \Delta_{K_{r_2}}}e^{\hbar \pa_{P_{r_1}^{r_2}}}=e^{\hbar \pa_{P_{r_1}^{r_2}}}\bracket{Q+\hbar \Delta_{K_{r_1}}},
$$
i.e., the following diagram commutes
$$
\xymatrix{
    \OO(\E)[[\hbar]] \ar[rr]^{Q+\hbar \Delta_{K_{r_1}}} \ar[d]_{\exp\bracket{\hbar \pa_{P_{r_1}^{r_2}}}} && \OO(\E)[[\hbar]] \ar[d]^{\exp\bracket{\hbar \pa_{P_{r_1}^{r_2}}}}\\
    \OO(\E)[[\hbar]] \ar[rr]_{Q+\hbar \Delta_{K_{r_2}}} && \OO(\E)[[\hbar]]
}
$$
\end{lem}
\begin{proof}[Sketch of proof] $Q(P_{r_1}^{r_2})=K_{r_1}-K_{r_2}$ implies
$$
  \bbracket{Q, \pa_{P_{r_1}^{r_2}}}=\Delta_{K_{r_1}}-\Delta_{K_{r_2}}. 
$$
It follows that 
$$
\exp\bracket{\hbar \pa_{P_{r_1}^{r_2}}} Q \exp\bracket{-\hbar \pa_{P_{r_1}^{r_2}}}=Q-  \hbar \bbracket{Q, \pa_{P_{r_1}^{r_2}}}=Q-\hbar \Delta_{K_{r_1}}+\hbar\Delta_{K_{r_2}}.
$$
Since the operators $\Delta_{K_{r_1}}, \Delta_{K_{r_2}}, \pa_{P_{r_1}^{r_2}}$ graded commute with each other, we find
$$
\exp\bracket{\hbar \pa_{P_{r_1}^{r_2}}} 
\bracket{Q+\hbar\Delta_{K_{r_1}}} \exp\bracket{-\hbar \pa_{P_{r_1}^{r_2}}}=Q+\hbar \Delta_{K_{r_2}}.
$$
This proves the Lemma. See \cite[Chapter 5, Section 9]{Kevin-book} for further discussions. 
\end{proof}

\begin{defn}\label{defn-stable} Let 
$$
\OO^+(\E)[[\hbar]]:=\Sym^{\geq 3}(\E^*)+\hbar \OO(\E)[[\hbar]]\subset \OO(\E)[[\hbar]]
$$ 
be the subspace of those functionals which are at least cubic modulo $\hbar$. Given any two parametrices $P_{r_1}, P_{r_2}$, we define the homotopic renormalization group (HRG) operator 
\begin{align*}
   W(P_{r_1}^{r_2}, -)&:  \OO^+(\E)[[\hbar]]\to \OO^+(\E)[[\hbar]]\\
    W(P_{r_1}^{r_2}, I)&:=\hbar \log \bracket{e^{\hbar \pa_{P_{r_1}^{r_2}}}e^{I/\hbar}}. 
\end{align*}
\end{defn}
The real content of the above formula is 
$$
  W(P_{r_1}^{r_2}, I)=\sum_{\Gamma\ \text{connected}}  {W_\Gamma(P_{r_1}^{r_2}, I)}
$$
where the summation is over all connected Feynman graphs with $P_{r_1}^{r_2}$ being the propagator and $I$ being the vertex (see for example \cite{graph} for Feynman graph techniques). Wick's Theoreom identifies the above two expressions. In particular,  the graph expansion formula implies that HRG operator $W(P_{r_1}^{r_2},-)$ is well-defined on $\OO^+(\E)[[\hbar]]$. We refer to \cite{Kevin-book} for a thorough discussion in the current context. 

Lemma \ref{lem-HRG} motivates  the following definition. 

\begin{defn}[\cite{Kevin-book}] A solution of effective quantum master equation is an assignment $I[r]\in \OO(\E)[[\hbar]]$ for each parametrix $P_r$ satisfying
\begin{itemize}
\item Renormalized quantum master equation (RQME)
$$
(Q+\hbar \Delta_{K_{r}})e^{I[r]/\hbar}=0.
$$
\item Homotopic renormalization group flow equation (HRG): for any two parametrices $P_{r_1}, P_{r_2}$,
$$
 I[r_2]=W(P_{r_1}^{r_2}, I[r_1]). 
 $$
\end{itemize}
\end{defn}
RQME and HRG are compatible by Lemma \ref{lem-HRG}. A solution of effective quantum master equation is completely determined by its value at a fixed parametrix. Functionals at other paramatrices are obtained via HRG. 

\begin{rmk} Here we adopt the name ``homotopic RG flow" as opposed to the name "RG flow" in \cite{Kevin-book}. If the manifold preserves a rescaling symmetry, it will induce a rescaling action on the solution space of effective quantum master equations. Such flow equation will be called RG flow to be consistent with the physics terminology. 

\end{rmk}

\subsubsection{Locality and counter-term technique}\label{section-CT}
In practice, we obtain solutions of renormalized quantum master equations via local functionals with the help of the method of counter-terms. We explain this construction in this subsection. 

\begin{defn}
A functional $I\in \OO(\E)$ is called \emph{local} if it can be expressed in terms of an integration of a Lagrangian density $\L$
$$
   I(\phi)=\int_X \mathcal L(\phi), \quad \phi \in \E. 
$$
$\L$ can be viewed as a $\Dens(X)$-valued function on the jet bundle of $\E$. The subspace  of local functionals of $\OO(\E)$ is denoted by $\Ol(\E)$.  We also denote
$$
\Ol^+(\E)[[\hbar]]= \OO^+(\E)[[\hbar]]\cap \Ol(\E)[[\hbar]].
$$
\end{defn}

Let us assume we are in the situation of Example \ref{eg-heat-kernel}. Given $L>0$, we have a smooth regularized BV kernel $K_L$ in terms of the heat kernel. Let $P_\epsilon^L$ be the regularized propagator for $0<\epsilon<L<\infty$. 

Given any local functional $I\in \Ol^+(\E)[[\hbar]]$, we can find an $\epsilon$-dependent local functional $I^{CT}(\epsilon)\in \hbar \Ol(\E)[[\hbar]]$, such that the following limit exists
$$
   I[L]=\lim_{\epsilon\to 0}W(P_\epsilon^L, I+I^{CT}(\epsilon))\in \OO^+(\E)[[\hbar]]. 
$$
Here $I^{CT}(\epsilon)$ has a singular dependence on $\epsilon$ when $\epsilon\to 0$, and this singularity exactly cancels those that come from the naive graph integral $W(P_\epsilon^L, I)$. $I^{CT}(\epsilon)$ is called the \emph{counter-term}, which plays an important role in the renormalization theory of quantum fields. For a proof of the existence of counter-terms in the current context, see for example \cite[Appendix 1]{Kevin-book}. 

By construction, $I[L]$ satisfies HRG for all heat kernel regularizations:  
$$
   I[L_2]=W(P_{L_1}^{L_2}, I[L_1]), \quad \forall 0<L_1, L_2<\infty. 
$$
It remains to analyze the quantum master equation. 

Firstly we observe that although the BV operator $\Delta_L$ becomes singular as $L\to 0$, its associated BV bracket is in fact well-defined on local functionals. This is because only a single delta-function appears in the limit $L\to 0$ of the bracket and it can be naturally integrated within local functionals. 

\begin{defn}\label{defn-BV-bracket} We define the classical BV bracket on $\Ol(\E)$ by
$$
  \fbracket{I_1, I_2}:=\lim_{L\to 0}\fbracket{I_1, I_2}_L, \quad I_1, I_2 \in \Ol(\E). 
$$
\end{defn}

\begin{defn} A local functional $I\in \Ol(\E)$ is said to satisfy classical master equation (CME) if
$$
  QI+{1\over 2}\fbracket{I, I}=0. 
$$
\end{defn}
Therefore classical master equation does not require renormalization. In practice, any local functional with a gauge symmetry can be completed into a local functional that satisfies the classical master equation. This fact lies in the heart of the \BV formalism. 

\begin{rmk}\label{rmk-CME}
Given $I_0\in \Ol(\E)$ satisfying the classical master equation, $Q+\fbracket{I_0,-}$ defines a square-zero vector field on $\E$. The physical meaning is that it generates the infinitesimal gauge transformation. In mathematical terminology, it defines a (local) $L_\infty$ structure on $\E$. 
\end{rmk}

To proceed to construct solutions of effective quantum master equations, let us naively use counter-terms to construct a family $I[L]$ as above from $I_0$. It can be  shown that $I[L]$ satisfies renormalized quantum master equation modulo $\hbar$
$$
   Q I[L]+\hbar \Delta_L I[L]+{1\over 2}\fbracket{I[L], I[L]}_L= O(\hbar). 
$$
Then we need to find quantum corrections deforming $I_0$ via higher order terms in $\hbar$
$$
I_0+ \hbar I_1+\hbar^2 I_2+\cdots \in \Ol^+(\E)[[\hbar]]
$$  
such that the renormalized quantum master equation holds true at all orders of $\hbar$. This  can be formulated as a deformation problem. The quantum corrections may become very complicated in general, and intrinsic obstructions for solving the quantum master equation could exist at certain $\hbar$-order (gauge anomalies). Nevertheless, one of our main purposes here is to understand the geometric meaning of such quantum corrections and find their solutions.

\subsection{Example: Topological quantum mechanics}\label{section-example}
The simplest nontrivial example is when $X$ is one-dimensional. This corresponds to quantum mechanical models. We explain the one dimensional Chern-Simons theory that is studied in detail in \cite{LLG}. In such a geometric situation, the renormalized BV master equation is related to Fedosov's abelian connection on Weyl bundles \cite{Fedosov}. In the next section, we generalize this analysis to two dimensional models.

Let $X=S^1$. Let $V$ be a graded vector space  with a degree 0 symplectic pairing
$$
  (-,-): \wedge^2 V\to \R. 
$$ 
The space of fields will be 
$$
   \E= \Omega^\bullet(S^1)\otimes V. 
$$
Let us denote $X_{dR}$ by the super-manifold whose underlying topological space is $X$ and whose structure sheaf is the de Rham complex on $X$. We can identify $\varphi\in \E$ with a map $\hat \varphi$ between super-manifolds
$$
 \hat \varphi: X_{dR} \to V
$$
whose underlying map on topological spaces is around the constant map to $0\in V$. 

The differential $Q=d_{S^1}$ on $\E$ is the de Rham differential on $\Omega^\bullet(S^1)$. The $(-1)$-shifted symplectic pairing is the pairing 
$$
  \omega(\varphi_1, \varphi_2):=\int_{S^1} (\varphi_1, \varphi_2). 
$$
The induced differential BV structure is precisely the AKSZ-formalism \cite{AKSZ} applied to the one-dimensional $\sigma$-model. Given $I\in \OO(V)$ of degree $k$, it induces an element $\hat I\in \OO(\E)$ of degree $k-1$ via
$$
   \hat I(\varphi):=\int_{S^1} \hat\varphi^*(I), \quad \forall \varphi \in \E. 
$$

We choose the standard flat metric on $S^1$ and use the heat kernel regularization as in Example \ref{eg-heat-kernel}. The following Theorem is established in \cite{LLG}.

\begin{thm}[\cite{LLG}]\label{thm-weyl} Given $I\in \Sym^{\geq 3}(V^*)+\hbar\OO(V)[[\hbar]]$ of degree $1$, the limit 
$$
   \hat I[L]= \lim_{\epsilon \to 0} W(P_\epsilon^L, \hat I)
$$
exists as an degree $0$ element of $\OO^+(\E)[[\hbar]]$. The family $\{\hat I[L]\}_{L>0}$ solves the effective quantum master equation if and only if 
$$
   \bbracket{I, I}_\star=0. 
$$
Here $\OO(V)[[\hbar]]$ inherits a natural Moyal-Weyl product $\star$ from the linear symplectic form $(-,-)$. $[-,-]_\star$ is the commutator with respect to the Moyal-Weyl product. 
\end{thm}

In \cite{LLG}, the above theorem is further extended to a family version of $V$ parametrized by a symplectic manifold. Then the effective quantum master equation is equivalent to a flat connection gluing the Weyl bundle (i.e. the bundle of Weyl algebra with fiberwise Moyal-Weyl product). As beautifully shown by Fedosov \cite{Fedosov}, flat sections of the Weyl bundle leads to a deformation quantization of the Poisson algebra of smooth functions on $X$. A further analysis of the partition function leads to a geometric formulation of the algebraic index theorem \cite{LLG,GLX}.

\section{Vertex algebra and BV master equation}\label{section-VA}
In this section we study \BV quantization of two dimensional chiral deformations of free theories on flat surfaces. We give a self-contained treatment of Costello's homotopic renormalization framework in this case using heat kernel estimates of graph integrals established in \cite{Li-modular}. We show that such two dimensional chiral theories are free of ultra-violet divergence (Theorem \ref{thm-defn}), and establish our main theorem on the correspondence between solutions of the renormalized quantum master equation and Maurer-Cartan elements of the modes Lie algebra of chiral vertex operators (Theorem \ref{main-thm}). We prove the modularity property of generating functions of chiral theories on elliptic curves and establish the polynomial nature of their anti-holomorphic dependence on the moduli (Theorem \ref{thm-modularity}). 

\subsection{Vertex algebra} In this section we collect some basics on vertex algebras that will be used in this paper. We refer to \cite{Kac,Frenkel} for details. 

\subsubsection{Definition of vertex algebras} In this section $\V$ will always denote a $\Z/2\Z$-graded superspace over $\C$. It has two components with different parities 
$$
   \V=\V_{\bar 0}\oplus \V_{\bar 1}, \quad \Z/2\Z=\{\bar 0, \bar 1\}. 
$$
We say $v\in \V$ has parity $p(a)\in \Z/2\Z$ if $v\in \V_{p(v)}$.  

\begin{defn} A \emph{field} on $\V$ is a power series 
$$
  A(z)=\sum_{k\in \Z}A_{(k)}z^{-k-1}\in \End(\V)[[z, z^{-1}]]
$$
such that for any $v\in \V,   A(z) v\in \V((z))$.  We will denote 
$$
  A(z)_+=\sum_{k< 0}A_{(k)} z^{-k-1}, \quad A(z)_-=\sum_{k\geq 0}A_{(k)} z^{-k-1}. 
$$
$A_{(k)}$'s will be called Fourier modes of $A$.

\begin{itemize}
\item Given two fields $A(z), B(z)$, we define their \emph{normal ordered product} 
$$
  :A(z)B(w):= A(z)_+B(w)+(-1)^{p(A) p(B)} B(w)A(z)_-. 
$$
Note that $\End(\V)$ is naturally a superspace.  $p(A), p(B)$ are the parities. 
\item Two fields $A(z), B(z)$ are called \emph{mutually local} if 
$$ 
   A(z)B(w)=\sum_{k=0}^{N-1} {C_k(w)\over (z-w)^{k+1}}+:A(z)B(w):.
$$
for some $N\in \Z^{\geq 0}$ and fields $C_k$. The first term on the right is called the \emph{singular part}, and we shall write
$$
A(z)B(w)\sim \sum_{k=0}^{N-1} {C_k(w)\over (z-w)^{k+1}}.
$$
The above two formulae are called the \emph{operator product expansion} (OPE). 
\end{itemize}
\end{defn}

\begin{defn} A \emph{vertex algebra} is a collection of data
\begin{itemize}
\item (space of states) a superspace $\V$; 
\item (vacuum) a vector $\ket{0}\in \V_{\bar 0}$;
\item (translation operator) an even linear operator $T: \V\to \V$;
\item (state-field correspondence) an even linear operation (vertex operators)
$$
    Y(\cdot, z): \V\to \End \V[[z,z^{-1}]]
$$
taking each $A\in \V$ to a field 
$
  Y(A,z)=\sum\limits_{n\in Z}A_{(n)}z^{-n-1};
$
\end{itemize}
satisfying the following axioms:
\begin{itemize}
\item (vacuum axiom) $Y(\ket{0},z)=\Id_\V$. Furthermore, for any $A\in \V$ we have
$$
   Y(A,z)\ket{0}\in \V[[z]], \quad \text{and}\quad \lim_{z\to 0} Y(A,z)\ket{0}=A. 
$$
\item (translation axiom) $T\ket{0}=0$. For any $A\in \V$, $
   \bbracket{T, Y(A,z)}=\pa_z Y(A,z);
$
\item (locality axiom) All fields $Y(A,z), A\in \V$, are mutually local. 
\end{itemize}
\end{defn}

Let $A, B \in \V$.  Their OPE can be expanded as
$$
   Y(A,z)Y(B,w)=\sum_{n\in \Z}{Y(A_{(n)} \cdot B, w)\over (z-w)^{n+1}},
$$ 
where $\fbracket{A_{(n)}\cdot B}_{n\in \Z}$ can be viewed as defining an infinite tower of products. 

\begin{defn} A vertex algebra $\V$ is called \emph{conformal}, of \emph{central charge} $c\in \C$, if there is a vector $\omega_{vir}\in \V$ (called a conformal vector) such that under the state-field correspondence $Y(\omega_{vir},z)=\sum\limits_{n\in \Z}L_n z^{-n-2}$:
\begin{itemize}
\item $L_{-1}=T$, $L_0$ is diagonalizable on $V$;
\item $\{L_n\}_{n\in \Z}$ span the Virasoro algebra with central charge $c$.
\end{itemize}
\end{defn}

Let $\V$ be a conformal vertex algebra. The field $Y(\omega,z)$ will often be denoted by $T(z)$, called the \emph{energy momentum tensor}.  $\V$ can be decomposed as
$$
  \V=\bigoplus_{\alpha} \V^\alpha, \quad L_0|_{\V_\alpha}=\alpha. 
$$
$A\in \V^\alpha$ is said to have conformal weight $\alpha$. Its field will often be expressed by
$$
   Y(A,z)=\sum_{k} A_kz^{-k-\alpha}, \quad A_k: \V^{\bullet}\to \V^{\bullet-k},
$$
where $A_k$ has conformal weight $-k$. In previous notations, $A_k=A_{(k+\alpha-1)}$. 

\subsubsection{Modes Lie algebra}\label{defn-Lie-algebra} Following the presentation in \cite{Frenkel},  we associate a canonical Lie algebra to a vertex algebra via Fourier modes of vertex operators. 

\begin{defn}\label{defn-modes} Given a vertex algebra $\V$, we define a Lie algebra, denoted by $\oint \V$, as follows. As a vector space, the Lie algebra $\oint \V$ has a basis given by $A_{(k)}$'s
$$
  \oint \V:=\text{Span}_{\C}\fbracket{\oint dz z^k Y(A,z):=A_{(k)}}_{A\in \V, k\in \Z}. 
$$
The Lie bracket is determined by the OPE  (Borcherds commutator formula)
$$
    \bbracket{A_{(m)}, B_{(n)}}=\sum_{j\geq 0}\binom{m}{j}\bracket{A_{(j)}B}_{(m+n-j)}. 
$$ 
\end{defn}

Here we use a different notation $\oint$ than $U(\cdot)$ in \cite[Section 4.1]{Frenkel} to emphasize its nature of mode expansion. The commutator relations are better illustrated formally in terms of residues
$$
\bbracket{\oint dz z^{m} Y(A,z), \oint dw w^n Y(B,w)}=\oint dw w^n \oint_{w}dz z^m \sum_{j\in \Z}{Y(A_{(j)} \cdot B, w)\over (z-w)^{j+1}},
$$
where $\oint_w dz:=\int_{C_w}{dz\over 2\pi i}$ is the integration over a small loop $C_w$ around $w$. 

Equivalently, the vector space $\oint \V$ can be described as
$$
   \oint \V= \V[z,z^{-1}]/\im \pa
$$
where $\pa (A \otimes z^k):=T(A)\otimes z^k+ kA\otimes z^{k-1}, A\in V$. Then $A\otimes z^k$ represents the Fourier mode $\oint dz z^k Y(A,z)$ and $\im \pa$ represents the space of total derivatives. 

\begin{rmk} For an element $I$ of $\V$, its zero mode $I_{0}$ is a derivation. Hence if $I_{0}$ has degree $1$ and squares to zero, then the pair $(\V, I_{0})$ is a dg vertex algebra whose cohomology plays the role of BRST reduction.  See \cite[Lemma 3.3.8 and Section 5.7.3]{Frenkel} for a detailed discussion on this. One of the main purpose of this section is to developed a rigorous field theory realization of such BRST reduction within the \BV quantization method. This is illustrated by Theorem \ref{main-thm}. 
\end{rmk}

\subsection{Free theory examples} In this subsection, we describe several free theory examples of vertex algebras given by $\beta \gamma-bc$ systems and chiral bosons (Heisenberg vertex algebra). Our main goal in this paper is to study chiral deformations of such free theories. The basic reference are \cite{Frenkel, CFT}.

\subsubsection{$\beta \gamma-b c$ system} \label{section-VA-example} Let $\h=\oplus_{\alpha\in \Q} \h^\alpha$, where $\h^\alpha=\h^\alpha_{\bar 0}\oplus \h^\alpha_{\bar 1}$,  be a $\Q$-graded superspace. Here the $\Q$-grading is the conformal weight and we assume for simplicity only finitely many weights appear in $\h$. The data of conformal weight will not be used for our main theorem in Section \ref{section-2d}. It gives important information on various selection rules.

Let $\h$ be equipped with an even symplectic pairing
$$
   \abracket{-,-}: \wedge^2 \h\to \C
$$
which is of conformal weight $-1$, that is, the only nontrivial pairing is 
$$
    \abracket{-,-}: \h^{\alpha}\otimes \h^{1-\alpha}\to \C. 
$$
For each $a\in \h^\alpha$, we associate a field
$$
    a(z)=\sum_{r\in \Z-\alpha} a_r z^{-r-\alpha}.
$$
We define their singular part of OPE by
$$
   a(z)b(w)\sim \bracket{i\hbar \over \pi}{\abracket{a,b}\over (z-w)}, \quad \forall a, b\in \h. 
$$
This is equivalent to the commutator relations
$$
   [a_{r}, b_{s}]={i\hbar\over \pi }\abracket{a,b} \delta_{r+s,1}, \quad \forall a, b\in \h, r\in \Z-\alpha, s\in \Z+\alpha.
$$
The vertex algebra $\V[\h]$ that realizes the above OPE relations is given by the Fock representation space. The vacuum vector satisfies 
$$
    a_r\ket{0}=0,\quad \forall a\in \h^\alpha, r+\alpha>0,
$$
and $\V[\h]$ is freely generated from the vacuum by the operators $\{a_r\}_{r+\alpha\leq 0}, a\in \h^\alpha$. For any $a\in \h$,  $a(z)$ becomes a field acting naturally on the Fock space $\V[\h]$.  $\V[\h]$ is a conformal vertex algebra, with energy momentum tensor 
$$
   T(z)={1\over 2}\sum_{i,j}\omega_{ij}:\pa a^i(z)a^j(z):
$$
where $\{a^i\}$ is a basis of $\h$, $\omega_{ij}$ is the inverse matrix of $\abracket{a^i,b^j}$ so that $\sum_k \omega_{ik}\abracket{a^k,a^j}=\delta^j_i$.  The central charge is $\dim \h_{\bar 0}-\dim \h_{\bar 1}$.

\begin{rmk}In quantum field theory language, this system is described by the free chiral theory 
$$
   {1\over 2} \sum_{i,j} \omega_{ij}  \int d^2z a^i \dbar a^j
$$
and $a^i$'s are fundamental fields. When $\h=\h_{\bar 0}$ is purely bosonic, the associated vertex algebra is the $\beta\gamma$ system (or the chiral differential operators \cite{CDR, CDO}). When $\h=\h_{\bar 1}$ is purely fermionic, the associated vertex algebra is the $bc$ system. 
\end{rmk}

Under the state-field correspondence, the vertex algebra $\V[\h]$ can be identified with the polynomial algebra 
$$
   \V[\h]\iso \C\bbracket{\pa^k a^i}, \quad a^i\ \text{is a basis of}\ \h, k\geq 0. 
$$

We also denote its formal completion with respect to the polynomial degree by (See \hyperlink{Conventions}{Conventions})
$$
   \V[[\h]]:=\C[[\pa^k a^i]]. 
$$
$\V[\h],\V[[\h]]$ can be viewed as the chiral analogue of the Weyl algebra.

\subsubsection{Heisenberg vertex algebra}\label{sec:chiral-boson} Let $\h=\h_{\bar 0}$ be a finite dimensional vector space (of pure even type) equipped with an inner product $\abracket{-,-}$.
For each $a\in \h$, we associate a field
$$
    a(z)=\sum_{r\in n} a_n z^{-n-1}.
$$
We define their singular part of OPE by
$$
   a(z)b(w)\sim \bracket{i\hbar \over \pi}{\abracket{a,b}\over (z-w)^2}, \quad \forall a, b\in \h. 
$$
This is equivalent to the commutator relations
$$
   [a_{n}, b_{m}]={i\hbar\over \pi }\abracket{a,b} n\delta_{n+m,0}, \quad \forall a, b\in \h, \quad m, n \in \Z.
$$
The Heisenberg vertex algebra $\V[\h]$ that realizes the above OPE relations is also given by the Fock representation space. The vacuum vector satisfies 
$$
    a_n\ket{0}=0,\quad \forall a\in \h,\quad n\geq 0,
$$
and $\V[\h]$ is freely generated from the vacuum by the operators $\{a_n\}_{n<0}, a\in \h$. For any $a\in \h$,  $a(z)$ becomes a field acting naturally on the Fock space $\V[\h]$.  $\V[\h]$ is a conformal vertex algebra, with energy momentum tensor 
$$
   T(z)={1\over 2}\sum_{i}:\pa a^i(z)\pa a^i(z):
$$
where $\{a^i\}$ is an orthonormal basis of $\h$.  The central charge is $\dim \h$.

\begin{rmk}In quantum field theory language, this system is described by free chiral bosons
$$
   \sum_{i} \int d^2z \pa_z \varphi^i \dbar_{\bar z} \varphi^i.
$$
The fields $a^i$ can be identified with $a^i=\pa_z \phi^i$ which has conformal weight $1$. 
\end{rmk}

Under the state-field correspondence, the vertex algebra $\V[\h]$ can be identified with the polynomial algebra 
$$
   \V[\h]\iso \C\bbracket{\pa^k a^i}, \quad a^i\ \text{is a basis of}\ \h, k\geq 0. 
$$
Under this isomorphism, the vacuum $\ket{0}$  corresponds to $1$, the operator $a^i_{-n-1}$ corresponds to multiplication by ${\pa^n a^i\over n!}$, and the translation $T$ corresponds to the differentiation $\pa$. 

We also denote its formal completion with respect to the polynomial degree by (See \hyperlink{Conventions}{Conventions})
$$
   \V[[\h]]:=\C[[\pa^k a^i]]. 
$$

This example will play an important role in our application in Section \ref{section-B}.

\subsection{Chiral deformation of $\beta \gamma-bc$ system on flat spaces}\label{section-2d} Let $\Sigma$ be a complex curve, $K_\Sigma$ be its canonical line bundle. We are interested in the following data as a BV set-up:
\begin{itemize}
\item $(E,\delta)$ is a $\Z$-graded holomorphic bundle $E$ on $\Sigma$ with a square-zero holomorphic differential operator $\delta$ of degree 1;
\item A degree $0$ symplectic pairing of bundles 
$$
      (-,-):  E \otimes E \to K_\Sigma.
$$
Here $K_\Sigma$ is viewed as a line bundle concentrated at degree $0$. 
\end{itemize}
The space of quantum fields associated to the above data will be compactly supported Dolbeault complex
$$
 \E= \Omega_c^{0,\bullet}(\Sigma, E). 
$$
The $(-1)$-symplectic pairing $\omega$ on $\E$ is obtained via
$$
    \xymatrix{ \E \otimes \E \ar[dr]_{\omega} \ar[rr]^{(-,-)} && \Omega_c^{0,\bullet}(\Sigma, K_\Sigma)= \Omega_c^{1, \bullet}(\Sigma)\ar[dl]^{\int_{\Sigma}}\\
    & \C &
    }
$$
and  we require its compatibility with $\delta$. 

Then the triple $(\E, Q=\dbar+\delta, \omega)$ defines an infinite dimensional $(-1)$-symplectic structure in the sense of section \ref{section-BV}. Our main goal in this section is to study the associated renormalized quantum master equation, which can be viewed as chiral deformations of $\beta\gamma-bc$ systems.

We will mainly consider the flat space in this paper when $\Sigma=\C, \C^*$ or the elliptic curve $E_\tau=\C/ (\Z\oplus \Z\tau)$ ($\tau \in \H$). We assume this from now on. 

\begin{defn}We will fix a coordinate $z$ on $\Sigma$: this is the linear coordinate on $\C$; $z\sim z+1$ on $\C^*$; and $z\sim z+1\sim z+\tau$ on $E_\tau$. Our convention of the volume form is
$$
   {d^2z}:= {i \over 2} dz\wedge d\bar z. 
$$
We also use the following notations.  Let $z=x+i y$ where $x,y$ are real coordinates. Let $\ora{z}$ denotes $(x,y)$. Then $f(\ora{z})$ means a smooth function on $z$ while $f(z)$ means a holomorphic function. 
\end{defn}

Since $\Sigma$ is flat, we will focus on the situation in this paper when $\bracket{E, (-,-)}$ comes from linear data $(\h, \abracket{-,-})$ as follows:
\begin{itemize}
\item $\h=\bigoplus\limits_{m\in \Z} \h_m$ is a $\Z$-graded vector space (the grading is the cohomology degree) and each $\h_m$ is of finite dimension;
\item $\abracket{-,-}: \wedge^2 \h \to \C$ is a degree $0$ symplectic pairing;
\item $E=\Sigma\times \h$, and $\E=\Omega_c^{0,\bullet}(\Sigma)\otimes \h$;
\item $\delta\in   \C\bbracket{\pa\over \pa z} \otimes \bigoplus\limits_{m} \Hom(\h_m, \h_{m+1})$. Here $\C\bbracket{\pa\over \pa z}$ represents translation invariant holomorphic differential operators on $\Sigma$. Moreover, $\delta^2=0$;
\item $\omega\bracket{-,-}$ is induced from fiberwise $\abracket{-,-}$ via the identification $K_\Sigma\iso \OO_{\Sigma} dz$
$$
  \omega \bracket{\varphi_1, \varphi_2}:=\int dz\wedge \abracket{ \varphi_1, \varphi_2}. 
$$
We require that $\delta$ is compatible with $\omega\bracket{-,-}$. 
\end{itemize}

We consider the dual $\Z$-graded vector space $\h^*$ with the induced symplectic pairing still denoted by $\abracket{-,-}$. The $\Z/2\Z$-grading associated to the $\Z$-grading defines the parity on $\h^*$. The extra $\Z$-grading of conformal weight will not be explicit at this stage. At this stage, we assume for simplicity that the pairing $\abracket{-,-}$ on $\h^*$ has conformal weight $-1$ and obtain the vertex algebra $\V[[\h^*]]$ as in Section \ref{section-VA-example}. The same discussion can be easily generalized when the conformal weight $\abracket{-,-}$ is different from $-1$ or examples of chiral bosons. We describe such an example of Heisenberg vertex algebra in Section \ref{section-B}.

\subsubsection{chiral local functional}\label{sec:chiral local}
The relevant chiral local functionals will be described by elements of $\oint \V[[\h^*]]$. 

\begin{defn}\label{defn-functional} Given $I\in \V[[\h^*]][[\hbar]]$, we extend it $\Omega_c^{0,\bullet}[[\hbar]]$-linearly to a map 
$$
   I:  \E \to \Omega_c^{0,\bullet}[[\hbar]]. 
$$
Explicitly, if $I=\sum\pa^{k_1}a_1\cdots \pa^{k_m}a_m\in \V[[\h^*]]$, where $a_i \in \h^*$,  $\varphi \in \E$, then 
$$
    I(\varphi)= \sum \pm \pa^{k_1}_z{a_1(\varphi)}\cdots \pa^{k_m}_z{a_m(\varphi)}.
$$
Here ${a_i(\varphi)}\in \Omega_c^{0,\bullet}$ comes from the natural pairing between $\h^*$ and $\h$. $\pa_z$ is the holomorphic derivative with respect to our prescribed linear coordinate $z$ on $\Sigma$. $\pm$ is the Koszul sign. We associate a local functional $\hat I\in \Ol(\E)[[\hbar]]$ on $\E$  by
$$
   \hat I(\varphi):={i}\int_{\Sigma} dz\  I(\varphi) , \quad \varphi \in \E.  
$$
Note that $\deg(\hat I)=\deg(I)-1$ for each homogenous element $I=\pa^{k_1}a_1\cdots \pa^{k_m}a_m$. It differs by $-1$ because the integration 
map $
\int dz (-)
$ is of cohomology degree $-1$ (only non-vanishing on $\Omega^{0,1}$ component). 
\end{defn}

The following definition is analogous to Definition \ref{defn-stable}. 

\begin{defn} We define the subspace $\V^+[[\h^*]][[\hbar]]\subset \V[[\h^*]][[\hbar]]$ by saying that $I\in \V^+[[\h^*]][[\hbar]]$ if and only if $\lim\limits_{\hbar\to 0}I$ is at least cubic when $\V[[\h^*]]$ is  viewed as a (formal) polynomial ring. 
\end{defn}

\subsection{Ultra-violet finiteness}\label{sec:regularization}
\subsubsection{Regularization} We work with heat kernel regularization as in Example \ref{eg-heat-kernel}.

We will fix the standard flat metric on $\Sigma$. Let $\dbar^*$ denote the adjoint of $\dbar$, and $h_t\in  \cinfty(\Sigma\times\Sigma),t>0$,  be the heat kernel function of the Laplacian operator $H=\bbracket{\dbar, \dbar^*}$. It is normalized by 
$$
  (e^{-tH}f)(\ora{z_1})=\int_{\Sigma} d^2z_2 \ h_t(\ora{z_1},\ora{z_2}) f(\ora{z_2}), \quad t>0. 
$$

For $\Sigma=\C$,  we have $\dbar^*=-2\pa_z \iota_{\bar\pa_{\bar z}}$ and $H=-2\pa_{z}\bar\pa_{\bar z}={1\over 2}\bracket{\pa_x^2+\pa_y^2}$ where $z=x+i y$, and $\iota_{\bar \pa_{\bar z}}$ is the contraction with the vector field ${\pa \over \pa {\bar z}}$. Then explicitly 
$$
   h_t(\ora{z}_1,\ora{z_2})={1\over 2\pi t}e^{-|z_1-z_2|^2/2t}. 
$$
When $\Sigma=\C^*$ or $E_\tau$, $h_t$ is obtained from the above heat kernel on $\C$ by a further summation over the relevant lattices.

Note that $\delta$ commutes with $\dbar$ and $\dbar^*$, hence $H=[Q, \dbar^*]=[\dbar+\delta, \dbar^*]$. The regularized BV kernel $K_L\in \Sym^2(\E)$ is given by 
$$
 K_L(\ora{z_1}, \ora{z_2})= {i}\ h_L(\ora{z_1}, \ora{z_2})(d\bar z_1\otimes 1-1\otimes d\bar z_2)C_\h.  
$$
Here 
$C_\h=\sum_{i,j} \omega_{ij}(a^i\otimes a^j)
$ is the Casimir element where $\{a^i\}$ is a basis of $\h$, $\omega_{ij}$ is the inverse matrix of $\abracket{a^i,b^j}$. The normalization constant is chosen such that 
$$
    (e^{-L H}\varphi)(\ora{z_1})= -{1\over 2}\int_{\Sigma} dz_2 \wedge \abracket{ K_L(\ora{z_1}, \ora{z_2}), \varphi(\ora{z_2})}, \quad \forall \varphi \in \E
$$
where $\abracket{-,-}$ inside the integral is the pairing in the $z_2$-component. The factor ${1\over 2}$ is the symmetry factor, while the extra minus sign comes from passing $K_L$ through $dz_2$. Then $K_0$ is precisely the desired singular Poisson tensor. The regularized propagator is 
$$
  P_\epsilon^L=\int_{\epsilon}^L (\dbar^*\otimes 1)K_u du, \quad 0<\epsilon<L<\infty.
$$

\subsubsection{UV finiteness}
The first important property of our 2d chiral deformed theory is the absence of ultra-violet divergences. 

\begin{thm}\label{thm-defn} Assume the set-up of $\beta\gamma-bc$ system in Section \ref{section-2d}. Let $I\in \V^+[[\h^*]][[\hbar]]$ of degree $1$ and $\hat I$ be the associated chiral local functional  defined in Definition \ref{defn-functional}. Let $W(-,-)$ be the homotopy RG flow operator defined in Definition \ref{defn-stable}. Then the limit 
$$
   {\hat I[L]}:= \lim_{\epsilon \to 0} W(P_\epsilon^L, \hat I), \quad L>0
$$
exists as a degree $0$ element of $\OO^+(\E)[[\hbar]]$. $\{\hat I[L]\}_{L>0}$ defines a family of $\OO^+(\E)[[\hbar]]$ satisfying the homotopic RG flow equation, and $\lim\limits_{L\to 0}\hat I[L]=\hat I$. 
\end{thm}
\begin{proof} When $\Sigma=\C$, the existence of the limit $\epsilon\to 0$ follows from Prop \ref{finiteness lem}. The other cases $\Sigma=\C^*,\Sigma=E_\tau$ follow from the result on $\C$ since the corresponding propagators have the same singular behavior (see the argument of \cite[Lemma 3.1]{Li-modular} on this formal reasoning). The homotopic RG flow equation follows by construction (see also the discussion in Section \ref{section-CT}).
\end{proof}

In terms of the discussion in Section \ref{section-CT}, it says that singular $\epsilon$-dependent counter terms are not needed. In other words, the UV divergence is absent and we say the theory is UV finite.  

We remark that the order of limit is important in the proof of Theorem \ref{thm-defn}. The chiral nature of the problem implies that all potential singularities in $W(P_\epsilon^L, I)$ in fact vanish upon integration by parts before taking the limit $\epsilon\to 0$ \cite{Li-modular}. 

\begin{rmk}\label{rmk:chiral-boson} Theorem \ref{thm-defn} is stated only for $\beta\gamma-bc$ system for convenience. It in fact works also for chiral deformations of chiral bosons. This is because the relevant propagator (regularization of ${1\over (z_1-z_2)^2}$) is given by a further holomorphic derivative of $P_\epsilon^L$ described as above, and Prop \ref{finiteness lem} still holds in this case. We choose not to state a general effective BV renormalization theory for chiral bosons in this paper. It involves further complications by introducing degenerate BV theory.  Instead, we explain by a concrete example in Section \ref{section-B} (see Theorem \ref{thm-chiral-boson}) and illustrate the modifications in practice.

\end{rmk}

\begin{rmk}
The UV property for chiral deformations is known to physicists via the method of point-splitting regularization \cite{Dijkgraaf-chiral, Dou95}. A geometric theory for such regularization is systematically developed in  \cite{Si-Jie}. We use the heat kernel regularization in this paper to fit into the effective renormalization method as developed in \cite{Kevin-book}. 
\end{rmk}

\subsection{Quantum master equation}\label{sec: QME-OPE} In this subsection, we study the quantization problem for chiral deformations of $\beta\gamma-bc$ systems and analyze the renormalized quantum master equation for $\hat I[L]$ constructed in Theorem \ref{thm-defn}. We assume the set-up in Section \ref{section-2d}.

 The differential $\delta$ induces naturally a differential on the formal polynomial ring $\V[[\h^*]]$ (denoted by the same symbol)
$$
    \delta: \V[[\h^*]]\to \V[[\h^*]].
$$
Let us write $\delta=D\otimes\phi$,  where $D$ represents a holomorphic differential operator and $\phi\in \Hom(\h,\h)$. Let $\phi^*\in  \Hom(\h^*,\h^*)$ be the dual of $\phi$. Then in terms of generators, 
$$
  \delta(\pa_z^ka):=\pa_z^k D (\phi^*(a)), \quad a\in \h^*. 
$$
Since $\delta$ commutes with $\pa_z$, it induces a differential on the Fourier modes
$$
  \delta: \oint \V[[\h^*]]\to \oint \V[[\h^*]]. 
$$
Recall we have a natural Lie bracket defined on $\oint \V[[\h^*]][[\hbar]]$ as in Section \ref{defn-Lie-algebra}. 

\begin{lem}The differential $\delta$ and the Lie bracket $[,]$ define a structure of differential graded Lie algebra on $\oint \V[[\h^*]][[\hbar]]$.
\end{lem}
\begin{proof} We leave this formal check to the interested reader. 
\end{proof}

The main theorem in this paper is the following, which can be viewed as the  two dimensional chiral analogue of Theorem \ref{thm-weyl}. 
\begin{thm}\label{main-thm} Assume the set-up of $\beta\gamma-bc$ system in Section \ref{section-2d}. Let $I\in \V^+[[\h^*]][[\hbar]]$ of degree $1$ and $\hat I$ be the associated chiral local functional defined in Definition \ref{defn-functional}.  Let $W(-,-)$ be the homotopy RG flow operator defined in Definition \ref{defn-stable}. Let $P_\epsilon^L$ be the regularized propagator ($\epsilon, L>0$) defined in Section \ref{sec:regularization}.  Define
$$
   {\hat I[L]}= \lim_{\epsilon \to 0} W(P_\epsilon^L, \hat I)
$$
where the limit exists by Theorem \ref{thm-defn}. Let $\oint dz I=I_{(0)}$ be the corresponding mode of the vertex operator $I$ as an element of the modes Lie algebra $\oint \V[[\h]][[\hbar]]$.

Then the  family $\{\hat I[L]\}_{L>0}$ satisfies renormalized quantum master equation if and only if 
$$
  \delta \oint dz I+{1\over 2}\bracket{i\hbar \over \pi }^{-1} \bbracket{\oint dz I, \oint dz I}=0, 
$$
i.e.. $\oint dz I$ is a Maurer-Cartan element of $\oint \V[[\h]][[\hbar]]$. 
\end{thm}

\begin{rmk} As in Remark \ref{rmk:chiral-boson}, this Theorem also holds for chiral deformations of chiral bosons. We will not present a general theory for this, but illustrate how to  set things up and modify in practice by a concrete example in Section \ref{section-B} (Theorem \ref{thm-chiral-boson}).
\end{rmk}

\begin{proof} It  is enough to work with $\Sigma=\C$, since $\C^*, E_\tau$ are quotients by translations and our $\delta$ and $I$ are translation invariant. The quantization on $\C$ is translation invariant and will descent to quotient spaces by translations due to the locality of the obstruction for solving the renormalized quantum master equation \cite[Chapter 5, Section 13]{Kevin-book}. In the following we assume $\Sigma=\C$. 

Since the renormalized quantum master equations 
\begin{equation*}
   Q \hat I[L]+\hbar \Delta_L\hat I[L]+{1\over 2}\fbracket{\hat I[L], \hat I[L]}_L=0, \quad \text{or}\ \bracket{Q+\Delta_L}e^{\hat I[L]/\hbar}=0,  \tag{$\dagger$}
\end{equation*}
are equivalent for each $L$ under homotopy RG flow,  it is enough to analyze the limit $L\to 0$. 

By Theorem \ref{thm-defn},
\begin{align*}
  \bracket{Q+\hbar \Delta_L}  e^{\hat I[L]/\hbar}&=\lim\limits_{\epsilon\to  0}\bracket{Q+\hbar\Delta_L} \bracket{ e^{\hbar \pa_{P_\epsilon^L}}e^{\hat I/\hbar}}\\
  &=\lim\limits_{\epsilon\to  0} e^{\hbar \pa_{P_\epsilon^L}} \bracket{ \bracket{Q+\hbar\Delta_\epsilon}  e^{\hat I/\hbar}}\\
  &={1\over\hbar} \lim\limits_{\epsilon\to  0}e^{\hbar \pa_{P_\epsilon^L}} \bracket{Q \hat I+\hbar\Delta_\epsilon\hat I+{1\over 2}\fbracket{\hat I, \hat I}_\epsilon} e^{\hat I/\hbar}\\
  &={1\over\hbar} \lim\limits_{\epsilon\to  0}e^{\hbar \pa_{P_\epsilon^L}} \bracket{\delta \hat I+{1\over 2}\fbracket{\hat I, \hat I}_\epsilon} e^{\hat I/\hbar}.
  \end{align*}
Here we have observed
\begin{itemize}
\item $\Delta_\epsilon \hat I=0$,  since $\hat I$ is local and $K_\epsilon$ becomes zero when restricted to the diagonal (the factor $d\bar z_1-d\bar z_2$ vanishes when $z_1=z_2=z$). 
\item $\dbar \hat I=0$, since it contributes to a total derivative.
\end{itemize}
Therefore formally
$$
\hbar e^{-\hat I/\hbar}\lim\limits_{L\to  0} \bracket{Q+\hbar \Delta_L}  e^{\hat I[L]/\hbar}=  \delta \hat I +{1\over 2}\lim\limits_{L\to 0}\lim\limits_{\epsilon\to  0} e^{-\hat I/\hbar}e^{\hbar \pa_{P_\epsilon^L}} \bracket{\fbracket{\hat I, \hat I}_\epsilon e^{\hat I/\hbar}}.
$$
The first term matches with that in the theorem up to a factor of $i$ (recall Definition \ref{defn-functional}). It remains to analyze the second term involving effective BV bracket. We proceed in several steps. 

\noindent \emph{Step 1}:  Reduction to  two-vertex diagram. We claim
$$
\lim\limits_{L\to 0}\lim\limits_{\epsilon\to  0} e^{-\hat I/\hbar}e^{\hbar \pa_{P_\epsilon^L}} \bracket{\fbracket{\hat I, \hat I}_\epsilon e^{\hat I/\hbar}}=\lim\limits_{L\to 0}\lim\limits_{\epsilon \to 0} \sum_{m\geq 0}  W_{\Gamma_m}(\hat I, P_\epsilon^L, K_\epsilon).
$$
Here $W_{\Gamma_m}(\hat I, P_\epsilon^L, K_\epsilon)$ is the Feynman graph integral whose diagram contains 
\begin{itemize}
\item two vertices where we put $\hat I$,
\item $m$ propagators where we put $P_\epsilon^L$(draw by solid lines),
\item and one extra propagator by $K_\epsilon$ (draw by dashed line): 
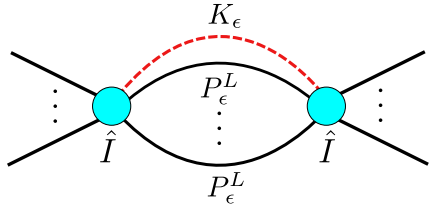
\begin{figure}[h]\centering
\begin{tikzpicture}[scale=1.5]
\clip(-1.5,-1) rectangle (3.5,1);
\draw[thick](-.5,0)node[rotate=90]{$\cdots$}(2.5,0)node[rotate=90]{$\cdots$}
    (0,0)edge[bend left=45]node[below]{$P_\epsilon^L$}(2,0)edge[bend left=-45] node[below]{$P_\epsilon^L$}(2,0)
    edge[bend left=75,red,dashed]node[above]{$K_\epsilon$}(2,0);
\draw[thick](-2,1)to(0,0)to(-2,-1) (4,1)to(2,0)to(4,-1);
\draw[fill=cyan](0,0)circle(.2)(2,0)circle(.2) (0,-.2)node[below]{$\hat{I}$}(2,-.2)node[below]{$\hat{I}$}
    (1,-.1)node[rotate=90]{$\cdots$};
\end{tikzpicture}
\caption{Two-vertex diagram contributing to the quantum master equation}\label{}
\end{figure}
\end{itemize}

This says that only diagrams with two vertices contribute to the limit $L\to 0$. 

To see this, we first observe that the limit
$$
\lim_{\epsilon\to 0} W_{\Gamma_m}(\hat I, P_\epsilon^L, K_\epsilon)
$$
exists as a chiral local functional which does not depend on $L$. This follows from the special case of Prop \ref{deformed-bracket} when $k=E-1$. Let us denote it by 
$$
O_m=\lim_{\epsilon\to 0} W_{\Gamma_m}(\hat I, P_\epsilon^L, K_\epsilon)=\lim_{L\to 0}\lim_{\epsilon\to 0} W_{\Gamma_m}(\hat I, P_\epsilon^L, K_\epsilon)
$$
and
$$
  O=\sum_{m\geq 0} O_m. 
$$

Next, let us analyze the combinatorial structure of the expression
$$
e^{\hbar \pa_{P_\epsilon^L}} \bracket{\fbracket{\hat I, \hat I}_\epsilon e^{\hat I/\hbar}}.
$$
The analogue of the Feynman graph explanation of Definition \ref{defn-stable} leads to 
$$
e^{\hbar \pa_{P_\epsilon^L}} \bracket{\fbracket{\hat I, \hat I}_\epsilon e^{\hat I/\hbar}}=\bracket{ {1\over \hbar}\sum_{(\widehat\Gamma,e)\ \text{connected}}W_{\widehat\Gamma}(\hat I, P_\epsilon^L, K_\epsilon) } \exp\bracket{{1\over \hbar}\sum_{\Gamma\ \text{connected}}  {W_\Gamma(P_{\epsilon}^{L}, \hat I)}}
$$
where the summation $(\widehat\Gamma,e)$ is over all connected graphs with a choice of a distinguished edge $e$ of $\widehat \Gamma$. And $W_{\widehat\Gamma}(\hat I, P_\epsilon^L, K_\epsilon)$ is the graph integral where we put $\hat I$ on each vertex, put $K_\epsilon$ on the distinguished edge $e$, and put $P_\epsilon^L$ on all the other edges. Here is an example for illustration. 

\begin{figure}[H]
  \centering
  
\definecolor{ffqqqq}{rgb}{1,0,0}
\begin{tikzpicture}[line cap=round,line join=round,>=triangle 45,x=1cm,y=1cm]
\clip(-4.0606504120875195,-2.8022512214056894) rectangle (4.051180397509822,2.5413787700876207);
\draw [line width=2pt,color=ffqqqq] (-2,-1) circle (0.1947446234309578cm);
\draw [line width=2pt,color=ffqqqq] (2,-1) circle (0.20578207460771017cm);
\draw [line width=2pt] (-1.431245178178118,1.5081227544164757) circle (0.21438425731630847cm);
\draw [line width=2pt] (1,1.5) circle (0.2107502542246136cm);
\draw [line width=2pt] (-1.9531435322596409,-0.8109763512515901)-- (-1.4777516532668564,1.2988436165171482);
\draw [line width=2pt] (1.964243126974814,-0.7973483080225731)-- (1.057029813037287,1.2971126665346717);
\draw [line width=2pt] (-1.5713709886494462,1.6703736928569621)-- (-2.418235999117717,2.178968078023861);
\draw [line width=2pt] (1.1812741812232965,1.6074957714403812)-- (1.8767162451272617,2.1712572122352705);
\draw [shift={(-0.26289142209462607,0.2675013803966105)},line width=2pt]  plot[domain=0.892931216571804:2.2054591256527063,variable=\t]({1*1.7536054365465363*cos(\t r)+0*1.7536054365465363*sin(\t r)},{0*1.7536054365465363*cos(\t r)+1*1.7536054365465363*sin(\t r)});
\draw [shift={(-0.17927357735647928,3.2273143136844364)},line width=2pt]  plot[domain=4.183103749139873:5.21494258768717,variable=\t]({1*2.1435871944079565*cos(\t r)+0*2.1435871944079565*sin(\t r)},{0*2.1435871944079565*cos(\t r)+1*2.1435871944079565*sin(\t r)});
\draw [shift={(0.009591351350330693,-2.254361964456088)},line width=2pt,dotted,color=ffqqqq]  plot[domain=0.6743055943069916:2.503045640653656,variable=\t]({1*2.321155259962054*cos(\t r)+0*2.321155259962054*sin(\t r)},{0*2.321155259962054*cos(\t r)+1*2.321155259962054*sin(\t r)});
\draw [shift={(-0.04963988371590323,-5.362603267524424)},line width=2pt,color=ffqqqq]  plot[domain=1.1692136079381283:1.9521433055513626,variable=\t]({1*4.719415414834463*cos(\t r)+0*4.719415414834463*sin(\t r)},{0*4.719415414834463*cos(\t r)+1*4.719415414834463*sin(\t r)});
\draw [shift={(0.020498099537760568,3.2804726310908245)},line width=2pt,color=ffqqqq]  plot[domain=4.311093885970003:5.1043886925316295,variable=\t]({1*4.780516577598397*cos(\t r)+0*4.780516577598397*sin(\t r)},{0*4.780516577598397*cos(\t r)+1*4.780516577598397*sin(\t r)});
\draw [line width=2pt] (-2.1843589169998654,-0.9372524257209712)-- (-3.050526993782147,-0.5892327400801771);
\draw [line width=2pt] (-2.1395460813159826,-1.1358394623981511)-- (-3.011972664839194,-1.6610430846942754);
\draw [line width=2pt] (2.196106413386636,-0.9376418661401497)-- (3.0102135160500856,-0.9207999689895745);
\draw [line width=2pt] (2.107095814162949,-0.8242821898081238)-- (3.0102135160500856,-0.25766551117077985);
\draw [line width=2pt] (2.146596547727933,-1.144415076858703)-- (3,-1.5);
\draw (-1.6500407623727461,1.8253529701547325) node[anchor=north west] {$\mathbf{\hat I}$};
\draw (0.7865928268218924,1.8476421043661418) node[anchor=north west] {$\mathbf{\hat I}$};
\draw (-2.218355696517043,-0.6729675453486333) node[anchor=north west] {$\mathit{\mathbf{\hat I}}$};
\draw (1.7812945135500836,-0.6729675453486333) node[anchor=north west] {$\mathbf{\hat I}$};
\draw (-0.49051955197005814,2.65704175691648) node[anchor=north west] {$\mathbf{P}_{\epsilon}^L$};
\draw (-0.4750978203928769,1.7240269964970134) node[anchor=north west] {$\mathbf{P}_\epsilon^L$};
\draw (-0.3671456993526081,0.6753492492486871) node[anchor=north west] {$\mathbf{K}_\epsilon$};
\draw (-0.39027829671838,-0.37332849799963935) node[anchor=north west] {$\mathbf{P}_\epsilon^L$};
\draw (-0.37485656514119875,-1.329475855784878) node[anchor=north west] {$\mathbf{P}_\epsilon^L$};
\draw (0.3345430874091391,0.2049864361446583) node[anchor=north west] {$\mathbf{e}$};
\draw (-3.37485656514119875, 1.029475855784878) node[anchor=north west] {$\mathbf{\widehat\Gamma:}$};
\end{tikzpicture}

    \caption{An example of $(\hat \Gamma, e)$}
\end{figure}

$\widehat\Gamma$ contains a maximal subgraph of the type $\Gamma_m$ of two-vertex graph described above. It is the read part of the graph in the above example. Then Prop \ref{deformed-bracket} can be pictured by 

\begin{figure}[H]
  \centering
  
\definecolor{ffqqqq}{rgb}{1,0,0}
\begin{tikzpicture}[line cap=round,line join=round,>=triangle 45,x=1cm,y=1cm]
\clip(-4.0606504120875195,-2.8022512214056894) rectangle (12.051180397509822,2.5413787700876207);
\draw [line width=2pt,color=ffqqqq] (-2,-1) circle (0.1947446234309578cm);
\draw [line width=2pt,color=ffqqqq] (2,-1) circle (0.20578207460771017cm);
\draw [line width=2pt] (-1.431245178178118,1.5081227544164757) circle (0.21438425731630847cm);
\draw [line width=2pt] (1,1.5) circle (0.2107502542246136cm);
\draw [line width=2pt] (-1.9531435322596409,-0.8109763512515901)-- (-1.4777516532668564,1.2988436165171482);
\draw [line width=2pt] (1.964243126974814,-0.7973483080225731)-- (1.057029813037287,1.2971126665346717);
\draw [line width=2pt] (-1.5713709886494462,1.6703736928569621)-- (-2.418235999117717,2.178968078023861);
\draw [line width=2pt] (1.1812741812232965,1.6074957714403812)-- (1.8767162451272617,2.1712572122352705);
\draw [shift={(-0.26289142209462607,0.2675013803966105)},line width=2pt]  plot[domain=0.892931216571804:2.2054591256527063,variable=\t]({1*1.7536054365465363*cos(\t r)+0*1.7536054365465363*sin(\t r)},{0*1.7536054365465363*cos(\t r)+1*1.7536054365465363*sin(\t r)});
\draw [shift={(-0.17927357735647928,3.2273143136844364)},line width=2pt]  plot[domain=4.183103749139873:5.21494258768717,variable=\t]({1*2.1435871944079565*cos(\t r)+0*2.1435871944079565*sin(\t r)},{0*2.1435871944079565*cos(\t r)+1*2.1435871944079565*sin(\t r)});
\draw [shift={(0.009591351350330693,-2.254361964456088)},line width=2pt,dotted,color=ffqqqq]  plot[domain=0.6743055943069916:2.503045640653656,variable=\t]({1*2.321155259962054*cos(\t r)+0*2.321155259962054*sin(\t r)},{0*2.321155259962054*cos(\t r)+1*2.321155259962054*sin(\t r)});
\draw [shift={(-0.04963988371590323,-5.362603267524424)},line width=2pt,color=ffqqqq]  plot[domain=1.1692136079381283:1.9521433055513626,variable=\t]({1*4.719415414834463*cos(\t r)+0*4.719415414834463*sin(\t r)},{0*4.719415414834463*cos(\t r)+1*4.719415414834463*sin(\t r)});
\draw [shift={(0.020498099537760568,3.2804726310908245)},line width=2pt,color=ffqqqq]  plot[domain=4.311093885970003:5.1043886925316295,variable=\t]({1*4.780516577598397*cos(\t r)+0*4.780516577598397*sin(\t r)},{0*4.780516577598397*cos(\t r)+1*4.780516577598397*sin(\t r)});
\draw [line width=2pt] (-2.1843589169998654,-0.9372524257209712)-- (-3.050526993782147,-0.5892327400801771);
\draw [line width=2pt] (-2.1395460813159826,-1.1358394623981511)-- (-3.011972664839194,-1.6610430846942754);
\draw [line width=2pt] (2.196106413386636,-0.9376418661401497)-- (3.0102135160500856,-0.9207999689895745);
\draw [line width=2pt] (2.107095814162949,-0.8242821898081238)-- (3.0102135160500856,-0.25766551117077985);
\draw [line width=2pt] (2.146596547727933,-1.144415076858703)-- (3,-1.5);
\draw (-1.6500407623727461,1.8253529701547325) node[anchor=north west] {$\mathbf{\hat I}$};
\draw (0.7865928268218924,1.8476421043661418) node[anchor=north west] {$\mathbf{\hat I}$};
\draw (-2.218355696517043,-0.6729675453486333) node[anchor=north west] {$\mathit{\mathbf{\hat I}}$};
\draw (1.7812945135500836,-0.6729675453486333) node[anchor=north west] {$\mathbf{\hat I}$};
\draw (-0.49051955197005814,2.65704175691648) node[anchor=north west] {$\mathbf{P}_{\epsilon}^L$};
\draw (-0.4750978203928769,1.7240269964970134) node[anchor=north west] {$\mathbf{P}_\epsilon^L$};
\draw (-0.3671456993526081,0.6753492492486871) node[anchor=north west] {$\mathbf{K}_\epsilon$};
\draw (-0.39027829671838,-0.37332849799963935) node[anchor=north west] {$\mathbf{P}_\epsilon^L$};
\draw (-0.37485656514119875,-1.329475855784878) node[anchor=north west] {$\mathbf{P}_\epsilon^L$};
\draw (-2.49051955197005803,0.6570417569164775) node[anchor=north west] {$\mathbf{P}_\epsilon^L$};
\draw (1.49051955197005803,0.6570417569164775) node[anchor=north west] {$\mathbf{P}_\epsilon^L$};
\draw (0.3345430874091391,0.2049864361446583) node[anchor=north west] {$\mathbf{e}$};
\draw (-3.37485656514119875, 1.029475855784878) node[anchor=north west] {$\lim\limits_{\epsilon\to 0}$};
\draw (3.37485656514119875, 0.829475855784878) node[anchor=north west] {$=$};
\draw (4.37485656514119875, 1.029475855784878) node[anchor=north west] {$\lim\limits_{\epsilon\to 0}$};

\draw [line width=2pt] (6.731245178178118,1.5081227544164757) circle (0.21438425731630847cm);
\draw [line width=2pt] (9,1.5) circle (0.2107502542246136cm);
\draw [line width=2pt] (6.5713709886494462,1.6703736928569621)-- (5.818235999117717,2.178968078023861);
\draw [line width=2pt] (9.1812741812232965,1.6074957714403812)-- (9.8767162451272617,2.1712572122352705);
\draw [shift={(7.86289142209462607,0.3375013803966105)},line width=2pt]  plot[domain=0.892931216571804:2.2054591256527063,variable=\t]({1*1.7536054365465363*cos(\t r)+0*1.7536054365465363*sin(\t r)},{0*1.7536054365465363*cos(\t r)+1*1.7536054365465363*sin(\t r)});
\draw [shift={(7.87927357735647928,3.1273143136844364)},line width=2pt]  plot[domain=4.183103749139873:5.21494258768717,variable=\t]({1*2.1435871944079565*cos(\t r)+0*2.1435871944079565*sin(\t r)},{0*2.1435871944079565*cos(\t r)+1*2.1435871944079565*sin(\t r)});
\draw (6.5200407623727457,1.8553529701547299) node[anchor=north west] {$\mathbf{\hat I}$};
\draw (8.7691232162849798,1.8367986412117768) node[anchor=north west] {$\mathbf{\hat I}$};
\draw (7.49051955197005803,2.6570417569164775) node[anchor=north west] {$\mathbf{P}_\epsilon^L$};
\draw (7.4750978203928768,1.724026996497011) node[anchor=north west] {$\mathbf{P}_\epsilon^L$};
\draw (6.49051955197005803,0.6570417569164775) node[anchor=north west] {$\mathbf{P}_\epsilon^L$};
\draw (8.49051955197005803,0.6570417569164775) node[anchor=north west] {$\mathbf{P}_\epsilon^L$};
\draw [line width=2pt,color=ffqqqq] (7.86666318884925174,-1.0287520900298435) circle (0.21060688165965621cm);
\draw [line width=2pt] (6.7659060602749322,1.303938010969031)-- (7.71558474303534734,-0.8798305358437478);
\draw [line width=2pt] (8.9483881167144776,1.2956672244610257)-- (8,-0.8999940219913592);
\draw [line width=2pt] (7.6761521257595115,-1.0070808206943005)-- (6.7466669097552958,-0.33477416905668783);
\draw [line width=2pt] (7.61917162039589864,-1.1739982153123658)-- (6.7312451781781146,-1.7150191452144112);
\draw [line width=2pt] (8,-0.8999940219913592)-- (9.036231874170886,-0.21140031643923815);
\draw [line width=2pt] (8.040422768012579574,-1.067101341300555)-- (9.0902079346910205,-1.0210412242412548);
\draw [line width=2pt] (8.07579569198713987,-1.2187477652869885)-- (9.1596057267883362,-1.668753950482868);
\draw (7.55859137040965486,-0.7546544716573605) node[anchor=north west] {\color{blue}$\mathbf{O}$};
\end{tikzpicture}

\end{figure}
In other words, the limit $\lim\limits_{\epsilon\to 0}W_{\widehat\Gamma}(\hat I, P_\epsilon^L, K_\epsilon) $ is simply computed by the $\lim\limits_{\epsilon\to 0}$ of the new graph integral obtained by collapsing all the edges with the same ends as the edge $e$ in $\hat \Gamma$ and put the chiral local functional $O_m=\lim\limits_{\epsilon\to 0} W_{\Gamma_m}(\hat I, P_\epsilon^L, K_\epsilon)$ on the new vertex. We conclude that
$$
\lim\limits_{\epsilon\to 0}e^{\hbar \pa_{P_\epsilon^L}} \bracket{\fbracket{\hat I, \hat I}_\epsilon e^{\hat I/\hbar}}=\lim\limits_{\epsilon\to 0}e^{\hbar \pa_{P_\epsilon^L}} \bracket{ O e^{\hat I/\hbar}}.
$$
Note that the limit on the RHS exists, again by Prop \ref{finiteness lem}, since both $\hat I$ and $O$ contain only holomorphic derivatives. Furthermore, if we take the limit $L\to 0$, all graphs containing propagators will vanish. We find
$$
\lim\limits_{L\to 0}\lim\limits_{\epsilon\to 0}e^{\hbar \pa_{P_\epsilon^L}} \bracket{\fbracket{\hat I, \hat I}_\epsilon e^{\hat I/\hbar}}=\lim\limits_{L\to 0}\lim\limits_{\epsilon\to 0}e^{\hbar \pa_{P_\epsilon^L}} \bracket{ O e^{\hat I/\hbar}}= O e^{\hat I/\hbar}. 
$$
This establishes our claim for Step1
$$
\lim\limits_{L\to 0}\lim\limits_{\epsilon\to  0} e^{-\hat I/\hbar}e^{\hbar \pa_{P_\epsilon^L}} \bracket{\fbracket{\hat I, \hat I}_\epsilon e^{\hat I/\hbar}}=O=\lim\limits_{L\to 0}\lim\limits_{\epsilon \to 0} \sum_{m\geq 0}  W_{\Gamma_m}(\hat I, P_\epsilon^L, K_\epsilon).
$$

\noindent \emph{Step 2}: It remains to show 
\begin{equation*}
\lim\limits_{L\to 0}\lim\limits_{\epsilon \to 0} \sum_{m\geq 0}  W_{\Gamma_m}(\hat I, P(\epsilon, L), K_\epsilon)={\pi \over i\hbar }\widehat{\bbracket{\oint dz I, \oint dz I}}. 
\end{equation*}

First, the vertex operator $\bbracket{\oint dz I, \oint dz I}$ is computed by the residue form of Borcherds commutator formula as described in Section \ref{defn-Lie-algebra}. This formula has a nice interpretation in terms of Wick contractions using normal ordered product and OPE's (see \cite[Chapter 4]{CFT} and \cite[Section 3.3.10]{Frenkel}).  Explicitly, assume $I=\sum\pa^{k_1}a_1\cdots \pa^{k_n}a_n$ is expressed in terms of  holomorphic derivatives of fields. Then the commutator $\bbracket{\oint dz I, \oint dz I}$ is a sum 
$$
\bbracket{\oint dz I, \oint dz I}=\sum_{m\geq 0}  J_m
$$
where each $J_m$ consists of arbitrary $m$ contractions of fundamental fields between the two copies of $\oint dz I$, replacing them by the singular part of their OPE, and them compute the residue via the contour integral as described in Section \ref{defn-Lie-algebra}. Since the singular part of the OPE of fundamental fields in the current step-up is 
$$
   a(z_1)b(z_2)\sim \bracket{i\hbar \over \pi}{\abracket{a,b}\over (z_1-z_2)}, \quad \forall a, b\in \h,
$$
$J_m$ can be pictured by 

\begin{figure}[H]
  \centering
  
  \begin{tikzpicture}[line cap=round,line join=round,>=triangle 45,x=1cm,y=1cm]
\clip(-3.57457732623514,-2.3969562463674143) rectangle (3.607082280828758,2.3339278864227784);
\draw [line width=2pt] (-1.5,0) circle (0.2645585406392744cm);
\draw [line width=2pt] (1.5,0) circle (0.24513824307698515cm);
\draw [shift={(0.016134270485672796,-0.33277498143453027)},line width=2pt]  plot[domain=0.38074956022716805:2.7596584262403843,variable=\t]({1*1.465297314194159*cos(\t r)+0*1.465297314194159*sin(\t r)},{0*1.465297314194159*cos(\t r)+1*1.465297314194159*sin(\t r)});
\draw [shift={(-0.006053111950442549,-3.086281051241854)},line width=2pt]  plot[domain=1.1869435246198425:1.945259453185303,variable=\t]({1*3.379039242896354*cos(\t r)+0*3.379039242896354*sin(\t r)},{0*3.379039242896354*cos(\t r)+1*3.379039242896354*sin(\t r)});
\draw [shift={(0.02369899747873505,0.9620159414361169)},line width=2pt]  plot[domain=3.852106609217421:5.567681866933409,variable=\t]({1*1.7821683899607785*cos(\t r)+0*1.7821683899607785*sin(\t r)},{0*1.7821683899607785*cos(\t r)+1*1.7821683899607785*sin(\t r)});
\draw [line width=1.5pt,dotted] (0,0)-- (0,-0.6083680172317281);
\draw (-0.504227540594057362,1.9283805160486907) node[anchor=north west] {${i\hbar\over \pi(z_1-z_2)}$};
\draw (-0.504227540594057362,1.0136063835899811) node[anchor=north west] {${i\hbar\over \pi(z_1-z_2)}$};
\draw (-0.504227540594057362,-0.8677815771827054) node[anchor=north west] {${i\hbar\over \pi(z_1-z_2)}$};
\draw (-1.7334022097787656,0.30160593631831527) node[anchor=north west] {$I$};
\draw (1.2725470391363177,0.28256597210004854) node[anchor=north west] {$I$};
\draw [line width=2pt] (-1.7153963671371018,0.15360867960281893)-- (-2.734896592709608,0.7706198540866254);
\draw [line width=2pt] (-1.6896017884711545,-0.18450578103062848)-- (-2.7622032832307633,-0.6629813982740391);
\draw [line width=2pt] (1.702592296021131,0.1380185487960131)-- (2.5011613147219602,0.7023531277837366);
\draw [line width=2pt] (1.657168650785304,-0.18812435628912647)-- (2.6786548031094712,-0.7175947793163501);
\draw (-3.7334022097787656,0.30160593631831527) node[anchor=north west] {$J_m=$};

\end{tikzpicture}
   
  \caption{There are $m+1$ internal edges for Wick contractions. External edges are non-contracted fields}
\end{figure}

In other words, we can write $J_m$ as a sum of terms of the form (up to some normalization factor)
$$
{\pi\over (m+1)!} \int_{\C} d^2z_2 B(\ora{z_2}) \oint_{z_2} dz_1 A(\ora{z_1})\prod_{i=0}^m \bracket{\pa_{z_1}^{k_i} {i\over \pi  (z_1-z_2)}}
$$
where $A, B$ collects all those fields which are not contracted. The factor $(m+1)!$ is the symmetry factor of the graph since it has $m+1$ identical edges. The derivatives $\pa_{z_1}^{k_i}$ come from those in $I$.

Similarly, we can explicitly write $
\lim\limits_{\epsilon\to 0} W_{\Gamma_m}(\hat I, P_\epsilon^L, K_\epsilon)
$ as a sum of terms of the form (up to some normalization factor)
$$
\lim\limits_{\epsilon\to 0}{1\over (m+1)!}\sum_{\sigma \in S_{m+1}}\int_{\C} d^2 z_1 \int_{\C} d^2 z_2  A(\ora{z_1})B(\ora{z_2})\pa_{z_1}^{k_{\sigma(0)}}K_\epsilon(\ora{z_1}, \ora{z_2}) {1\over m!}\prod_{i=1}^m \pa_{z_1}^{k_{\sigma(i)}}P_\epsilon^L(\ora{z_1},\ora{z_2})
$$
where again $A, B$ collects the external inputs. $m!$ is the symmetry factor indicating $m$ identical edges where we put $P_\epsilon^L$. The average over permutations $\sigma$ of $\{0,1,\cdots, m\}$ is due to the fact that the holomorphic derivatives are distributed symmetrically because the local functional $\hat I$ (which contains only holomorphic derivatives) is a (graded) symmetric multi-linear map.

The Theorem follows by showing that the above two expressions are the same for all $m$. After a tedious match of the normalization factor, it is sufficient to prove the following identity 
\begin{equation*}
\tag{$\dagger$} \begin{split}
\lim\limits_{\epsilon\to 0}{1\over (m+1)!}\sum_{\sigma \in S_{m+1}}\int_{\C} d^2 z_1 \int_{\C} d^2 z_2  A(\ora{z_1})B(\ora{z_2})\pa_{z_1}^{k_{\sigma(0)}}K_\epsilon(\ora{z_1}, \ora{z_2}) {1\over m!}\prod_{i=1}^m \pa_{z_1}^{k_{\sigma(i)}}P_\epsilon^L(\ora{z_1},\ora{z_2}) \\
={\pi\over (m+1)!} \int_{\C} d^2z_2 B(\ora{z_2}) \oint_{z_2} dz_1 A(\ora{z_1})\prod_{i=0}^m \bracket{\pa_{z_1}^{k_i} {i\over \pi  (z_1-z_2)}}. 
\end{split} 
\end{equation*} 
for any $k_0, \cdots, k_m\in \Z^{\geq 0}$, and smooth functions $A, B$ with compact support. Here 
\begin{align*}
 &K_\epsilon(\ora{z_1}, \ora{z_2})= {i\over {2\pi \epsilon}}e^{-|z_1-z_2|^2/2\epsilon}\\
  &P_\epsilon^L(\ora{z_1},\ora{z_2})=(-2i)\int_\epsilon^L dt \pa_{z_1}h_t(\ora{z_1},\ora{z_2})=(-2i)\int_\epsilon^L {dt\over 2\pi t} {\bar z_2-\bar z_1\over 2t}e^{-|z_1-z_2|^2/2t}
\end{align*}
where we have used the same symbols but only keep factors of the BV kernel and the regularized propagator that are relevant in this computation. $A, B$ are test functions viewed as collecting all inputs on external edges of the two vertices.  $\oint_{z_2}dz_1$ means a loop integral of $z_1$ around $z_2$ normalized by 
$$
\oint_{z_2}{dz_1\over z_1-z_2}=1.
$$
Our notation $\oint_{z_2} dz_1 A(\ora{z_1})\prod\limits_{i=0}^m \bracket{\pa_{z_1}^{k_i} {1\over z_1-z_2}}$ means only picking up the holomorphic derivative of $A$ according to the pole  condition, i.e., 
$$
    \oint_{z_2} dz_1 A(\ora{z_1}){1\over (z_1-z_2)^{m+1}}:={1\over m!}\pa_{z_2}^m A(\ora{z_2}). 
$$

The required identity ($\dagger$) follows from Prop \ref{deformed-bracket}.  In fact, Prop \ref{deformed-bracket} implies
\begin{align*}
&\lim\limits_{\epsilon\to 0}\int_{\C} d^2 z_1 \int_{\C} d^2 z_2  A(\ora{z_1})B(\ora{z_2})\pa_{z_1}^{k_0}K_\epsilon(\ora{z_1}, \ora{z_2}) {1\over m!} \prod_{i=1}^m \pa_{z_1}^{k_i}P_\epsilon^L(\ora{z_1},\ora{z_2})\\
=&{i^{m+1}(-2)^m2^m C(k_0;k_1,\cdots,k_m)\over (4\pi)^m m!}\int d^2z A(\ora{z}) \bracket{\pa_{z}^{k_0+\sum\limits_{i=1}^m(k_i+1)} B(\ora{z})}\\
=&{i^{m+1}(-2)^m2^m C(k_0;k_1,\cdots,k_m)\over (4\pi)^m m!}\int d^2z B(\ora{z}) \bracket{(-\pa_{z})^{k_0+\sum\limits_{i=1}^m(k_i+1)} A(\ora{z})}
\\
=&{i^{m+1}(-1)^{m+1}C(k_0;k_1,\cdots,k_m) \bracket{k_0+\sum\limits_{i=1}^m (k_i+1)}!\over \pi^m m!}\int d^2z_2 B(\ora{z_2}) \bracket{ \oint_{z_2} dz_1{A(\ora{z_1})\over \prod_{i=0}^m(z_2-z_1)^{k_i+1}}}
\end{align*}
where 
$$
     C(k_0;k_1,\cdots,k_m)=\int_0^1\cdots \int_0^1 \prod\limits_{i=1}^{m}{du_i} {\prod\limits_{i=1}^m u_i^{k_i}\over \left(1+\sum\limits_{i=1}^m u_i\right)^{\sum\limits_{j=0}^m (k_j+1)}}.
$$
To simplify the coefficient, we use the identity
$$
 {n!\over a^{n+1}}=\int_0^\infty du u^{n} e^{-ua}
$$
to write 
\begin{align*}
&C(k_0;k_1,\cdots,k_m)\bracket{k_0+\sum\limits_{i=1}^m (k_i+1)}!\\
=&\int_0^1\cdots \int_0^1 \prod\limits_{i=1}^{m}{du_i} \prod\limits_{i=1}^m u_i^{k_i}\int_0^\infty du_0\ u_0^{k_0+\sum\limits_{i=1}^m (k_i+1)} e^{-u_0(1+\sum\limits_{i=1}^m u_i)}\\
\substack{u_i\to u_i u_0^{-1}\\ =}&\int_{0\leq u_i\leq u_0, 1\leq  i\leq m} \prod_{i=0}^m du_i  \prod_{i=0}^m\bracket{u_i^{k_i}e^{-u_i}}.
\end{align*}
Since we eventually need to arerage over the permutations of $\{0,1,\cdots,m\}$, we find
\begin{align*}
&{1\over (m+1)!}\sum_{\sigma\in S_{m+1}}C(k_{\sigma(0)};k_{\sigma(1)},\cdots,k_{\sigma(m)})\bracket{k_{\sigma(0)}+\sum\limits_{i=1}^m (k_{\sigma(i)}+1)}!\\
=&{1\over m+1} \int_0^\infty\prod_{i=0}^m du_i  \prod\limits_{i=0}^m\bracket{u_i^{k_i}e^{-u_i}}={\prod_{i=0}^m k_i!\over m+1}.
\end{align*}
It follows that
\begin{align*}
&\lim\limits_{\epsilon\to 0}\sum_{\sigma \in S_{k+1}}{1\over (m+1)!}\int_{\C} d^2 z_1 \int_{\C} d^2 z_2  A(\ora{z_1})B(\ora{z_2})\pa_{z_1}^{k_{\sigma(0)}}K_\epsilon(\ora{z_1}, \ora{z_2}) {1\over m!} \prod_{i=1}^m \pa_{z_1}^{k_{\sigma(i)}}P_\epsilon^L(\ora{z_1},\ora{z_2})\\
=& { i^{m+1}(-1)^{m+1}\prod\limits_{i=0}^m k_i!\over \pi^{m}(m+1)!}\int d^2 z_2 B(\ora{z_2})\oint_{z_2}dz_1 {A(\ora{z_1})\over \prod\limits_{i=0}^m (z_2-z_1)^{k_i+1}}\\
=&{\pi \over (m+1)!}\int d^2 z_2 B(\ora{z_2})\oint_{z_2}dz_1 A(\ora{z_2})\prod\limits_{i=0}^m \pa_{z_1}^{k_i}{i\over \pi  (z_1-z_2)}.
\end{align*}

This proves the Theorem. 

\noindent \emph{An another approach to Step 2}: The above proof of Step 2 relies on \ref{deformed-bracket}, which hides heavy analytic estimates established in \cite{Li-modular}. Here below we  give a different but more direct computation (though still need a different computation in  \cite{Li-modular}) in order to give reader a feeling on what is happening.  

Let us change coordinates by 
$$
  (z_1, z_2)  \to  (z=z_1-z_2, z_2), \quad z=re^{i\theta}.
$$
Let us first focus on the term when $\sigma$ is the trivial permutation.
\begin{align*}
&\int_{\C} d^2 z_1 \int_{\C} d^2 z_2  A(\ora{z_1})B(\ora{z_2})\pa_{z_1}^{k_0}K_\epsilon(\ora{z_1}, \ora{z_2}){1\over m!} \prod_{i=1}^m \pa_{z_1}^{k_i}P_\epsilon^L(\ora{z_1},\ora{z_2})\\
=&-{i^{m+1}(-2)^m\over  m!}\int_{\C} d^2 z_2 B(\ora{z_2}) \int_{\C} d^2 z  {A(\ora{z_2}+\ora{z})\over (-z)^{k_0+\cdots+k_m+m}} \\
 & \qquad \bracket{\int_\epsilon^L \prod_{i=1}^m{dt_i\over 2\pi t_i}}e^{-{r^2\over 2\epsilon}-\sum\limits_{i=1}^m{r^2\over 2 t_i}} {1\over 2\pi \epsilon}\bracket{ r^2\over 2\epsilon}^{k_0} \prod_{i=1}^k \bracket{r^2\over 2t_i}^{k_i+1} \\
=&-{i^{m+1}(-2)^m\over m!}\int_{\C} d^2 z_2 B(\ora{z_2})\\
&\quad\quad \int_0^\infty {r dr}  \bracket{\int_\epsilon^L \prod_{i=1}^m{dt_i\over 2\pi t_i}}e^{-{r^2\over 2\epsilon}-\sum\limits_{i=1}^m{r^2\over 2 t_i}} {1\over 2\pi \epsilon}\bracket{ r^2\over 2\epsilon}^{k_0} \prod_{i=1}^k \bracket{r^2\over 2t_i}^{k_i+1} \int_{|z|=r} d\theta  {A(\ora{z_2}+\ora{z})\over (-z)^{k_0+\cdots+k_m+m}}. 
\end{align*}
To compute $ \int_{|z|=r} d\theta  {A(\ora{z_2}+\ora{z})\over z^{k_0+\cdots+k_m+m}}$, we need to do Taylor expansion of $A(\ora{z_2}+\ora{z})$ around $z=0$ to get the Taylor series expansion
$$
    A(\ora{z_2}+\ora{z})\sim A(\ora{z_2})+\pa_{z_2}A(\ora{z_2})z+\bar\pa_{z_2}A(\ora{z_2})\bar z+\cdots. 
$$
However, as shown by \cite[Proposition B.2]{Li-modular}, only terms involving holomorphic derivatives will survive in the $\epsilon\to 0$ limit. Therefore the $\theta$ integral can be  replaced by 
$$
\int_{|z|=r} d\theta  {A(\ora{z_2}+\ora{z})\over (-z)^{k_0+\cdots+k_m+m}}\to 2\pi \oint_{z_2}{dz_1\over z_1-z_2} {A(\ora{z_1})\over (z_2-z_1)^{k_0+\cdots+k_m+m}},
$$
where the meaning of $\oint_{z_2} dz_1$ is explained above. In particular, its value does not depend on $r$. Therefore the $r$-integral can be evaluated as
\begin{align*}
 &\lim\limits_{\epsilon\to 0}\int_0^\infty {r dr}  \bracket{\int_\epsilon^L \prod_{i=1}^m{dt_i\over 2\pi t_i}}e^{-{r^2\over 2\epsilon}-\sum\limits_{i=1}^m{r^2\over 2 t_i}} {1\over 2\pi \epsilon}\bracket{ r^2\over 2\epsilon}^{k_0} \prod_{i=1}^k \bracket{r^2\over 2t_i}^{k_i+1}\\
 \substack{r^2=2 \epsilon u_0\\ =\joinrel=\joinrel=\joinrel=\joinrel\\ t_i=\epsilon  u_0/u_i }&
{1\over (2\pi)^{m+1}} \int_0^\infty du_0 \int_0^{u_0}\prod_{i=1}^k du_i e^{-\sum\limits_{i=0}^m u_i} \prod_{i=0}^m u_i^{k_i}\\
=&{1\over (2\pi)^{m+1}}\int_{0\leq u_i\leq u_0, 1\leq  i\leq m} \prod_{i=0}^m du_i  \prod_{i=0}^m\bracket{u_i^{k_i}e^{-u_i}}.
\end{align*}
Areraging over the permutations of $\{0,1,\cdots,m\}$,  this integral leads to 
$$
{1\over (2\pi)^{m+1}}{1\over (m+1)}\int_0^\infty\prod_{i=0}^m du_i  \prod\limits_{i=0}^m\bracket{u_i^{k_i}e^{-u_i}}={1\over (2\pi)^{m+1}}{\prod\limits_{i=0}^m k_i!\over (m+1)}. 
$$
It follows by combining the above computations (we arrive that the same final computation as in the above proof of Step 2) that 
\begin{align*}
&\lim\limits_{\epsilon\to 0}\sum_{\sigma \in S_{k+1}}{1\over m!}\int_{\C} d^2 z_1 \int_{\C} d^2 z_2  A(\ora{z_1})B(\ora{z_2})\pa_{z_1}^{k_{\sigma(0)}}h_\epsilon(\ora{z_1}, \ora{z_2}) \prod_{i=1}^m \pa_{z_1}^{k_{\sigma(i)}}P_\epsilon^L(\ora{z_1},\ora{z_2})\\
=&- { i^{m+1}(-2)^m\prod\limits_{i=0}^m k_i!\over (2\pi)^{m}(m+1)!}\int d^2 z_2 B(\ora{z_2})\oint_{z_2}dz_1 {A(\ora{z_1})\over \prod\limits_{i=0}^m (z_2-z_1)^{k_i+1}}\\
=&{\pi \over (m+1)!}\int d^2 z_2 B(\ora{z_2})\oint_{z_2}dz_1 A(\ora{z_2})\prod\limits_{i=0}^m \pa_{z_1}^{k_i}{i\over \pi  (z_1-z_2)}.
\end{align*}
\end{proof}

\subsection{Generating function and modularity} In this subsection we focus on the chiral deformed theory on the flat torus. We show that in the limit $L\to 0$, the Taylor components of the effective interaction $I[\infty]$ exhibit modularity properties with prescribed anti-holomorphic dependence. In Section \ref{section-B}, this will be applied to explain why the generating function of Gromov-Witten invariants on the elliptic curve are quasi-modular forms in terms of the mirror B-model. 

\subsubsection{Generating function}
We describe the generating function when $\Sigma=E_\tau=\C/(\Z+\Z\tau)$ is an elliptic curve.

Let $I\in \V[[\h^*]][[\hbar]]$ satisfy the Maurer-Cartan equation in Theorem \ref{main-thm}. $\hat I[L]$ be the associated family solving the renormalized quantum master equation
$$
    \bracket{Q+\hbar \Delta_L} e^{\hat I[L]/\hbar}=0. 
$$

\begin{defn} We define the generating function $\hat I\bbracket{E_\tau }=\sum_{g\geq 0}\hbar^g \hat I_g\bbracket{E_\tau } \in \OO(\h[\epsilon])[[\hbar]]$ as a formal function on $\h[\epsilon]$ where 
\begin{itemize}
\item $\epsilon$ is an odd element of degree $1$, representing the generator ${d\bar z\over \Im \tau} $ of harmonics $\mathbb H^{0,1}(E_\tau)$. 
\item Under the identification $H^\bullet(\E, \dbar)\iso \h\otimes \mathbb H^{0,\bullet}(E_\tau)\iso \h[\epsilon]$, $\hat I[E_\tau]$ is the restriction of $\hat I[\infty]$ to harmonic elements $\h\otimes \mathbb H^{0,\bullet}(E_\tau)$. 
\item Given $a_1,\cdots, a_m\in \h[\epsilon]$, we denote its $m$-th Taylor coeffeicient by 
$$
  \abracket{a_1, \cdots, a_m}_g={\pa \over \pa a_1}\cdots {\pa\over \pa a_m} \hat I_g[E_\tau](0). 
$$
\end{itemize}

\end{defn}

\begin{rmk} Elements of $\h\otimes \mathbb H^{0,\bullet}(E_\tau)$ are ofter called \emph{zero modes}. Our definition of $\hat I[E_\tau]$ is just the effective theory on zero modes following physics terminology. 
\end{rmk}

The space $\h[\epsilon]$ carries naturally a $(-1)$-symplectic structure. The nontrivial pairing is between $\h$ and $\h \epsilon$ where 
$$
     \omega(a, b\epsilon )= \abracket{a,b}, \quad a, b\in \h.  
$$
Let $\Delta$ denote the associated BV operator on $\OO(\h[\epsilon])$. It is not hard to see that $\Delta$ can be identified with $\Delta_\infty$. The differential $\delta$ also induces a differential on $\h[\epsilon]$. Here we only need to keep the  constant part of the differential operator in $\delta$, since ${\pa\over \pa z}$ will annihilate the harmonics $\mathbb H^{0,\bullet}(E_\tau)$. 

\begin{prop} The triple $(\OO(\h[\epsilon]), \delta, \Delta)$ is a differential BV algebra. The generating function $\hat I[\Sigma_\tau]$ satisfies the BV master equation 
$$
   \bracket{\delta+\hbar \Delta}e^{\hat I[E_\tau]/\hbar}=0. 
$$
\end{prop}
\begin{proof} We observe that the BV kernel $K_L$ lies in $\Sym^2(\h[\epsilon])$ when $L\to \infty$, which defines the BV operator $\Delta$ as that on $\h[\epsilon]$. The proposition is just the quantum master equation 
$$
  (Q+\hbar \Delta_\infty)e^{\hat I[\infty]/\hbar}=0
$$
restricted to harmonic subspaces $\h[\epsilon]$. 

\end{proof}

\subsubsection{Modularity}

Now we analyze the dependence of $\hat I[E_\tau]$ on the complex structure $\tau$. We consider the modular group $SL(2,\Z)$, which acts on the upper half plane  $\H$ by 
$$
  \tau \to \gamma\tau:= {A\tau+B\over C\tau+D}, \quad \text{for}\ \gamma\in \begin{pmatrix}A & B\\ C &D\end{pmatrix}\in SL(2,\Z), \quad \tau\in \H. 
$$
Recall that a function $f: \H\to \C$ is said to have modular weight $k$ under the modular transformation $SL(2,\Z)$ if
$$
   f(\gamma\ora{\tau})=(C\tau+D)^k f(\ora{\tau}), \quad \text{for}\ \gamma\in \begin{pmatrix}A & B\\ C &D\end{pmatrix}\in SL(2,\Z).
$$
\begin{defn} We extend the $SL(2,\Z)$ action to $\C^n\times \H$ by
$$
 \gamma: (z_1,\cdots, z_n, \tau)\to  (\gamma z_1,\cdots, \gamma z_n, \gamma \tau):=\bracket{{z_1\over C\tau+D}, \cdots, {z_n\over C\tau +D}, {A\tau+B\over C\tau+D}}, 
$$
for $\gamma\in \begin{pmatrix}A & B\\ C &D\end{pmatrix}\in SL(2,\Z)$.
\end{defn}

It is easy to see that this defines a group action, in other words, 
$$
   \gamma_1( \gamma_2 (z_1,\cdots, z_n,\tau))=(\gamma_1\gamma_2)(z_1,\cdots,z_n,\tau), \quad \gamma_1, \gamma_2\in SL(2,\Z). 
$$

\begin{defn}\label{defn-modular} A differential form $\Omega$ on $\C^n\times \H$ is said to have modular weight $k$ under the above $SL(2,\Z)$ action if 
$$
    \gamma^* \Omega= (C\tau+D)^k  \Omega. 
$$ 
\end{defn}
When $n=0$ and $\Omega$ being a $0$-form on $\H$, this reduces to the above modular function of weight $k$. 

Let $h_L$ be the heat kernel function on $E_\tau$ as before. Pulled back to the universal cover $\C$ of $E_\tau$, $h_L$ gives rise to a function $\tilde h_L$ on $\C\times \C  \times \H$ by
$$
  \tilde h_L(\ora{z_1}, \ora{z_2};\ora{\tau})={1\over 2\pi L}\sum_{\lambda\in \Lambda_\tau}e^{-|z_1-z_2+\lambda|^2/2L}, \quad \Lambda_\tau=\Z\oplus \Z \tau. 
$$
$SL(2,\Z)$ transforms the lattice $\Lambda$ by 
$$
   \Lambda_{\gamma \tau}={1\over  C\tau+D}\Lambda_\tau, \quad  \gamma\in \begin{pmatrix}A & B\\ C &D\end{pmatrix}\in SL(2,\Z).
$$
It follows  that the heat kernel $\tilde h_L$ transforms under $SL(2,\Z)$ as
$$
  \tilde h_L(\gamma \ora{z_1}, \gamma \ora{z_2};\gamma \tau)=\abs{C\tau+D}^2 h_{\abs{C\tau+D}^2L}(\ora{z_1}, \ora{z_2};\ora{\tau}), \quad \forall \gamma \in SL(2,\Z). 
$$

The regularized BV kernel $K_L$  and propagator $P_\epsilon^L$ are 
$$
  K_L(\ora{z_1}, \ora{z_2};\ora{\tau})=i\ h_L(\ora{z_1},\ora{z_2};\ora{\tau})\bracket{d\bar z_1\otimes 1-1\otimes d\bar z_2}   C_{\h}
$$
and
$$
   P_\epsilon^L(\ora{z_1}, \ora{z_2};\ora{\tau})=-2i \int_{\epsilon}^L du \pa_{z_1}h_{u}(\ora{z_1}, \ora{z_2};\ora{\tau}) C_{\h}. 
$$
We define similarly $\tilde K_L, \tilde P_\epsilon^L$ as $\tilde h_L$ on the universal cover $\C$ of $E_\tau$. The following lemma is straight-forward. 
\begin{lem}\label{lem-modular} We have  the  following modular transformation properties of the  regularized propagators  and effective BV kernel
$$
   \gamma^*\tilde P_\epsilon^L=(C\tau+D)\tilde P_{\abs{C\tau+D}^2  \epsilon}^{\abs{C\tau+D}^2 L}, \quad \gamma^* \tilde K_L=(C\tau+D)\tilde K_{|C\tau+D|^2L}. 
$$
In particular, $\tilde P_0^\infty,  \tilde K_\infty$ have weight $1$ in the sense of Definition \ref{defn-modular}.  More generally, the $k$-th holomorphic derivative $\pa_{z_1}^k \tilde P_0^\infty(\ora{z_1},\ora{z_2};\ora{\tau})$ has modular weight $k+1$. 
\end{lem}

\begin{rmk}
$ K_\infty(\ora{z_1}, \ora{z_2};\ora{\tau})={d\bar z_1\otimes 1-1\otimes d\bar z_2\over 2\Im \tau} C_\h$. 
\end{rmk} 

\begin{defn}\label{defn-modular-quantization} A quantization $\hat I[L]$ defined by $I=\sum\limits_{g\geq 0}I_g\hbar^g \in \oint \V[[\h^*]][[\hbar]]$ in Theorem \ref{thm-defn} is called \emph{modular invariant} if each $I_g$ contains exactly $g$ holomorphic derivatives. 

\end{defn}

\begin{rmk}
This definition may vary according to the conformal weight of the pairing $\abracket{-,-}$ on $\h^*$. See Remark \ref{rmk-modular} for a specific example. 
\end{rmk}

\begin{thm}\label{thm-modularity} Let $\hat I[E_\tau]=\sum\limits_{g\geq 0}\hat I_g[E_\tau] \hbar^g$ be the generating function of a modular invariant quantization, $\hat I_g[E_\tau]\in \OO(\h[\epsilon])$. Then for any $a_1, \cdots, a_k\in \h, b_1,\cdots,b_m \in \h\epsilon$, the Taylor coefficient of $\hat I_g[E_\tau]$
$$
\abracket{a_1,\cdots,a_k,b_1,\cdots,b_m}_g
$$
is modular of weight $m+g-1$ as a function on $\H$. Moreover, It has the following expansion 
$$
\abracket{a_1,\cdots,a_k,b_1,\cdots,b_m}_g=\sum_{i=0}^N {f_i(\tau)\over (\Im \tau)^i}
$$
where $f_i(\tau)$'s are holomorphic functions in $\tau$ and $N<\infty$ is an integer. 
\end{thm}
\begin{proof} Let us write 
$$
   \hat I[\infty]=\sum_{\Gamma\ \text{connected}}{W(\Gamma, I)}.
$$
Let $\Gamma$ be a Feynman graph which contributes to $\abracket{a_1,\cdots,a_k,b_1,\cdots,b_m}_g$. Due to the type of differential forms used (the propagator only contains $0$-forms on $E_\tau$), $\Gamma$ contains $m$ vertices $\{v_j\}_{j=1}^m$ each of which has an external input given by $b_j, j=1,\cdots m$. 
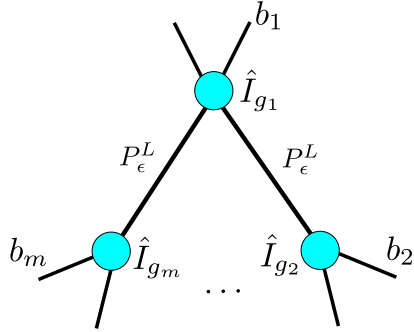
\begin{figure}[h]\centering
\begin{tikzpicture}[scale=1.5]
\draw[thick](-90:1)node[above]{$\cdots$};
\draw(210:1)to node[left]{$P_\epsilon^L$} (90:1)to node[right]{$P_\epsilon^L$}(-30:1);
\draw(90:1) to (210:1) (90+120*1-10:2)node{$b_m$}(90+120*2+10:2)node{$b_2$}(90-10:2)node{$b_1$};
\draw (220:1) node[right]{$\hat{I}_{g_2}$} (-40:1) node[left]{$\hat{I}_{g_2}$}(80:1) node[right]{$\hat{I}_{g_1}$};
\foreach \j in {1,2,3}{
\draw[thick](90+120*\j+10:1.8) to (90+120*\j:1)to (90+120*\j-10:1.8);
\draw[thick,fill=cyan] (90+120*\j:1)circle(.2); }
\end{tikzpicture}
\caption{}\label{}
\end{figure}

Assume the vertex $v_j$ has genus $g_j$, i.e., given by the local functional $I_{g_j}$. Let $E$ be the set of propagators in $\Gamma$. Then 
$$
   m+g=\sum_{j=1}^m g_j+|E|+1. 
$$
The graph integral $W(\Gamma, I)$ can be written as
$$
  W(\Gamma, I)=\lim_{\substack{\epsilon\to 0\\ L\to \infty}} A \prod_{i=1}^m \int_{E_\tau}{d^2z_i\over \Im \tau} {\prod_{e\in E}\pa_{z_{h(e)}}^{n_e}P_\epsilon^L(\ora{z_{h(e)}},\ora{z_{t(e)}};\ora{\tau})} 
$$
where $A$ is a combinatorial coefficient not depending on $\tau$. 
$$
h,t: E\to \{1,\cdots, m\}
$$ 
denote the head and the tale of the edge, where we have chosen an arbitrary orientation on edges. $n_e$ denotes the number of holomorphic derivatives applied to propagator at the edge $e$. Since all the external inputs $a_i, b_j$'s are harmonics, all holomorphic derivatives in the vertex $I_{g_j}$ will go to the propagators. By the modular invariance of the quantization, $I_{g_j}$ contains exactly $g_j$ holomorphic derivatives, hence 
$$
  \sum_{e\in E} n_e= \sum_{j=1}^m g_j. 
$$
It follows from the modular property of the measure ${d^2z\over \Im \tau}$ and the propagator $\hat P_0^\infty$ that $W(\Gamma, I)$ is a modular function on $\H$ of weight 
$$
  \sum_{e\in E}(n_e+1)=|E|+\sum_{j=1}^m g_j=m+g-1. 
$$
The polynomial dependence on ${1\over \Im \tau}$ follows from \cite[Proposition 5.1]{Li-modular}.
\end{proof}

Given a function $f$ on $\H$ of the form $
f=\sum\limits_{i=0}^N {f_i(\tau)\over (\Im \tau)^i}
$, 
we denote
$$
\lim_{\bar \tau\to \infty}f:= f_0(\tau). 
$$

It is shown in \cite{almost-modular-form} that $f$ is determined by the leading term $f_0$ and modular property. In particular, the operation $\lim\limits_{\bar\tau\to \infty}$ identifies the space of almost holomorphic modular forms with the space of quasi-modular forms \cite{almost-modular-form}. In general, the $\bar \tau\to \infty$ limit of the generating function $\hat I[E_\tau]$ will be reduced to certain characters on vertex algebras. This is argued in \cite{Dijkgraaf-chiral} by the method of contact terms, and realized in \cite{L-elliptic} by a method of cohomological localization via the quantum master equation. This pheonomenon will be systematically studied in \cite{Si-Jie}.

\subsection{Example: Poisson $\sigma$-model} We illustrate the application of Theorem \ref{main-thm} by the example of the AKSZ formalism of Poisson $\sigma$-model as described in \cite{Poisson-sigma-model}. 

Let $V=\R^{n}$ and $P$ be a Poisson bi-vector field on $V$.  Let $\Sigma$ be a flat surface as before. We consider the BV formalism of Poisson sigma model in the formal neighborhood of constant maps from $\Sigma$ to the origin of $V$. The space of fields is
$$
  \E= \Omega^\bullet(\Sigma)\otimes (V\oplus V^*[1]). 
$$
The differential on $\E$ is the de Rham differential $d_{\Sigma}$ on $\Omega^\bullet(\Sigma)$. The $(-1)$-shifted symplectic pairing is 
$$
  \omega(\varphi, \eta):=\int_{\Sigma} (\varphi, \eta), \quad \text{where}\ \varphi \in \Omega^\bullet(\Sigma)\otimes V, \eta\in \Omega^{\bullet}(\Sigma)\otimes V^*[1]. 
$$

Let us choose linear coordinates $x^i$ on $V$ and 
$
P=\sum\limits_{i,j}P^{ij}(x)\pa_{x^i}\wedge \pa_{x^j}.
$
The above fields $\varphi, \eta$ in coordinate components are  
$$
    \varphi=\{\varphi^i\}_{i=1}^n, \quad \eta=\{\eta_i\}_{i=1}^n, \quad \text{where}\ \varphi^i\in \Omega^\bullet(\Sigma), \eta_i \in \Omega^\bullet(\Sigma)[1]. 
$$
The action functional is given by \cite{Poisson-sigma-model}
$$
  S=\sum_{i}\int_{\Sigma} \eta_i d_{\Sigma}\varphi^i +\sum_{i,j} \int_{\Sigma} P^{ij}(\varphi) \eta_i \eta_j. 
$$

Let us split the differential $d_{\Sigma}=Q+\delta$, where $Q=\dbar$ and $\delta=\pa$ are the $(0,1)$-differential and $(1,0)$-differential on $\Sigma$. Then the above theory falls into the setting of Section \ref{section-2d}, and we can apply Theorem \ref{main-thm} to study its quantization via chiral deformations. In terms of notations in Section \ref{section-2d}, we have
$$
   \h=V[dz]\oplus V^*[dz][1], 
$$
where $V[dz]=V\otimes \C[dz], V^*[dz]=V^*\otimes \C[dz]$. The relevant vertex algebra is generated by $\varphi^i, \eta_i$: $\varphi^i$ represent components in $V[dz]$ and $\eta_i$ represent components in $V^*[dz][1]$. The nontrivial OPEs are given by 
$$
     \varphi^i(z) \eta_j(w)\sim \delta_{i,j}\bracket{i\hbar \over \pi}{dz-dw\over z-w}.
$$ 
Here it is understood that we have to match the corresponding components in $dz, dw$ in the above formula. For example, if we write $\varphi^i(z)=\varphi^i_0(z)+ \varphi^i_1(z) dz$ and $\eta_i(z)=\eta_{i0}(z)+\eta_{i1}(z)dz$, then matching the $dw$ component we find
$$
   \varphi^i_0(z) \eta_{j1}(w)\sim \delta_{i,j}\bracket{i\hbar \over \pi}{-1\over z-w}.
$$

The classical interaction is represented by the vertex operator 
$$
\oint I, \quad \text{where}\ I=\sum_{i,j} P^{ij}(\varphi) \eta_i \eta_j.
$$
Here it is understood (due to the type of differential forms used) that only the $dz$-component of $I$ contributes to $\oint I$ when we expand the fields $\varphi^{i}, \eta_i$ into forms in $z$. 

The classical master equation is satisfied because it satisfies two independent equations 
$$
  \delta \oint I=0, \quad \fbracket{\oint I, \oint I}=0.
$$
The first term vanishes since it produces a total derivative while the vanishing of the second term follows from Jacobi identity for $P$ \cite{Poisson-sigma-model}. 
\begin{prop} The classical interaction $\oint I$ satisfies the quantum master equation of Theorem \ref{main-thm}. 
\end{prop}
\begin{proof}[Sketch of proof] We only need to prove $\bbracket{\oint I, \oint I}=0$ since $\delta \oint I=0$. By definition, $\bbracket{\oint I, \oint I}$ is computed by Wick contractions and OPE's (see \cite{Kac}). A single contraction gives rise to the classical bracket $\fbracket{\oint I, \oint I}$. If we have two or more contractions between the fields, the form of $I$ implies that each contraction contributes a factor of $dz-dw$. However, the product of two copies of $dz-dw$ vanishes by the type of the differential forms used. This implies
$
\bbracket{\oint I, \oint I}=\fbracket{\oint I, \oint I}=0. 
$
\end{proof}

This proposition says that no quantum correction is needed at all! This remarkable fact is due to cancellations between bosons and fermions, which is just an incidence of supersymmetry. This result gives a natural interpretation of Kontsevich's graph formula  of star product \cite{Kontsevich-DQ}, which is argued in \cite{CF} by the requirement of BV quantum master equation.  We remark that the tadpole diagrams (i.e. with edges that start and end at the same vertex) that appeared in \cite{CF} vanish by our renormalization scheme on flat spaces: the regularized BV kernel $K_\epsilon$ and propagator $P_\epsilon^L$  on the tadpole are zero before we take the limit $\epsilon\to 0$ (see the proof of Theorem \ref{main-thm}). In particular,  our construction here should lead to a rigorous formulation of \cite{CF} by gluing the above linear case to Poisson manifolds.  

In general, when the surface $\Sigma$ is a compact Riemann surface which is no longer flat,  further obstructions may exist for quantization. One such example is the topological B-model from a genus $g$ surface $\Sigma$ to a complex manifold $X$. The perturbative BV quantization in the current sense is analyzed in \cite{B-model}. It is found that the tadpole diagram gives rise to the obstruction class (anomaly) by $(2g-2)c_1(X)$, requiring for a Calabi-Yau geometry to be quantizable.  It would be very interesting to extend the results in this paper systematically to arbitrary surfaces and to the string-theoretical formulation of coupling with 2d gravity. 
 
\section{Application: quantum B-model on elliptic curves}\label{section-B}
In this section, we apply our theory to solve the higher genus $B$-model on elliptic curves. Part of the results are presented in \cite{L-elliptic} based mainly on symmetry argument. Using the technique we have developed in this paper, we present a self-contained description of quantum B-model on elliptic curves in terms of chiral vertex algebra and give stronger results on the exact solution of the full system. Combining with the A-model results in \cite{virasoro}, it leads to a manifest interpretation of higher genus mirror symmetry on elliptic curves in terms of boson-fermion correspondence.   

\subsection{BCOV theory on elliptic curves}\label{sec:BCOV-elliptic} We consider topological B-model on Calabi-Yau geometry, which concerns with the geometry of complex structures. In \cite{BCOV}, Bershadsky, Cecotti, Ooguri, and Vafa proposed \emph{Kodaira-Spencer gauge theory} on Calabi-Yau 3-folds as the leading approximation of B-twisted closed string field theory. This is fully generalized in \cite{Si-Kevin} to arbitrary Calabi-Yau manifolds,  giving rise to a complete description of B-twisted closed string field theory in the sense of Zwiebach \cite{Zwiebach}. We shall call this BCOV theory. An earlier related work on the finite dimensional
toy model of BCOV theory appeared in \cite{Losev} in the absence of the issue of renormalization.
 
 In this section, we describe BCOV theory on elliptic curves \cite{L-elliptic} and study its quantum geometry using the tools we have developed in this paper. 

Let $\Etau=\C/(\Z\oplus Z\tau)$ as before. Let $z$ denote the linear holomorphic coordinate. The space of fields of BCOV theory  \cite{Si-Kevin} specializing to $\Etau$ is given by
$$
   \E=\Omega^{0,\bullet}(\Etau, \OO_{\Etau})[[t]]\oplus \Omega^{0,\bullet}(\Etau, T_{\Etau}[1])[[t]]. 
$$
Here $T_{\Etau}[1]$ is the holomorphic tangent bundle sitting at degree $-1$. Let $t$ be a formal variable of cohomology degree $0$ that represents ``gravitational descendants". Note that we have used a different grading from \cite{Si-Kevin} \footnote{In \cite{Si-Kevin}, the grading is $\E=\Omega^{0,\bullet}(\Etau, \OO_{\Etau})[[t]]\oplus \Omega^{0,\bullet}(\Etau, T_{\Etau}[-1])[[t]]$ and $t$ has cohomology degree $2$} so our interaction will have degree $0$. The differential is given by 
$$
  Q=\dbar+ \delta, \quad \delta=t\pa,
$$
where $\pa: T_{\Etau}\to \OO_\Etau$ is the divergence operator with respect to the holomorphic volume form $dz$. In terms of the set-up in section \ref{section-2d}, we have
\begin{itemize}
\item $
   \h= \C[[t,\theta]], \deg(t)=0, \deg(\theta)=-1
$. Here $\theta$ represents the global vector field $\pa_z$.  Then $\E\iso \Omega^{0,\bullet}(\Etau)\otimes \h$. 
\item $\delta={\pa\over \pa z}\otimes t{\pa \over \pa \theta}$.
\item However, $\E$ is $(-1)$-shifted Poisson instead of symplectic. Let 
$$
    h_L(\ora{z_1},\ora{z_2})=\sum_{\lambda\in \Lambda}{1\over 2\pi L}e^{-|z_1-z_2+\lambda|^2/2L}, \quad \Lambda=\Z+\Z\tau,
$$
be the heat kernel function on $E_\tau$. The regularized BV kernel is given by 
$$
   K_L=i \pa_{z_1}h_L(\ora{z_1}, \ora{z_2})(d\bar z_1\otimes 1-1\otimes d\bar z_2)C_\h
$$
where $C_\h:=t^0\otimes t^{0}\in \h\otimes \h$.  The regularized propagator is 
$$
   P_\epsilon^L=-2i\int_\epsilon^L du\ \pa_{z_1}^2 h_u(\ora{z_1},\ora{z_2})C_\h. 
$$
\end{itemize}

The situation differs a bit from our set-up in Section \ref{section-2d}. There is one more holomorphic derivative for our BV kernel and the factor $C_\h$ is highly degenerate. Nevertheless, techniques in section \ref{section-VA} can be applied to the Poisson case without much change (see also Remark \ref{rmk-Poisson}). In particular, 
 \begin{itemize}
 \item $K_L$ defines a regularized BV operator $\Delta_L$ such that $(\OO(\E), Q, \Delta_L)$ is a differential BV algebra. 

\item $K_0$ well-defines a BV-bracket on local functionals as in  Definition \ref{defn-BV-bracket}
 $$
    \{-,-\}: \Ol(\E)\otimes \Ol(\E)\to \Ol(\E). 
 $$
\end{itemize}
Introduce
$$
   \abracket{-}_0: \Sym^\bullet(\C[[t]])\to \C, \quad \abracket{t^{k_1}\otimes \cdots\otimes t^{k_n}}_0=\binom{n-3}{k_1,\cdots, k_n}.
$$
These coefficients represents intersection numbers of $\psi$-classes on moduli space of stable rational curves. We extend it $\Omega^{0,\bullet}[\theta]$-linearly to 
$$
    \abracket{-}_0: \Sym^\bullet(\E)\to \Omega^{0,\bullet}[\theta]. 
$$
\begin{defn}[\cite{Si-Kevin}]\label{defn-BCOV} We define the classical BCOV interaction $I^{BCOV}\in \Ol(\E)$ by the local functional
$$
  I^{BCOV}(\varphi)=i\int_{\Etau} {dz}\int d\theta \abracket{e^{\varphi}}_0, \quad \varphi \in \E.
$$
Here $\abracket{e^{\varphi}}_0$ is understood as
$$
\abracket{e^\varphi}_0=\sum_{k\geq 3} \abracket{\varphi^{\otimes{k}}\over k!}_0.
$$
The fermionic integral $\int d\theta$ means taking the coefficient of the term with $\theta$
$$
  \int d\theta (a+\theta b):=b. 
$$
\end{defn}

\begin{prop} $I^{BCOV}$ satisfies the classical master equation
$$
  QI^{BCOV}+{1\over 2}\fbracket{I^{BCOV}, I^{BCOV}}=0. 
$$
\end{prop}

This is proved in \cite{Si-Kevin} for any Calabi-Yau. See \cite[Theorem 3.5]{HLTY} for another  proof and generalization. To be self-contained, we remark that this proposition also follows from Corollary \ref{cor-classical-BCOV}.

\begin{rmk}By Remark \ref{rmk-CME}, $Q+\{I^{BCOV},-\}$ defines a $L_\infty$-structure on $\E$. It is shown in \cite{Si-Kevin} (see \cite{Si-String-Math} for a geometric presentation) that this $L_\infty$-structure is quasi-isomorphic to the standard dg Lie algebra structure with differential $Q$ and Schouten-Nijenhuis bracket. 
\end{rmk}

Since the Poisson kernel is degenerate and contains one more holomorphic derivative, the application of Theorem \ref{main-thm} to our situation requires slight modification. Let us parametrize
$$
 \varphi=\sum_{k\geq 0}b_kt^k+\eta_k \theta t^k, \quad \varphi \in \h. 
$$
We introduce mutually local fields $b_k(z), \eta_k(z)$ with OPE relations
\begin{align*}
     b_0(z)b_0(w)\sim {i\hbar \over \pi}{ 1\over (z-w)^2},&\quad b_k(z)b_m(w)\sim 0, k+m>0.\\
     b_\bullet(z)\eta_\bullet(w)\sim 0,& \quad \eta_\bullet(z)\eta_\bullet(w)\sim 0. 
\end{align*}
The OPEs exactly respect the structure of  our Poisson kernel: ${1\over (z-w)^2}$ represents the derivative of the delta-function, and only $b_0$ appears in the propagator. The associated vertex algebra $\V[[\h^*]]$ is the tensor product of a Heisenberg vertex algebra (generated by the field $b_0(z)$) with several copies of commutative vertex algebra (generated by the fields $(b_{>0}(z), \eta_{\bullet}(z))$). Equivalently, if we collect the fields with parameter $t$
$$
   b(z,t)=\sum_{k\geq 0}b_k(z)t^k, \quad \eta(z,t)=\sum_{k\geq 0}\eta_k(z)t^k, 
$$
then the OPE's can be simply written as 
$$
   b(z_1,t_1)b(z_2,t_2)\sim   {i\hbar \over \pi} {1 \over (z_1-z_2)^2}, \quad b(z_1,t_1)\eta(z_2,t_2)\sim  0, \quad \eta(z_1,t_1)\eta(z_2,t_2)\sim 0. 
$$
The differential $\delta$ is then dually expressed as
$$
     \delta b(z,t)=t\pa_z \eta(z,t), \quad \delta \eta(z,t)=0. 
$$
In terms of components, 
$$
  \delta b_{k+1}=\pa_z \eta_k, \quad \delta \eta_k=0. 
$$
We adapt notations in section \ref{section-2d}. The classical BCOV interaction can be expressed as $I^{BCOV}=\hat I_0$, where $I_0\in \V[[\h^*]]$ is defined similarly to Definition \ref{defn-BCOV}
$$
I_0(b,\eta)= \abracket{e^{b}\otimes \eta}_0. 
$$
Classical master equation implies that it satisfies the following Maurer-Cartan equation modulo $\hbar$
$$
 \delta \oint dz I_0+{1\over 2}{\pi \over i\hbar} \bbracket{\oint dz I_0, \oint dz I_0}=O(\hbar).
$$

Our goal is to find $I=\sum\limits_{g\geq 0}\hbar^g I_g\in \V[[\h^*]][[\hbar]]$ as a quantum correction of the classical BCOV interaction $I_0$ satisfying the exact Maurer-Cartan equation 
$$
  \delta \oint dz I+{1\over 2}{\pi \over i\hbar} \bbracket{\oint dz I, \oint dz I}=0. 
$$
This will give a solution of the effective BV quantum master equation for our BCOV theory by the next theorem, which is the analogue of Theorem \ref{main-thm} but for chiral bosons. 

\begin{thm}\label{thm-chiral-boson} Let $I=\sum\limits_{g\geq 0}\hbar^g I_g\in \V[[\h^*]][[\hbar]]$ of degree $1$ and $\hat I$ be the associated chiral local functional for BCOV theory on elliptic curves.  Let $W(-,-)$ be the homotopy RG flow operator defined in Definition \ref{defn-stable}. Let   $P_\epsilon^L=-2i\int_\epsilon^L du\ \pa_{z_1}^2 h_u(\ora{z_1}\ora{z_2})C_\h
$ be the regularized propagator and $K_L=i \pa_{z_1}h_L(\ora{z_1}, \ora{z_2})(d\bar z_1\otimes 1-1\otimes d\bar z_2)C_\h
$ be the regularized BV kernel for BCOV theory on elliptic curves. Let $\oint dz I=I_{(0)}$ be the corresponding mode of the vertex operator $I$ as an element of the modes Lie algebra $\oint \V[[\h]][[\hbar]]$.

 Then the limit
$$
   {\hat I[L]}:= \lim_{\epsilon \to 0} W(P_\epsilon^L, \hat I)
$$
exists. The  family $\{\hat I[L]\}_{L>0}$ satisfies the renormalized quantum master equation 
$$
   \bracket{Q +\hbar \Delta_L} \hat I[L]+{1\over 2}\fbracket{\hat I[L],\hat I[L]}_L=0
$$
if and only if 
$$
  \delta \oint dz I+{1\over 2}\bracket{i\hbar \over \pi }^{-1} \bbracket{\oint dz I, \oint dz I}=0, 
$$
i.e.. $\oint dz I$ is a Maurer-Cartan element of $\oint \V[[\h]][[\hbar]]$.  Here the BV operator $\Delta_L$ is associated to the BV kernel $K_L$ as defined in Definition \ref{defn-BV-kernel}. 
\end{thm}
\begin{proof} The existence of the limit ${\hat I[L]}= \lim\limits_{\epsilon \to 0} W(P_\epsilon^L, \hat I)$ follows by Remark \ref{rmk:chiral-boson}. For the statement on quantum master equaton, the proof of Theorem \ref{main-thm} carries over here word by word. The only difference is that our propagator and BV kernel are the holomorphic derivative of those in Theorem \ref{main-thm}. This implies that at the end of the two vertex computation in Step 2 there, each contraction leads to a factor (up to some normalization)
$$
{\pa_{z_2}}{1\over (z_1-z_2)}={1\over (z_1-z_2)^2}
$$
which is precisely the singular part of the OPE in the BCOV theory. 
\end{proof}

In the rest of this section, I will explain how to find a canonical solution of the Maurer-Cartan equation. It will be related to Gromov-Witten theory on elliptic curves by mirror symmetry.

\subsection{Hodge weight and dilaton equation}\label{section-dilaton} The first simplification we will make is to use rescaling symmetries of quantum master equation to restrict the possible form of the action functional. 

We assigne the following gradings in $\V[[\h^*]][[\hbar]]$ and $\oint \V[[\h^*]][[\hbar]]$.
\begin{center}
\begin{tabular}{|c|c|c|c|c|c|c|c|}
\hline
 & $\pa_z^mb_k$ & $\pa_z^m\eta_k$ & $\hbar$ & $\delta$ & $z$ & $\oint dz$ & $[-,-]$\\ \hline
cohomology degree (deg) & 0 & 1 & 0 & 1 & 0& 0 & 0\\ \hline
conformal weight (cw) & -k+m+1 & -k+m & 0 & 1 & -1 &-1& 0\\ \hline
dilaton dimension (dim)& m & m & -2 & 1 & 0& 0 & -1\\ \hline
Hodge weight (hw=cw-dim) & -k+1 & -k & 2 & 0 & 0 & -1 & 1\\
\hline
\end{tabular}
\end{center}
They are all compatible with quantum master equation. The classical BCOV interaction has 
$$
 \text{deg}(I_0)=1,\quad \text{cw}(I_0)=2, \quad \text{dim}(I_0)=0.
$$
\begin{defn}
A solution $I\in \V[[\h^*]][[\hbar]]$ of quantum master equation for which
$$
 \text{deg}(I)=1,\quad \text{cw}(I)=2, \quad \text{dim}(I)=0,
$$
will be called an \emph{equivariant quantization}. 
\end{defn}

\begin{rmk}
For an equivariant quantization, 
$$
    hw(I_g)=\text{cw}(I_g)-\text{dim}(I_g)=2-2g. 
$$ 
This is exactly the Hodge weight condition for elliptic curves described in \cite{Si-Kevin}. 
The condition $\text{dim}(I)=0$ is essentially equivalent to the dilaton equation \cite{L-elliptic}. 
Therefore the equivariance condition is naturally viewed as imposing the Hodge weight condition and the dilaton equation. 
\end{rmk}
\begin{rmk}\label{rmk-modular} The dilaton dimension condition implies that the genus $g$ correction $I_g$ contains exactly $2g$ holomorphic derivatives. This is the modification of modular invariant quantization (Definition \ref{defn-modular-quantization}). It changes the number of holomorphic derivatives from $g$ to $2g$ exactly because there exists an extra holomorphic derivative on our propagator. In particular, an analogue of Theorem \ref{thm-modularity} holds in this situation. See also \cite{L-elliptic}. 
\end{rmk}

\begin{defn}Let us denote the homogeneous component
$$
   \V[[\h^*]]^d_{w}:=\{a\in \V[[\h^*]]| \text{deg}(a)=d, \text{cw}(a)=w\}.
$$
\end{defn}

We will focus on equivariant quantizations, which are given by elements  $I\in \V[[\h^*]]^1_{2}$ satisfying the Maurer-Cartan equation. In this case, we can use the rescaling symmetry to set $\hbar=\pi/i$. The power of $\hbar$ can be recovered from counting the number of derivatives. 

From now on, we will work with the normalized OPE
\begin{align*}
     b_0(z)b_0(w)\sim {1 \over (z-w)^2},&\quad b_k(z)b_m(w)\sim 0, k+m>0.\\
     b_\bullet(z)\eta_\bullet(w)\sim 0,& \quad \eta_\bullet(z)\eta_\bullet(w)\sim 0. 
\end{align*}

\subsection{Reduction to Fedosov's equation} In this section, we use boson-fermion correspondence to further reduce quantum master equation to a deformation quantization problem of Fedosov's equation \cite{Fedosov}. 

\begin{defn}\label{defn-background}
Let us package the fields $\{b_{>0}, \eta_{\bullet}\}$ into new series denoted by
$$
   \tilde b(z)=\sum_{k\geq 1}{t^k\over k!} b_k(z), \quad \tilde \eta(z)=\sum_{k\geq 1}{t^k\over k!} \eta_{k-1}(z). 
$$
These fields do not appear in the propagator, and will be called the \emph{background fields}. $b_0$ does appear in the propagator, and will be called the \emph{dynamical field}. 
\end{defn}

The deformation quantization problem arises from viewing the linear coordinate $z$ and the descendant variable $t$ as a Darboux system of a holomorphic symplectic structure on $\C^2$:
$$
   \omega=dz\wedge dt.  
$$

 We consider the following differential ring freely generated by two generators $\tilde b, \tilde\eta$ and two derivatives $\pa_z, \pa_t$:
$$
    \B=\C[[\pa_z^\bullet \pa_t^\bullet \tilde b, \pa_z^\bullet \pa_t^\bullet \tilde \eta]]. 
$$
Here $\tilde b$ has even parity of cohomology degree $0$, $\tilde \eta$ has odd parity of cohomology degree $1$. Here we have confused ourselves to use the same symbols $\tilde b, \tilde \eta$ for our generators in $\B$. Later on, we will plug into expressions in terms of the background fields in Definition \ref{defn-background} when the meaning is clear from the context.

The symplectic form induces a Poisson structure on $\B$ by
$$
   \{F, G\}=\pa_t F\pa_z G-\pa_z F \pa_t G, \quad F, G\in \B.
$$ 
We also introduce a differential $\delta$ as an analogue of that in our BCOV theory
$$
\delta: \B\to \B, \quad \tilde b\to \pa_z \tilde\eta. 
$$
It is easy to check that $\delta$ is compatible with the Poisson bracket, hence $\B$ becomes a dg Poisson algebra. It has a natural deformation quantization in terms of the Moyal product $\star$
$$
  \B \star \B \to \B, \quad F\star G=\sum_{k_1, k_2\geq 0}{(-1)^{k_2}\over 2^{k_1+k_2}k_1!k_2!}\bracket{\pa_{t}^{k_1}\pa_z^{k_2}F}  \bracket{\pa_t^{k_2}\pa_z^{k_1}G}.
$$
The following lemma is straight-forward.

\begin{lem}The triple $(\B, \star, \delta)$ defines an associative differential graded algebra. In particular, $(\B, \delta, [-,-]_\star)$ is a DGLA, where $[-,-]_\star$ is the commutator with respect to the Moyal product. 
\end{lem}

We also introduce the analogue grading of conformal weight by 
$$
   \cw(\pa_z^m\pa_t^k\tilde b)=m-k+1, \quad \cw(\pa_z^m\pa_t^k \tilde \eta)=m-k+1. 
$$
Note that $\tilde \eta$ has now $\cw=1$ by the shift in our Definition \ref{defn-background}. Then 
$$
\cw(\star)=0, \quad \cw(\delta)=1.
$$ 
We denote the homogeneous component
$$
  \B^d_w=\{u\in \B| \text{deg}(u)=d, \quad \cw(u)=w\}. 
$$

Our goal in this section is to construct a morphism of DGLA preserving the conformal weight
$$
  \Phi:   (\B, \delta, [-,-]_\star)\to \bracket{\oint (\V[[\h^*]]), \delta, [-,-]}
$$ 
and construct a canonical solution of Maurer-Cartan equation in $\B$. This leads to a solution of quantum master equation for our BCOV theory. 

\subsubsection{boson-fermion correspondence}\label{section-BF} We will construct $\Phi$ in terms of boson-fermion correspondence. Let us first fix our notations here and refer details to \cite{Kac,soliton}. We introduce a pair of fermions $\psi, \psi^\dagger$ with OPE
$$
  \psi(z)\psi^\dagger(w) \sim {1 \over z-w}. 
$$
We introduce a free boson with OPE
$$
    \phi(z)\phi(w)\sim \log(z-w), \quad \phi(z)=\sum_{k\neq 0}{\alpha_n z^{-n}\over -n}+\alpha_0 \log z +p. 
$$
$p$ is the momemtum creation operator as a conjugate of $\alpha_0$. The boson-fermion correspondence says that a free boson is equivalent to a pair of fermions, under the following correspondence rule
$$
    \psi=:e^{\phi}:_B, \quad \psi^\dagger=:e^{-\phi}:_B, \quad
\text{and}\quad  
    {\pa\phi=:\psi \psi^\dagger:_F}. 
$$
Here $:-:_B, :-:_F$ denote the normal ordering for bosonic fields and fermionic fields respectively. The following fundamental relation holds 
$$
{:\psi(z) \psi^\dagger(w):_F={1\over z-w} \bracket{ :e^{\phi(z)-\phi(w)}:_B -1}}.
$$
Let us expand by 
$$
{:e^{\phi(z)-\phi(w)}:_B=1+\sum_{k\geq 1} {(z-w)^k\over k!}W^{(k)}(w), \quad W^{(k)}(z)=\sum_{n\in \Z} z^{-n-k}W_n^{(k)}},
$$
then $W_n^{(k)}$ generate the so-called $W_{1+\infty}$ algebra. In terms of bosons, 
$$
  {W^{(k)}(\pa_z\phi)=\sum_{\sum_{i\geq 1}ik_i=k}{k! \over \prod_i k_i!} :\prod_i\bracket{{ {1\over i!}\pa^i \phi}}^{k_i}}:_B.
$$
Note that $W^{(k)}$ only depends on $\pa_z\phi$. We can also express it in terms of fermions
$$
{ W^{(k)}(\psi, \psi^\dagger)=k :(\pa^{k-1}\psi) \psi^\dagger:_F }.
$$
It follows from the fermionic expression that the Fourier modes $\oint dz z^m W^{(k)}$ generates a central extension of the Lie algebra of differential operators on the circle
$$
k z^m \pa_z^{k-1}\rightsquigarrow \oint dz z^m W^{(k)}, \quad k\geq 1, m\in \Z. 
$$
If we only look at non-negative modes, then we have a Lie algebra isomorphism 
\begin{align*}
  \mathbf W:  \C\bbracket{z,\pa_z} &\stackrel{\iso}{\to} \text{Span}_{\C}\fbracket{ \oint dz z^m W^{(k)}}_{k\geq 1, m\geq 0} \\
     k z^m \pa_z^{k-1}  & \to   \oint  dz  z^m W^{(k)}.  
\end{align*}

Our field $b_0$ of Heisenberg vertex algebra can be identified via free boson by 
$$
   b_0(z)=\pa_z \phi(z). 
$$
The other fields $b_{>0}, \eta_{\bullet}$ generate holomorphic (commutative) vertex algebra (i.e., all OPE's between them have no singular terms). Under the boson-fermion correspondence, 
$$
   \V[[\h^*]][[e^{\pm p}]]\iso \C[[\pa_z^m \psi, \pa_z^m \psi^\dagger, \pa_z^m b_{>0}, \pa_z^m \eta_{\bullet}]]_{m\geq 0}.
$$
If we introduce the charge grading, 
$$
   \text{charge}(\psi)=1, \quad \text{charge}(\psi^\dagger)=-1, \quad \text{charge}(\text{others})=0,
$$
then $\V[[\h^*]e^{kp}$ corresponds the homogenous component of charge $k$ on the fermionic side. We will be mainly interested in the charge $0$ component. The operator $\delta$ does not involve $b_0$. The only nontrivial part of $\delta$ is
$$
  \delta (\pa_z^mb_{k+1})=\pa_z^{m+1} \eta_k, \quad \forall k,m\geq 0. 
$$

\subsubsection{Reduction to Fedosov's equation} Now we construct the map $\Phi$. 

\begin{defn}\label{defn-Phi} We define $  \Phi: \B \to \oint (\V[[\h^*]])$ by
$$
  \Phi(J)=\sum_{k\geq 0}{1\over k+1} \oint dz W^{(k+1)}(b_0) \oint dt t^{-k-1}e^{{1\over 2}{\pa\over \pa z}{\pa\over \pa t}} J(\tilde b, \tilde \eta), \quad J\in \B. 
$$
In this expression, we need to substitute the generators $\tilde b, \tilde \eta$ in terms of the background fields as in Definition \ref{defn-background}. $\oint dt$ is the same as taking the residue at $t=0$. $W^{(k+1)}(b_0)$ is defined in the previous subsection under  $b_0=\pa_z\phi$. 
\end{defn}

 It is easy to see that $\Phi$ preserves the conformal weight. Moreover, 
\begin{prop}\label{prop-Phi}
$\Phi$ is a morphism of DGLA's
$$
  \Phi:   (\B, \delta, [-,-]_\star)\to \bracket{\oint (\V[[\h^*]]), \delta, [-,-]}.
$$
\end{prop}
\begin{proof} In terms of notations in the previous subsection, $\Phi$ can be written as
$$
  \Phi(J)=\mathbf W\bracket{\sum_{k\geq 0} \bracket{\oint dt t^{-k-1}e^{{1\over 2}\pa_z\pa_t} J} \pa_z^k}.
$$
To clarify the meaning of this formula, let 
$$
  \rho: \C[z,t]\to \C[z,\pa_z], \quad \rho(z^m t^k)=z^m\pa_z^k. 
$$
An equivalent formal description is
$$
\rho(f)=\sum_{k\geq 0} \bracket{\oint dt t^{-k-1} f(t,z)}\pa_z^k, \quad f\in \C[z,t]. 
$$
Let $\star$ denote the Moyal product on $\C[z,t]$
$$
f\star g=\left. e^{{1\over 2}(\pa_{t_1}\pa_{z_2}-\pa_{z_1}\pa_{t_2})}(f(z_1,t_1)g(z_2,t_2))\right|_{z_i=z,t_i=t}, \quad f,g\in \C[z,t]. 
$$
Consider
\begin{align*}
   e^{{1\over 2}\pa_z\pa_t}(f\star g)&=e^{{1\over 2}\pa_z\pa_t}\bracket{\left. e^{{1\over 2}(\pa_{t_1}\pa_{z_2}-\pa_{z_1}\pa_{t_2})}(f(z_1,t_1)g(z_2,t_2))\right|_{z_i=z,t_i=t}}\\
   &=\left. e^{{1\over 2}(\pa_{z_1}+\pa_{z_2})(\pa_{t_1}+\pa_{t_2})} e^{{1\over 2}(\pa_{t_1}\pa_{z_2}-\pa_{z_1}\pa_{t_2})}(f(z_1,t_1)g(z_2,t_2))\right|_{z_i=z,t_i=t}\\
   &=\left. e^{\pa_{t_1}\pa_{z_2}}(e^{{1\over 2}\pa_{z_1}\pa_{t_1}}f(z_1,t_1)e^{{1\over 2}\pa_{z_2}\pa_{t_2}}g(z_2,t_2))\right|_{z_i=z,t_i=t}.
\end{align*}
Comparing with the associative composition $\circ$ of differential operators, we find
$$
   \rho(e^{{1\over 2}\pa_z\pa_t}(f\star g))=\rho(e^{{1\over 2}\pa_z\pa_t}f) \circ \rho(e^{{1\over 2}\pa_z\pa_t}g). 
$$
In particular, 
$$
   \rho(e^{{1\over 2}\pa_z\pa_t}[f,g]_{\star})=\bbracket{\rho(e^{{1\over 2}\pa_z\pa_t}f), \rho(e^{{1\over 2}\pa_z\pa_t}g)}. 
$$
It follows from this algebraic fact and $\mathbf W$ being a Lie algebra morphism that 
$$
  \Phi([J_1, J_2]_\star)=\bbracket{\Phi(J_1), \Phi(J_2)}, \quad \forall J_1, J_2\in \B. 
$$
The compatibility of $\Phi$ with $\delta$ is easy to verify. 
\end{proof}

To construct an equivariant solution of quantum master equation, we only need to find $J\in \B^1_1$ satisfying 
$$
  \delta J+{1\over 2}[J,J]_\star=0. 
$$
This can be viewed as a version of Fedosov's abelian connection \cite{Fedosov}. 

\begin{lem} 
$$
   H^\bullet(\B, \delta)=\C[[\pa_t^k \tilde \eta]]. 
$$ 
\end{lem}
\begin{proof} This follows from the observation that the complex $(\B, \delta)$ can be identified (as a cochain complex) with the de Rham complex of $\C[[\pa_z^{\bullet} \pa_t^{\bullet}\tilde b]]$ with coefficient in $\C[[\pa_t^k \tilde \eta]]$. The Lemma follows by Poincar\'{e}'s Lemma. 
\end{proof}

\begin{cor}
$$
  H^1(\B, \delta)_{1}= \C \tilde \eta, \quad H^2(\B, \delta)_2=0. 
$$
Here $H^k(\B, \delta)_{w}$ denotes the component of $H^k(\B, \delta)$ of conformal weight $w$. 
\end{cor}
\begin{proof} $H^1(\B, \delta)_{1}= \C \tilde \eta$ is obvious. The only possible term in $\C[[\pa_t^k \tilde \eta]]$ with $\deg=2$ and $\cw=2$ is $\tilde \eta^2$, which vanishes by the odd parity of $\tilde \eta$. 
\end{proof}

To solve the above equation, let us introduce an auxiliary grading by 
$$
    T(\pa_z^m\pa_t^k \tilde b)=k, \quad T(\pa_z^m\pa_t^k \tilde \eta)=k,
$$
i.e., $T$ counts the number of $\pa_t$'s. Let us introduce the operator
$$
  \delta^*: \B\to \B, \quad \pa_z^{m+1}\pa_t^k \tilde \eta\to \pa_z^m\pa_t^k \tilde b, (m\geq 0),  \quad \pa_t^k \tilde \eta\to 0. 
$$
Let $N=\delta \delta^*+\delta^*\delta$. We define 
$$
\delta^{-1}: \alpha \to \begin{cases} {1\over m}\delta^*\alpha & \text{if}\ N \alpha=m\alpha\\ 0 & \text{if}\ N\alpha=0  \end{cases}
$$
Then $\delta^{-1}$ can be viewed as a homotopic inverse of $\delta$, where $1-[\delta,\delta^{-1}]$ is the projection to $\C[[\pa_t^k \tilde \eta]]\iso H^\bullet(\B, \delta)$. 

\begin{lem}\label{lem-J} There exists a unique $J^B\in \B^1_1$ satisfying
\begin{enumerate} 
\item $\delta J^B+{1\over 2}\bbracket{J^B, J^B}=0$;
\item $\lim_{\lambda\to 0}\lambda^T(J^B)=\tilde \eta$, i.e., the leading degree $0$ component of $J^B$ under the $T$-grading is $\tilde \eta$;
\item $\delta^{-1}J^B=0$.
\end{enumerate}
\end{lem}
\begin{proof} Let us decompose $J^{B}=\sum\limits_{k\geq 0}J_{(k)}$, where $T(J_{(k)})=kJ_{(k)}$. The initial condition (2) is $J_{(0)}=\tilde \eta$.  We show that other $J_{(k)}$'s can be uniquely solved satisfying conditions (1) and (3). Let us denote $J_{(<k)}=\sum\limits_{0\leq i<k}J_{(i)}$. Suppose we have solved $J_{(<k)}$. The $J_{(k)}$ we are looking for satisfies 
$$
   \delta J_{(k)}=\left.-{1\over 2}\bbracket{J_{(<k)}, J_{(<k)}}\right |_{(k)}. 
$$  
Here the subscript $|_{(k)}$ on the right hand side means the component containing $k$ $\pa_t$'s (i.e. $T$-eigenvalue $k$). By the standard deformation theory argument, $\left.-{1\over 2}\bbracket{J_{(<k)}, J_{(<k)}}\right |_{(k)}$ is annihilated by $\delta$ and is of conformal weight $2$. Then the existence of $J_{(k)}\in \B^1_1$ follows from $H^2(\B, \delta)_2=0$. In particular, 
$$
   J_{(k)}=\delta^{-1}\bracket{\left.-{1\over 2}\bbracket{J_{(<k)}, J_{(<k)}}\right |_{(k)}}. 
$$
solves equation (1) and (3) up to $J_{<(k+1)}$. 

Assume $J_{(k)}+U$ is another solution, then $U$ satisfies 
$$
  \deg(U)=1, \cw(U)=1, T(U)=k>0, \quad \text{and}\ \delta U=\delta^{-1}U=0. 
$$
Since $H^1(\B, \delta)_1$ is spanned by $\tilde \eta$,  while $T(\tilde \eta)=0$. It follows that $U=0$. 
\end{proof}

\subsection{Exact solution of quantum BCOV theory}

\begin{defn}\label{defn-I} Let $J^B$ be in Lemma \ref{lem-J}.  We denote $\Phi(J^{B})=\oint dz I^B$. 
\end{defn}

By Proposition \ref{prop-Phi} and Lemma \ref{lem-J}, $\oint dz I^B$ defines a Maurer-Cartan element of $\V[[\h^*]]$. To justify that $\oint dz I^B$ indeed defines a quantization of our BCOV theory, we are left to check the following two properties:
\begin{enumerate}
\item\ [Integrality]: only terms with even number of derivatives contributes to $\oint dz I^B$. 
\item\ [Classical limit]:  the term in $\oint dz I^B$ containing no derivatives coincide with our classical BCOV interaction. 
\end{enumerate}

Property (1) on integrality comes from our discussion in Section \ref{section-dilaton} on dilaton dimension.  A term with $m$ holomorphic derivatives contributes to $\hbar^{m/2}$, while we only allow integer powers of $\hbar$ to appear in our quantization. 

Property (2) is just about the classical limit. 

We will explicitly check (1) and (2) below. 

\begin{rmk} Half integer powers of $\hbar$ appear naturally when open strings are included. In \cite{open-closed}, we have also developed an open-closed BCOV theory. The terms in $J^B$ with odd number of holomorphic derivatives will be total derivatives. They vanish upon integration, but may couple nontrivially with open string sectors. It would be extremely interesting to see how open string would play into a role here. 
\end{rmk}

\subsubsection{Integrality} Let us consider the following transformation
$$
   R: \B \to \B, \quad \pa_z^k\pa_t^m \tilde b\to (-\pa_z)^k\pa_t^m \tilde b, \ \ \pa_z^k\pa_t^m \tilde \eta\to (-\pa_z)^k\pa_t^m\tilde \eta,  
$$
i.e., $R$ is the reflection $\pa_z\to -\pa_z$. It is easy to check that $R(J^B)$ also satisfies (1)(2)(3) in Lemma \ref{lem-J}. It follows from the uniqueness that 
$$
  R(J^B)=J^B. 
$$
Let us identify $b_0=\pa_z\phi$ as in Section \ref{section-BF}. Then 
\begin{align*}
  \Phi(J^B)&=\sum_{k\geq 0}{1\over k+1} \oint dz W^{(k+1)}(b_0) \oint dt t^{-k-1}e^{{1\over 2}\pa_z\pa_t} J^B \\
  &=\oint dz \sum_{k\geq 0}{W^{(k+1)}(b_0)\over (k+1)!}\oint {dt\over t} \pa_t^k e^{{1\over 2}{ \pa_z}{\pa_t}} J^B\\
  &=\oint dz  \oint {dt\over t} {:e^{\phi(z+\pa_t)-\phi(z)}:_B-1\over \pa_t}e^{{1\over 2}\pa_z\pa_t} J^B\\
   &=\oint dz  \oint {dt\over t} e^{{1\over 2}\pa_z\pa_t}\bracket{ {:e^{\phi(z+{1\over 2}\pa_t)-\phi(z-{1\over 2}\pa_t)}:_B-1\over \pa_t} J^B}\\
      &=\oint dz  \oint {dt\over t}  {:e^{\phi(z+{1\over 2}\pa_t)-\phi(z-{1\over 2}\pa_t)}:_B-1\over \pa_t} J^B.
\end{align*}
Here we have formally identified $\phi(z+\pa_t)=\sum\limits_{k\geq 0}\pa_z^k\phi {\pa_t^k\over k!}$ in the above manipulation and used the fact that the operator $e^{{1\over 2}\pa_t\pa_z}$ amounts to shifting
$
 z\to z+{1\over 2}\pa_t. 
$
In the last line, we have thrown away terms which are total derivatives in $z$. Now
$
\phi(z+{1\over 2}\pa_t)-\phi(z-{1\over 2}\pa_t)
$
contains only even number of $\pa_z$'s in terms of the field $b_0=\pa_z\phi$. From $R(J^B)=J^B$, we know that $J^B$ also contains only even number of $\pa_z$'s. Therefore $\oint dz I^B=\Phi(J^B)$ satisfies Property (1) on integrality.

\subsubsection{classical limit} We check that the classical limit  $\oint dz I_0^B$ of $\oint dz I^B$ coincides with our classical BCOV interaction. By dilaton dimension, the classical limit is related to the component $J_0^B$ of $J^B$ which does not involve any $\pa_z$ and satisfies 
$$
  \delta J_0^B+{1\over 2}\{J_0^B, J_0^B\}=0,
$$
where $\{-,-\}$ is the Poisson bracket
$
  \{A, B\}=\bracket{\pa_t A \pa_z B-\pa_z A \pa_t B}. 
$
Smilarly, $J^B$ is uniquely determined by further imposing conditions (2)(3) in Lemma \ref{lem-J}.
\begin{lem}
$$
   J_0^B=\tilde \eta+\sum_{k\geq 1}{\pa_t^{k-1}\over k!} \bracket{\tilde b^k \pa_t \tilde \eta}.
$$
\end{lem}
\begin{proof}  Let $\hat J_0^B=\tilde \eta+\sum\limits_{k\geq 1}{\pa_t^{k-1}\over k!} \bracket{\tilde b^k \pa_t \tilde \eta}$. Let us first rewrite the above formula as
\begin{align*}
   \pa_t\hat J^B_0&=\sum_{k\geq 0} {\pa_t^k\over k!} (\tilde b^k \pa_t \tilde \eta)=\oint_0 {d\lambda\over \lambda} e^{\lambda\pa_t}{1\over 1-\lambda^{-1}\tilde b}\pa_t \tilde \eta \\
   &=\oint_0 d\lambda {\pa_t \tilde \eta(z,t+\lambda)\over \lambda-\tilde b(z,t+\lambda)}=\left. {\pa_t \tilde \eta(z,t+\lambda)\over 1-\pa_t \tilde b(z,t+\lambda)}\right |_{\lambda=\tilde b(z,t+\lambda)}. 
\end{align*}
Here the substitution $\lambda=\tilde b(z,t+\lambda)$ means solving $\lambda=\sum\limits_{k\geq 0}{\lambda^k\over k!}\pa_t^k\tilde b$ for $\lambda$ as a power series in $\pa_t^k \tilde b$, then plugging $\lambda$ into $\pa_t\tilde \eta(z,t+\lambda)=\sum\limits_{k\geq 0}{
\lambda^k\over k!}\pa_t^{k+1}\tilde \eta$ and $\pa_t\tilde b(z,t+\lambda)=\sum\limits_{k\geq 0}{
\lambda^k\over k!}\pa_t^{k+1}\tilde b$. This implies via a simple chain rule computation
$$
  \hat J_0^B=\left. {\tilde \eta(z,t+\lambda)}\right |_{\lambda=\tilde b(z,t+\lambda)}.
$$

Similar computations lead to 
\begin{align*}
  \delta \hat J_0^B&=\left. {{\pa_z\tilde \eta(t+\lambda) \pa_t \tilde \eta(t+\lambda)} \over 1-\pa_t \tilde b(t+\lambda)}\right |_{\lambda=\tilde b(t+\lambda)}\\
  \pa_z \hat J^B_0&=\left. \bracket{\pa_z\tilde \eta(t+\lambda)+{\pa_z\tilde b(t+\lambda)\over 1-\pa_t \tilde b(t+\lambda)}\pa_t \tilde \eta(t+\lambda)}\right |_{\lambda=\tilde b(t+\lambda)}.
\end{align*}
It follows that
\begin{align*}
\delta \hat J_0^B+{1\over 2}\fbracket{\hat J_0^B, \hat J_0^B}=\delta \hat J_0^B+\pa_t \hat J_0^B \pa_z \hat J_0^B=\left. {{\pa_z\tilde b(t+\lambda)\over \bracket{1-\pa_t \tilde b(t+\lambda)}^2}\bracket{\pa_t \tilde \eta(t+\lambda)}^2}\right |_{\lambda=\tilde b(t+\lambda)}=0,
\end{align*}
where we have used the odd parity of $\tilde\eta$. Also, $\hat J_0^B=0$ satisfies condition (2)(3) in Lemma \ref{lem-J}. It follows from the uniqueness that $\hat J_0^B=J_0^B$. 
\end{proof}
\begin{cor}\label{cor-classical-BCOV}
$\oint dz I_0^B$ coincides with the classical BCOV interaction. In particular, $\oint dz I^B$ is a quantization of the classical BCOV theory satisfying the Hodge weight condition and dilaton equation. 
\end{cor}
\begin{proof}
The above lemma allows us to compute $\oint dzI^B_0$ explicitly 
\begin{align*}
   \oint dz I_0^B&=\sum_{k>0}{1\over k}\oint dz b_0^{k} \oint dt t^{-k} J_0^B\\
    &=\oint dz\sum_{k,l\geq0, k_i>0}{(k-1)! \over k_1!\cdots k_m!l!}{b_0^k b_{k_1}\cdots b_{k_m}\eta_l\over k!}\oint dt t^{-k} \pa_t^{m-1}\bracket{t^{k_1+\cdots +k_m+l}}\\
    &=\oint dz\sum_{\substack{k,l\geq0, k_i>0 \\ k_1+\cdots+k_m+l=k+m-2}}{(k_1+\cdots+k_m+l)!\over k_1!\cdots k_m! l!} {b_0^k b_{k_1}\cdots b_{k_m}\eta_l\over k!}\\
    &=\oint dz \abracket{e^{b}\otimes \eta}_0.
\end{align*}
\end{proof}

\subsection{Higher genus mirror symmetry} 
We briefly explain how our exact solution of quantum BCOV theory on elliptic curves is related to the Gromov-Witten theory on mirror elliptic curves. For simplicity, let us consider the following subsector of our fields in BCOV theory on elliptic curves
$$
   b_{>0}=0, \quad \eta_k=\text{constant}. 
$$
We call this the \emph{stationary sector}. This is a mirror notion (introduced in \cite{L-elliptic}) of that in \cite{virasoro} which determines the full generating function with the help of Virasoro constraint. 

From our construction, the restriction of $\oint dz I^B$ to our stationary sector, denoted by $\oint dz I^S$, is given by 
$$
  \oint dz I^S=\Phi(\tilde \eta)=\sum_{k\geq 0}\oint dz {W^{(k+2)}\over k+2} \eta_k.
$$
The quantum master equation in the stationary sector becomes 
$$
   \bbracket{\oint dz I^S, \oint dz I^S}=0. 
$$
Expanding the coefficients $\eta_k$'s, this is equivalent to 
$$
   \bbracket{\oint dz {W^{(k+2)}\over k+2}, \oint dz {W^{(m+2)}\over m+2}}=0, \quad k,m\geq 0,
$$
which represents infinite number of commuting Hamiltonians. 

\begin{rmk}
The classical limit of $\oint dz {W^{k+2}\over k+2}$ gives rise to dispersionless KdV integrable hierarchy. This observation motivates a general proposal \cite{L-review} to study integrable hierarchy from reducing topological B-model on $X\times E$ to chiral algebras on $E$ for general Calabi-Yau geometry $X$. 
\end{rmk}

Consider the following charactor, which plays the role of generating function of stationary sector of our quantum BCOV theory on elliptic curves
$$
\Tr_{\mathcal H} q^{L_0-{1\over 24}}e^{{1\over \hbar}\sum\limits_{k\geq 0}\oint dz \eta_k {W^{(k+2)}\over k+2}}, \quad q=e^{2\pi i\tau}.
$$
Here  $\mathcal H$ is the Heisenberg vertex algebra generated by $b_0$, and $L_0$ is the conformal weight operator. Then it is observed that this trace coincides with the generating function of A-model Gromov-Witten computation \cite{virasoro} under the boson-fermion correspondence. This gives a simple explanation of the explicit computations in \cite{L-elliptic} and can be viewed as a full generalization of \cite{Dijkgraaf-elliptic} who considers the cubic interaction $W^{(3)}$.


\appendix

\section{Analytic results on Feynman graph integrals}
In this appendix, we collect several analytic results for Feynman graph integrals on $\C$ that are established in \cite{Li-modular} and essentially used in this paper. 

Let $z$ be the linear holomorphic coordinate on $\C$.  Let
$$
k_t(z,\bar z)= {1\over 4\pi t}e^{-|z|^2/4t}
$$
be the standard heat kernel on $\C$. We denote the following time integration by
$$
          H_{\epsilon}^{L}(z,\bar z)=\int_\epsilon^L {dt\over 4\pi t}e^{-|z|^2/4t}, \quad 0<\epsilon<L<\infty. 
$$
The regularized propagators that appear in this paper are related to holomorphic derivatives of $H_{\epsilon}^{L}$. For example, the analytic part of the regularized propagator for the $\beta\gamma$ and $bc$ systems in \ref{sec:regularization} is essentially $\pa_z H_{\epsilon}^{L}$, and that for chiral bosons (BCOV theory) in Section \ref{sec:BCOV-elliptic} is essential $\pa_z^2 H_{\epsilon}^{L}$. 
\subsection*{Ultra-violet finiteness}

We will be mainly interested in Feynman graph integrals for translation invariant chiral interactions. Such a graph integral involves propagators as described above, and several holomorphic derivatives that appears in the vertex of the chiral interaction. To treat them uniformly, let us define the following decorated graph integral. 

Let $\Gamma$ be a directed graph. Let $V(\Gamma)$ be the set of vertices, $E(\Gamma)$ be the set of edges, and
$$
 t, h: E\to V
$$
be the assignments of tail and head to each directed edge. We will consider the decorated graph
$$
   \bracket{\Gamma, n}\equiv (\Gamma, \fbracket{n_e}_{e\in E})
$$
where the decoration is given by
$$
    n: E(\Gamma)\to \Z^{\geq 0}, \ \ e\to n_e
$$
which associates each edge a non-negative integer.  Each $n_e$ indicates how many holomorphic derivatives of $\pa_z$ are put on the propagator associated to the edge $e$. 

We will assume that $\Gamma$ is connected without self-loops (tadpoles, i.e. edges that connect a vertex to itself). Thess are the only graphs that appear in this paper. If we work with a curved model, self-loops  will be relevant. 

We consider the following graph integral on $\C$
$$
   W_{(\Gamma,n)}(H_\epsilon^L,\Phi)\equiv\prod_{v\in V(\Gamma)} \int_{\C}  d^2z_v
    \left(\prod_{e\in E(\Gamma)}\pa^{n_e}_{z_e}H_\epsilon^L(z_e,\bar z_e) \right)
   \Phi, \ \ \mbox{where}\ z_e=z_{h(e)}-z_{t(e)}
$$
here $\Phi$ is a smooth function on $\mathbb C^{|V(\Gamma)|}$ with compact support. In the above integral, we view $H_\epsilon^L(z_e,\bar z_e)$ as propagators associated to the edge $e\in E$, and we have only holomorphic derivatives on the propagators.

\begin{prop} \cite[Prop B.1]{Li-modular}\label{finiteness lem}
 The following limit  exists for the above graph integral
\begin{eqnarray*}
   \lim_{\epsilon\to 0} W_{(\Gamma,n)}(H_\epsilon^L,\Phi).
\end{eqnarray*}
\end{prop}

This proposition shows that chiral deformations of combinations of free $\beta\gamma, bc$ and chiral bosons are free of ultra-violet divergence. 

\subsection*{Quantum master equation at the UV}
Once we know that the theory is ultra-violet finite, we can deduce the meaning of the renormalized quantum master equation in the  ultra-violet limit $L\to 0$. We next discuss the relevant analytic results. 

Let us denote in this subsection 
$$
 R_\epsilon^L(z,\bar a)=\pa_zH_{\epsilon}^{L}(z,\bar z).
$$

Let $(\Gamma, n)$ be a connected decorated graph without self-loops. We index the set of vertices by
$$
    v:\{1,2,\cdots, V\}\to V(\Gamma)
$$
and index the set of edges by
$$
   e: \{0,1,2,\cdots, E-1\}\to E(\Gamma)
$$
such that $e(0),e(1),\cdots,e(k)\in E(\Gamma)$ are all the edges connecting $v(1),v(V)$. We consider the following Feynman graph integral by putting $K_\epsilon$ on $e(0)$, putting $H_\epsilon^L$ on all the other edges, and putting a smooth function $\Phi$ on $\C^{|V(\Gamma)|}$ with compact support  for the vertices. We would like to compute the following limit of the graph integral
$$
    \lim_{\epsilon\to 0}\prod_{i=1}^{V}\int d^2z_i
\pa^{n_0}_{z_{e(0)}}k_\epsilon(z_{e(0)}, \bar z_{e(0)})\left(\prod\limits_{i=1}^{E-1}\pa^{n_i}_{z_{e(i)}}R_\epsilon^L(z_{e(i)}, \bar z_{e(i)})\right)\Phi
$$
where we use the notation that
$$
     z_e\equiv z_{i}-z_{j}, \ \ \mbox{if}\ h(e)=v(i), t(e)=v(j)
$$
\begin{prop}\label{deformed-bracket}
The above limit exists and we have the identity
\begin{align*}
\begin{split}
& \lim_{\epsilon\to 0}\prod_{i=1}^{V}\int d^2z_i
\pa^{n_0}_{z_{e(0)}}k_\epsilon(z_{e(0)}, \bar z_{e(0)})\left(\prod\limits_{i=1}^{E-1}\pa^{n_i}_{z_{e(i)}}R_\epsilon^L(z_{e(i)}, \bar z_{e(i)})\right)\Phi\\
=&\lim_{\epsilon\to 0}{C(n_0;n_1,\cdots,n_k)\over (4\pi)^k}\prod_{i=2}^{V}\int d^2z_i \left.\pa_{z_1}^{n_0+\sum\limits_{i=1}^k(n_i+1)}\left( \left(\prod\limits_{i=k+1}^{E-1}\pa^{n_i}R_\epsilon^L(z_{e(i)}, \bar z_{e(i)})\right)\Phi \right)\right|_{z_1=z_V}
\end{split}
\end{align*}
where the constant $C(n_0,n_1,\cdots,n_k)$ is a rational number given by
$$
     C(n_0;n_1,\cdots,n_k)=\int_0^1\cdots \int_0^1 \prod\limits_{i=1}^{k}{du_i} {\prod\limits_{i=1}^k u_i^{n_i}\over \left(1+\sum\limits_{i=1}^k u_i\right)^{\sum\limits_{j=0}^k (n_j+1)}}
$$
\end{prop}
\begin{proof} \cite[Prop B.2]{Li-modular} proves the same statement when $k_\epsilon$ is replaced by $\pa_z k_\epsilon$ (which was called $U_\epsilon$ there), and $R_\epsilon^L$ is replaced by $\pa_z R_\epsilon^L$ (which was called $P_\epsilon^L$ there). Such cases correspond to those graph integrals when each edge has decoration $n_e\geq 1$. However, literally the same proof as in \cite[Prop B.2]{Li-modular} shows that the argument is valid for any decorations. This proves the proposition. 

\end{proof}



\begin{thebibliography}{99}

\bib{AKSZ}{article}{
   author={Alexandrov, M.},
   author={Schwarz, A.},
   author={Zaboronsky, O.},
   author={Kontsevich, M.},
   title={The geometry of the master equation and topological quantum field
   theory},
   journal={Internat. J. Modern Phys. A},
   volume={12},
   date={1997},
   number={7},
   pages={1405--1429},
}

\bib{AB}{article}{
  title={A Lefschetz fixed point formula for elliptic complexes: I},
  author={Atiyah, M.},
  author={Bott, R.},
  journal={Annals of Mathematics},
  pages={374--407},
  year={1967},
  publisher={JSTOR}
}

\bib{CT}{article}{
   author={C\u ald\u araru, A.},
   author={Tu, J.},
title={Computing a categorical Gromov-Witten invariant}, 
journal={arXiv:1706.09912 [math.AG]},
}

\bib{BGV}{book}{
  title={Heat kernels and Dirac operators},
  author={Berline, N.},
  author={Getzler, E.},
  author={Vergne, M.},
  year={2003},
  publisher={Springer Science \& Business Media}
}

\bib{graph}{article}{
   author={Bessis, D.},
   author={Itzykson, C.},
   author={Zuber, J. B.},
   title={Quantum field theory techniques in graphical enumeration},
   journal={Adv. in Appl. Math.},
   volume={1},
   date={1980},
   number={2},
   pages={109--157},
}
   \bib{BCOV}{article}{
   author={Bershadsky, M.},
   author={Cecotti, S.},
   author={Ooguri, H.},
   author={Vafa, C.},
   title={Kodaira-Spencer theory of gravity and exact results for quantum
   string amplitudes},
   journal={Comm. Math. Phys.},
   volume={165},
   date={1994},
   number={2},
   pages={311--427},
   }
   
   \bib{BV}{article}{
   author={Batalin, I. A.},
   author={Vilkovisky, G. A.},
   title={Gauge algebra and quantization},
   journal={Phys. Lett. B},
   volume={102},
   date={1981},
   number={1},
   pages={27--31},
}

\bib{CF}{article}{
   author={Cattaneo, A. S.},
   author={Felder, G.},
   title={A path integral approach to the Kontsevich quantization formula},
   journal={Comm. Math. Phys.},
   volume={212},
   date={2000},
   number={3},
   pages={591--611},
}

\bib{Poisson-sigma-model}{article}{
   author={Cattaneo, A. S.},
   author={Felder, G.},
   title={On the AKSZ formulation of the Poisson sigma model},
   note={EuroConf\'erence Mosh\'e Flato 2000, Part II (Dijon)},
   journal={Lett. Math. Phys.},
   volume={56},
   date={2001},
   number={2},
   pages={163--179},
}

\bib{Partition}{article}{
   author={Costello, K.},
   title={The partition function of a topological field theory},
   journal={J. Topol.},
   volume={2},
   date={2009},
   number={4},
   pages={779--822},
}

\bib{TCFT}{article}{
   author={Costello, K.},
   title={Topological conformal field theories and Calabi-Yau categories},
   journal={Adv. Math.},
   volume={210},
   date={2007},
   number={1},
   pages={165--214},
   issn={0001-8708},
}

\bib{Kevin-book}{book}{
   author={Costello, K.},
   title={Renormalization and effective field theory},
   series={Mathematical Surveys and Monographs},
   volume={170},
   publisher={American Mathematical Society, Providence, RI},
   date={2011},
   pages={viii+251},
}

\bib{Kevin-genus}{article}{
   author={Costello, K.},
   title={A geometric construction of the Witten genus, I},
   conference={
      title={Proceedings of the International Congress of Mathematicians.
      Volume II},
   },
   book={
      publisher={Hindustan Book Agency, New Delhi},
   },
   date={2010},
   pages={942--959},
}

\bib{Kevin-Owen}{book}{
author={Costello, K.}, 
author={Gwilliam, O.},
title={Factorization algebras in quantum field theory},
eprint={http://www.math.
northwestern.edu/~costello/renormalization}
}

\bib{Si-Kevin}{article}{
      author={Costello, K.},
      author={Li, S.},
       title={Quantum {BCOV} theory on {Calabi-Yau} manifolds and the higher
  genus {B}-model},
     journal={arXiv:1201.4501 [math.QA]},
}

\bib{open-closed}{article}{
      author={Costello, K.},
      author={Li, S.},
       title={Quantization of open-closed {BCOV} theory, {I}},
     journal={arXiv:1505.06703 [hep-th]},
}

\bib{Dijkgraaf-elliptic}{article}{
   author={Dijkgraaf, R.},
   title={Mirror symmetry and elliptic curves},
   conference={
      title={The moduli space of curves},
      address={Texel Island},
      date={1994},
   },
   book={
      series={Progr. Math.},
      volume={129},
      publisher={Birkh\"auser Boston, Boston, MA},
   },
   date={1995},
   pages={149--163},
}

\bib{Dijkgraaf-chiral}{article}{
   author={Dijkgraaf, R.},
   title={Chiral deformations of conformal field theories},
   journal={Nuclear Phys. B},
   volume={493},
   date={1997},
   number={3},
   pages={588--612},
}

\bib{Dou95}{article}{
  author={Douglas, M.},
   title={Conformal Field Theory Techniques in Large N Yang-Mills Theory},
   journal={In: Baulieu L., Dotsenko V., Kazakov V., Windey P. (eds) Quantum Field Theory and String Theory. NATO ASI Series (Series B: Physics)},
   volume={328},
   date={1995},
}

 \bib{Fedosov}{article}{
    AUTHOR = {Fedosov, B. V.},
     TITLE = {A simple geometrical construction of deformation quantization},
   JOURNAL = {J. Differential Geom.},
    VOLUME = {40},
      YEAR = {1994},
    NUMBER = {2},
     PAGES = {213--238},
}


\bib{Fedosov-index}{book}{
   author={Fedosov, B. V.},
   title={Deformation quantization and index theory},
   series={Mathematical Topics},
   volume={9},
   publisher={Akademie Verlag, Berlin},
   date={1996},
   pages={325},
}


\bib{CFT}{book}{
  title={Conformal field theory},
  author={Francesco, P.},
  author={Mathieu, P.},
  author={S{\'e}n{\'e}chal, D},
  year={2012},
  publisher={Springer Science \& Business Media}
}

\bib{Frenkel}{book}{
   author={Frenkel, E.},
   author={Ben-Zvi, D.},
   title={Vertex algebras and algebraic curves},
   series={Mathematical Surveys and Monographs},
   volume={88},
   edition={2},
   publisher={American Mathematical Society, Providence, RI},
   date={2004},
   pages={xiv+400},
   isbn={0-8218-3674-9},
}

\bib{CDO}{article}{
   author={Gorbounov, V.},
   author={Malikov, F.},
   author={Schechtman, V.},
   title={Gerbes of chiral differential operators},
   journal={Math. Res. Lett.},
   volume={7},
   date={2000},
   number={1},
   pages={55--66},
}
\bib{GLX}{article}{
  title={Geometry of Localized Effective Theory, Exact Semi-classical Approximation and Algebraic Index},
  author={Gui, Z.}, 
  author={Li, Si},
  author={Xu, Kai},
  journal={arXiv preprint arXiv:1911.11173},
  year={2019}
}

\bib{GG}{article}{
   author={Gwilliam, O.},
   author={Grady, R.},
   title={One-dimensional Chern-Simons theory and the $\hat A$ genus},
   journal={Algebr. Geom. Topol.},
   volume={14},
   date={2014},
   number={4},
   pages={2299--2377},
}

\bib{HLTY}{article}{,
  title={Dispersionless Integrable Hierarchy via Kodaira-Spencer Gravity},
  author={He, W.},
  author={Li, S.},
  author={Tang, X.},
  author={Yoo, P.},
  journal={arXiv preprint arXiv:1910.05665},
  year={2019}
}


\bib{Kac}{book}{
   author={Kac, V.},
   title={Vertex algebras for beginners},
   series={University Lecture Series},
   volume={10},
   edition={2},
   publisher={American Mathematical Society, Providence, RI},
   date={1998},
   pages={vi+201},
   isbn={0-8218-1396-X},
}


  \bib{almost-modular-form}{article}{
   author={Kaneko, M.},
   author={Zagier, D.},
   title={A generalized Jacobi theta function and quasimodular forms},
   conference={
      title={The moduli space of curves},
      address={Texel Island},
      date={1994},
   },
   book={
      series={Progr. Math.},
      volume={129},
      publisher={Birkh\"auser Boston},
      place={Boston, MA},
   },
   date={1995},
   pages={165--172},
    }

\bib{Kontsevich-DQ}{article}{
   author={Kontsevich, M.},
   title={Deformation quantization of Poisson manifolds},
   journal={Lett. Math. Phys.},
   volume={66},
   date={2003},
   number={3},
   pages={157--216},
}

\bib{KS}{article}{
   author={Kontsevich, M.},
   author={Soibelman, Y.},
   title={Notes on $A_\infty$-algebras, $A_\infty$-categories and
   non-commutative geometry},
   conference={
      title={Homological mirror symmetry},
   },
   book={
      series={Lecture Notes in Phys.},
      volume={757},
      publisher={Springer, Berlin},
   },
   date={2009},
   pages={153--219},
   review={\MR{2596638}},
}
	

\bib{B-model}{article}{
   author={Li, Q.},
   author={Li, S.},
   title={On the B-twisted topological sigma model and Calabi-Yau geometry},
   journal={J. Differential Geom.},
   volume={102},
   date={2016},
   number={3},
   pages={409--484},
}


\bib{LLG}{article}
{
author={Grady, R.},
author={Li, Q.},
author={Li, S.},
title={Batalin--Vilkovisky quantization and the algebraic index},
   journal={Adv. Math.},
   volume={317},
   date={2017},
   pages={575--639},
}

\bib{thesis}{book}{
   author={Li, S.},
   title={Calabi-Yau Geometry and Higher Genus Mirror Symmetry},
   note={Thesis (Ph.D.)--Harvard University},
   publisher={ProQuest LLC, Ann Arbor, MI},
   date={2011},
   pages={174},
}

\bib{Li-modular}{article}{
   author={Li, S.},
   title={Feynman graph integrals and almost modular forms},
   journal={Commun. Number Theory Phys.},
   volume={6},
   date={2012},
   number={1},
   pages={129--157},
}

\bib{L-elliptic}{article}
{
author={Li, S.},
title={BCOV theory on the elliptic curve and higher genus mirror symmetry}, 
journal={arXiv:1112.4063 [math.QA]},
}

\bib{L-review}{article}{
   author={Li, S.},
   title={Effective {B}atalin-{V}ilkovisky quantization and geometric applications}, 
 journal={arXiv:1709.00669 [math.QA]},
}

\bib{Si-String-Math}{article}
{
author={Li, S.},
title={Simi-infinite Hodge structures: from BCOV theory to Seiberg-Witten geometry},
journal={talk slide at String-Math 2018.  Available at https://sili-math.github.io/2018-string-math.pdf}
}

\bib{Si-Jie}{article}
{author={Li, S.},
author={Zhou, J.}, 
title={{R}egularized {I}ntegrals on {R}iemann {S}urfaces and {M}odular {F}orms},
journal={arXiv:2008.07503 [math.DG]}
}

\bib{Losev}{article}{
   author={Losev, A.},
   author={Shadrin, S.},
   author={Shneiberg, I.},
   title={Tautological relations in Hodge field theory},
   journal={Nuclear Phys. B},
   volume={786},
   date={2007},
   number={3},
   pages={267--296},
}

\bib{CDR}{article}{
   author={Malikov, F.},
   author={Schechtman, V.},
   author={Vaintrob, A.},
   title={Chiral de Rham complex},
   journal={Comm. Math. Phys.},
   volume={204},
   date={1999},
   number={2},
   pages={439--473},
}

\bib{soliton}{book}{
   author={Miwa, T.},
   author={Jimbo, M.},
   author={Date, E.},
   title={Solitons},
   series={Cambridge Tracts in Mathematics},
   volume={135},
   note={Differential equations, symmetries and infinite-dimensional
   algebras;
   Translated from the 1993 Japanese original by Miles Reid},
   publisher={Cambridge University Press, Cambridge},
   date={2000},
   pages={x+108},
}


\bib{Nest-Tsygan}{article}{
   author={Nest, R.},
   author={Tsygan, B.},
   title={Algebraic index theorem},
   journal={Comm. Math. Phys.},
   volume={172},
   date={1995},
   number={2},
   pages={223--262},
}

\bib{treves}{book}{
    AUTHOR = {Tr{\`e}ves, F.},
     TITLE = {Topological vector spaces, distributions and kernels},
      NOTE = {Unabridged republication of the 1967 original},
 PUBLISHER = {Dover Publications, Inc., Mineola, NY},
      YEAR = {2006},
     PAGES = {xvi+565},
}

\bib{virasoro}{article}{
   author={Okounkov, A.},
   author={Pandharipande, R.},
   title={Virasoro constraints for target curves},
   journal={Invent. Math.},
   volume={163},
   date={2006},
   number={1},
   pages={47--108},
}

\bib{Zwiebach}{article}{
   author={Zwiebach, B.},
   title={Closed string field theory: quantum action and the
   Batalin-Vilkovisky master equation},
   journal={Nuclear Phys. B},
   volume={390},
   date={1993},
   number={1},
   pages={33--152},
   issn={0550-3213},
}

\end{thebibliography}
\end{document}